\documentclass[3p, final]{elsarticle}

\usepackage{bm}
\usepackage{amsmath,amsfonts,amssymb,amscd,amsthm,amsbsy}
\usepackage{subfigure}
\usepackage[colorlinks=true, pdfstartview=FitV, linkcolor=black, citecolor= black, urlcolor= black]{hyperref}
\usepackage{footnpag}			      	
\usepackage{booktabs}
\usepackage{graphicx}
\usepackage{xcolor}

\theoremstyle{lemma}
\newtheorem{lemma}{Lemma} 

\theoremstyle{remark}
\newtheorem{remark}{Remark} 

\theoremstyle{claim}
\newtheorem*{claim*}{Claim} 

\begin{document}

\begin{frontmatter}

\title{A new fast multiple-shooting method for computing periodic orbits in symplectic maps leveraging simultaneous Floquet vector computation to avoid large linear systems}

\author[umaddress]{Bhanu Kumar\corref{mycorrespondingauthor}}
\ead{bhkumar@umich.edu}
\cortext[mycorrespondingauthor]{Corresponding author}

\address[umaddress]{Department of Mathematics, University of Michigan, 530 Church St. Ann Arbor, MI, 48109, USA}

\begin{abstract}
Given a 4D symplectic map $F_0$ that has a normally hyperbolic invariant cylinder foliated by invariant tori, those with rational rotation numbers are themselves foliated by subharmonic periodic orbits (SPOs). If $F_0$ is part of a perturbative family $F_\varepsilon$, one is often interested in computing those SPOs which persist for $\varepsilon >0$. Assuming that a persisting SPO of $F_0$ has been identified, in this paper, we develop a quasi-Newton method which solves for the SPO simultaneously with its Floquet vectors and multipliers. This in turn enables continuation by the perturbation parameter $\varepsilon$. The resulting SPO and Floquet vectors are then used to compute Taylor parameterizations of the SPO's weak stable and unstable manifolds, if they exist. Our quasi-Newton method is based on an adaptation of the parameterization method for invariant tori, with this paper being the first-ever to apply such a framework to directly compute periodic orbit points themselves. The new algorithm improves on efficiency compared to prior multi-shooting methods for SPOs, and notably applies to the case of stroboscopic maps of 2.5 DOF Hamiltonian flows resulting from periodic perturbations of 2 DOF systems. The tools have been successfully used for studies of resonant orbits in perturbed real-life celestial systems, the results of which are summarized as a demonstration of the methods' utility.
\end{abstract}


\end{frontmatter}
		
\section{Introduction}

In 2 DOF Hamiltonian systems, such as the well-known planar circular restricted 3-body problem (PCR3BP) of celestial mechanics, unstable periodic orbits and their attached stable/unstable manifolds are key drivers of global dynamics. Such periodic orbits occur in 1-parameter families diffeomorphic to 2D cylinders, with orbit period varying along the family. When such a 2 DOF system is periodically perturbed, yielding a 2.5 DOF non-autonomous Hamiltonian system, unstable periodic orbits whose periods $T$ are non-resonant with the perturbation period $T_p$ generically persist as invariant tori for sufficiently small perturbation strength \cite{kumar2022}. On the other hand, the persistence of \emph{subharmonic} orbits with resonant periods, i.e., $T/T_p$ rational, is more delicate---with a strong dependence on the initial phases chosen on the periodic orbit and periodic perturbation. While most phases correspond to subharmonic orbits that are destroyed by the perturbation, certain phases yield orbits that persist as \emph{periodic orbits} in the perturbed 2.5 DOF system. Moreover, if $T_p/T = p/q$ with $p,q \in \mathbb{Z}^{+}$ coprime, these periodic orbits will have period $pT = qT_p$ in the 2.5 DOF system---potentially much longer than their period $T$ in the unperturbed 2 DOF system. 

While the above discussion is in the context of continuous-time flows, it also lends an equivalent characterization using symplectic maps. Namely, in the 2.5 DOF system, since the perturbation is periodic with period $T_p$, one can instead consider the dynamics of the time-$T_p$ flow map of the equations of motion. This gives rise to a (symplectic) \emph{stroboscopic map} on a 4D phase space. One can in fact also define this map in the 2-DOF unperturbed case, so that the stroboscopic map of the 2.5 DOF system becomes a perturbation of the 2 DOF system's map. In this context, the 2 DOF flow's unstable periodic orbit family becomes a 2D cylindrical normally hyperbolic invariant manifold (NHIM) of the \emph{unperturbed} map, entirely foliated by partially-hyperbolic invariant (1D) tori---each of which corresponds to a different flow-periodic orbit. Orbits with periods non-resonant with that of the perturbation become tori with irrational rotation numbers, while those with resonant periods yield invariant tori (of the unperturbed map) with rational rotation numbers. As indicated by KAM theory \cite{kamTutorial}, most tori with irrational rotation numbers will persist for perturbations sufficiently small, whereas those with rational rotation numbers $p/q$ will not. These latter tori are themselves foliated by $q$-iteration periodic orbits of the unperturbed 2-DOF system's stroboscopic map; of these last map-periodic orbits, generally only a discrete subset will persist under the perturbation. 

In both the flow as well as the map, given the relative scarcity of rational numbers as compared to irrationals, most unstable orbits from the unperturbed system persist as tori for sufficiently small perturbations. However, this does not necessarily hold in the case of larger perturbation strength. Indeed, \cite{kumar2023aas} showed that in a 2.5 DOF restricted 4-body problem model of spaceflight in the  Jupiter-Europa-Ganymede system, the aforementioned long, stroboscopic map-periodic orbits generate \emph{secondary resonance} regions which can overlap in the sense of Chirikov \cite{chirikov1960}, destroying all tori (corresponding to PCR3BP unstable flow-periodic orbits) in their midst. This overlap is guided by intersections of \emph{separatrices}---weak stable/unstable submanifolds, contained \emph{within} a persisting NHIM, of the long periodic orbits. Thus, to understand the destruction of whiskered tori in such perturbed 2.5 DOF systems (or equivalently, their stroboscopic maps) requires computing these long periodic orbits as well as their attached separatrices. 

Due to their long periods, particularly if $q$ is large, computing any such persisting long, unstable periodic orbits generally requires a multiple-shooting algorithm to numerically continue them from the unperturbed system. For example, one might seek to compute $q$ points on the periodic orbit such that propagating the $i$\textsuperscript{th} point by the stroboscopic map (i.e., the 2.5 DOF system flow for time $T_p$) yields the $(i+1)$\textsuperscript{th} point, modulo $q$. In this case, to solve for $q$ points in 4D space using Newton's method requires solving a $4q \times 4q$ system of linear equations. Additionally, even once Newton's method has converged to the desired long periodic orbit, determining its full stability properties---including linear approximations of potential separatrices, if they exist---traditionally involves a separate step that requires finding eigenvectors of a $4q \times 4q$ matrix. Both the Newton step and the stability analysis here have $O(q^{3})$ complexity. If $q$ is on the order of 100 or more, as is often required for analysis of breakdown of invariant tori, this can become a relatively expensive procedure. 

In this paper, we develop an efficient multiple-shooting algorithm for simultaneous computation of both these long periodic orbits as well as their corresponding Floquet stability multipliers and directions. Solving for Floquet vectors alongside the long periodic orbits in fact allows us to avoid solving large linear systems, making the algorithm more efficient than previous methods that solve for the orbit alone. To accomplish this, we adapt the parameterization method of \cite{kumar2022}, developed for computing unstable invariant tori with their center, stable, and unstable directions, to the case of periodic orbits as well. While the parameterization method \cite{haroetal} has been used in many past studies (e.g., \cite{Llave_2005, haroLlave, huguet2012, haro2021flow}) to compute invariant tori of varying properties, the only prior work applying related tools to periodic orbits is that of Calleja et al \cite{calleja2021POs}. However, this last paper only studied long periodic orbits in 2D symplectic maps, and moreover, only used a parameterization method-style algorithm to find a \emph{1D curve} containing the long periodic orbit, rather than the actual orbit points; the points themselves were only found after a phase search on this curve and a traditional large matrix-based multiple shooting scheme. Thus, we believe that the method of this paper is the first to apply the parameterization method to directly compute such long periodic orbits as well. 

The method developed here applies to any perturbative family of 4D symplectic maps satisfying certain conditions, not just stroboscopic maps. We therefore will present the method in this more general context, for families of 4D symplectic maps such that the unperturbed map has a 2D cylindrical NHIM
whose internal dynamics are an integrable twist map, foliated entirely by partially-hyperbolic invariant tori, near some orbit of interest.
We start by precisely defining the problem setting and statement in Section \ref{problemSection}, followed by some motivating background and physical models from celestial mechanics in Section \ref{modelSection}. Section \ref{spoSection} then presents the new parameterization method-style multiple-shooting algorithm for computation of the long periodic orbits and their Floquet directions and multipliers. Next, a recursive parameterization method  \citep{CabreFontichLlave, haroetal} leveraging the aforementioned Floquet directions to compute separatrices emanating from certain long periodic orbits is described in Section \ref{parambigsection}. Finally, some results from applied celestial mechanics studies in which this method was successfully implemented and used are summarized in Section \ref{numResults} as a demonstration of its utility. 

We have included several proofs in this paper to justify our methods and motivate possible adaptations, many of which are presented in the appendices (with references to them in the main body). These proofs may be skipped without detriment by readers primarily interested in details of the algorithm implementation. Similarly, those interested in more rigorous results may skip the implementation details in Section \ref{numResults}. 

\section{Problem Setting and Statement} \label{problemSection}

\subsection{Problem Setting} \label{settingSection}

In this paper, we will consider families of symplectic maps on $\mathbb{R}^4$ that depend on a perturbation parameter $\varepsilon \geq 0$. Such families will be denoted as $F_{\varepsilon}:\mathbb{R}^{4} \rightarrow \mathbb{R}^{4}$. We assume that the unperturbed map $F_{\varepsilon = 0}$ has a 2D cylindrical normally hyperbolic invariant manifold (NHIM) $\Xi_{0} \subset \mathbb{R}^{4}$ foliated entirely by a family of partially-hyperbolic (whiskered) invariant 1D tori, so that the hyperbolic directions of each such torus will be transverse to the NHIM. We also assume that the family of maps $F_{\varepsilon}$ is differentiable at $\varepsilon = 0$ with respect to the parameter $\varepsilon$ in a neighborhood of $\Xi_{0}$. While it may be possible to relax the assumptions on differentiability and phase space being $\mathbb{R}^{4}$, we only consider the previously-described case in this paper. 

Within the NHIM $\Xi_{0}$, each of its constituent invariant tori will have a \emph{rotation number} $\omega$; mathematically, this means that each $F_0$-invariant torus can be parameterized \cite{kumar2022} using a function $K: \mathbb{T} \rightarrow \mathbb{R}^{4}$ such that 
\begin{equation} \label{torusEquation} F_{0}(K(\theta)) = K(\theta + \omega) \end{equation}
Now, assume that in at least some neighborhood of some invariant torus in $\Xi_{0}$, the tori therein satisfy a twist condition---that is, the rotation number $\omega$ is not a constant for the tori in this neighborhood, but monotonically varies. This then implies that there must be (infinitely many) tori in this neighborhood with a \emph{resonant} rotation number, i.e., with $\frac{\omega}{2\pi}$ rational\footnote[1]{We use the convention of the angle $\theta$ being represented by the interval $[0,2\pi]$ with 0 and $2\pi$ identified.}. Hereafter, such tori are called \emph{resonant} tori. 

Now, consider a single resonant invariant torus of the unperturbed map $F_0$, the torus' parameterization $K_0:\mathbb{T}\rightarrow \mathbb{R}^{4}$, and its rotation number $\omega = 2\pi p/q$ for some $p, q \in \mathbb{Z}^{+}$ coprime. Eq. \eqref{torusEquation} then implies that
\begin{equation} \label{periodicTorusEquation} F_{0}^{q}(K_{0}(\theta)) = K_{0}(\theta + q\omega) = K_{0}(\theta + 2\pi p) = K_{0}(\theta) \end{equation}
In other words, any point $K_{0}(\theta)$ on this resonant torus is in fact part of a $q$-iteration periodic orbit of the map $F_0$, consisting of the $q$ distinct points $K_{0}(\theta), K_0(\theta+\omega), \dots, K_0(\theta+(q-1)\omega)$. Such periodic orbits will henceforth be referred to as \emph{subharmonic} periodic orbits. Any resonant invariant torus of $F_0$ is thus entirely foliated by a continuum of subharmonic periodic orbits, one for each value of $\theta \in [0, \omega)$. 

\subsection{The Problem: Computing Perturbed Subharmonic Periodic Orbits} \label{problemStateSection}

While the above NHIM, invariant tori, and subharmonic periodic orbits were all defined in the context of the unperturbed map $F_0$, now consider the  case of $\varepsilon > 0$. For $\varepsilon$ sufficiently small, results by Fenichel \cite{fenichel1971persistence} and others \cite{hirschPughShub} show that the $F_0$-invariant NHIM $\Xi_{0}$ should persist into the perturbed maps $F_{\varepsilon}$ as perturbed 2D cylindrical NHIMs, which we will denote as $\Xi_{\varepsilon}$. Then, KAM theory \cite{kamTutorial} tells us that tori  which satisfy the twist condition and have sufficiently irrational (Diophantine) $\frac{\omega}{2\pi}$ will persist inside $\Xi_{\varepsilon}$, for $\varepsilon > 0$ sufficiently small. On the other hand, resonant tori generically \emph{do not} persist as invariant tori of $F_\varepsilon$ for any $\varepsilon >0$; however, a finite number of subharmonic periodic orbits from such tori often \emph{do} persist into the perturbed system. Indeed, there exists a robust subharmonic Melnikov theory \cite{guckholmes, treschev1998} which can be used to determine which of the infinitely-many subharmonic periodic orbits $\{K_{0}(\theta), K_0(\theta+\omega), \dots, K_0(\theta+(q-1)\omega)\}$ that foliate a resonant torus $K_{0}(\theta)$ with $\omega = 2\pi p/q$ persist as periodic orbits of $F_\varepsilon$ for $\varepsilon > 0$. Symmetries of the maps $F_\varepsilon$ can also often be used to determine subharmonic periodic orbits of $F_0$ that might persist. 

We can now present the key problem which this work addresses: given a subharmonic periodic orbit of $F_0$ that persists into the perturbed map $F_\varepsilon$, at least for $\varepsilon >0$ sufficiently small, we would like to compute the corresponding periodic orbit of $F_\varepsilon$. Namely, if we denote the persisting $F_0$ subharmonic periodic orbit's points as $X_0(k) = K_0(\theta+k\omega)$ for $k = 0, 1, \dots, q-1$, so that $F_0(X_0(k)) = X_0(k+1 \mod q)$, we would like to solve for the points $X_\varepsilon(k) \in \mathbb{R}^{4}$, $k = 0, 1, \dots, q-1$, that satisfy the equation
\begin{equation} \label{invarianceEquation} F_\varepsilon(X_\varepsilon(k)) = X_\varepsilon(k+1 \mod q) \end{equation}
for all $k = 0, 1, \dots, q-1$. In addition, periodic orbits also have Floquet directions and multipliers; while these can theoretically be found by calculating the eigenvectors and eigenvalues of $DF^{q}_\varepsilon(X_\varepsilon(k))$, in practice this is inaccurate when $q$ is large due to the orbit's instability (recall that it is contained in a NHIM). Thus, to find the Floquet directions and multipliers, we will instead seek to solve the equation
\begin{equation} \label{floquetEquation} DF_\varepsilon(X_\varepsilon(k)) \bar P_{\varepsilon}(k) = \bar P_\varepsilon(k+1 \mod q) \bar \Lambda_{\varepsilon} \end{equation}
for matrices $\bar P_{\varepsilon}(k) \in \mathbb{C}^{4 \times 4}$, $k = 0, 1, \dots, q-1$ containing the Floquet directions at each point $X_\varepsilon(k)$, and $\bar \Lambda_{\varepsilon} \in \mathbb{C}^{4 \times 4}$ diagonal containing the Floquet multipliers. Note that the columns of $\bar P_{\varepsilon}(k)$ thus found will indeed be eigenvectors of $DF^{q}_\varepsilon(X_\varepsilon(k))$, as applying Eq. \eqref{floquetEquation} $q$ times yields that $DF^{q}_\varepsilon(X_\varepsilon(k)) \bar P_{\varepsilon}(k) = \bar P_\varepsilon(k) \bar \Lambda_{\varepsilon}^{q}$; keeping in mind that $\bar \Lambda_{\varepsilon}$ was diagonal, this implies that the columns of $\bar P_\varepsilon$ are eigenvectors of  $DF^{q}_\varepsilon(X_\varepsilon(k))$. 

\begin{remark} Note that $X_\varepsilon$ and $\bar P_\varepsilon$ can be considered to be functions from the set $\{0, 1, \dots, q-1\}$ into $\mathbb{R}^{4}$ and $\mathbb{C}^{4 \times 4}$, respectively; indeed, the choice of notation $X_\varepsilon(k)$ and $\bar P_\varepsilon(k)$, as opposed to letting $k$ be a subscript, is meant to highlight this characterization of $X_\varepsilon$ and $\bar P_\varepsilon$. It is hoped that this will facilitate easier comparisons and analogies with the invariant torus parameterization case (compare Eqs. \eqref{torusEquation} and \eqref{invarianceEquation}, for instance). \end{remark}

\section{Motivation: Hamiltonian Flow Maps and Physical Models} \label{modelSection}

2 DOF autonomous Hamiltonian systems and their 2.5 DOF periodic, non-autonomous perturbations are some of the most common situations which can give rise to the setting of Section \ref{settingSection}. We thus now give some background on these systems, their recharacterization as 4D symplectic maps, and some concrete real-life examples from celestial mechanics that will be used to illustrate the methods of this paper later on. Readers familiar with stroboscopic maps of 2.5 DOF periodic perturbations of 2 DOF Hamiltonian flows, and the effect of such perturbations on the latter's periodic orbits, may skip to Section \ref{spoSection} without detriment. 

\subsection{2 DOF Hamiltonian Systems and NHIMs of Flow-Periodic Orbits} \label{2dofSection}

If one represents $\mathbb{R}^{4}$ using position-momentum coordinates $(x,y, p_{x}, p_{y})$, a 2 DOF Hamiltonian dynamical system on this space is given by a time-independent Hamiltonian function $H_0: \mathbb{R}^{4}\rightarrow \mathbb{R}$ and the equations
\begin{equation} \label{2dofH_EOM} \dot x = \frac{\partial H_{0}}{\partial p_{x}} \quad \dot y = \frac{\partial H_{0}}{\partial p_{y}} \quad \quad \dot p_{x} = -\frac{\partial H_{0}}{\partial x} \quad \dot p_{y} = -\frac{\partial H_{0}}{\partial y} \end{equation}
As written above, this is a continuous time flow rather than a map like in Section \ref{settingSection}. However, any fixed-time flow map of the equations of motion of Eq. \ref{2dofH_EOM}---that is, the map which propagates any $\bold{x} \in \mathbb{R}^{4}$ by Eq. \eqref{2dofH_EOM} for that fixed time---yields a symplectic map on $\mathbb{R}^{4}$. Moreover, as the Hamiltonian is time-independent, being a function only of $(x, y, p_x, p_y)$, the \emph{energy} $H_0$ is a constant of motion along system trajectories. 

In such 2 DOF Hamiltonian systems, among the most important trajectories for characterizing the system's global dynamics are unstable periodic orbits (of the continous-time flow). These unstable orbits are of great interest due to their attached stable/unstable manifolds, whose intersections can yield trajectories that traverse large regions of phase space and destroy barriers to transport. Such unstable periodic orbits generally occur in continuous 1-parameter families; locally, this parameter can oftentimes be taken as the constant Hamiltonian energy value $H_0$ along the periodic orbit. Generally, the orbit period also varies continuously along the periodic orbit family. Unstable periodic orbits in 2 DOF Hamiltonian systems have real monodromy (Floquet) matrix eigenvalues 1, 1, $\lambda$, and $\lambda^{-1}$, where $| \lambda | < 1$. 

Note that topologically, any periodic orbit of a continuous-time flow is diffeomorphic to the torus $\mathbb{T}$. Now, given a family of unstable periodic orbits, consider the union of all periodic orbits across that family. The resulting set will be a 2D cylindrical manifold $\Xi_0$ in the 2 DOF Hamiltonian's phase space $\mathbb{R}^{4}$. Furthermore, at each point of $\Xi_0$, there are stable and unstable directions transverse to the manifold, which are just the stable and unstable Floquet eigenvectors of the periodic orbits which foliate $\Xi_0$. On the other hand, the generalized unit eigenspace of each periodic orbit forms the tangent space to the NHIM at that point. Since the phase space is 4D and $\Xi_0$ is 2D, at any point of $\Xi_0$, the stable \& unstable directions together with the 2D manifold tangent space span the entire phase space. This means that $\Xi_0$ is a \emph{normally hyperbolic invariant manifold} (NHIM) of the 2 DOF system's flow. For a rigorous definition of NHIMs for flows, see \cite{fenichel1971persistence}. 

\subsubsection{Example: planar circular restricted 3-body problem (PCR3BP)} \label{pcr3bpSection}
The PCR3BP is a 2 DOF Hamiltonian system which models the motion of an infinitesimally small particle (thought of as a spacecraft) under the gravitational influence of two large bodies of masses $m_{1}$ and $m_{2}$, collectively referred to as the primaries. In this model, $m_{1}$ and $m_{2}$ revolve about their barycenter in a circular orbit. Units are also normalized so that the distance between the two primaries becomes 1, their period of revolution becomes $2 \pi$, and $\mathcal{G}(m_{1}+m_{2})$ becomes 1. We define a mass ratio $\mu = \frac{m_{2}}{m_{1}+ m_{2}}$, and use a synodic, rotating non-inertial Cartesian coordinate system centered at the primaries' barycenter such that $m_{1}$ and $m_{2}$ are always on the $x$-axis. In the planar CRTBP, we also assume that the spacecraft moves in the same plane as the primaries. In this case, the equations of motion are generated by the Hamiltonian \citep{celletti}
\begin{equation}  \label{pcr3bpH} H_{0}(x,y,p_x,p_{y})= \frac{p_{x}^{2}+p_{y}^{2}}{2} + p_{x}y -p_{y}x - \frac{1-\mu}{r_{1}} - \frac{\mu}{r_{2}} \end{equation}
and Eq. \eqref{2dofH_EOM}. Here, $r_{1} = \sqrt{(x+\mu)^{2} + y^{2}}$ and $r_{2} = \sqrt{(x-1+\mu)^{2} + y^{2}} $ are the distances from the spacecraft to $m_{1}$ and $m_{2}$, respectively. Note that the PCR3BP equations of motion have a time-reversal symmetry; if $(x(t), y(t), t)$ is a solution of Eq. \eqref{2dofH_EOM}-\eqref{pcr3bpH} for $t > 0$, then $(x(-t), -y(-t), t)$ is a solution for $t < 0$. 

\subsection{Periodic Perturbations of 2 DOF Systems: the 2.5 DOF Case} \label{2_5dofSection}

Oftentimes in real-world applications, it is desirable to add a periodic perturbation to a 2 DOF Hamiltonian system. For instance,  the PCR3BP model includes gravitational forces from two large bodies, but a more accurate analysis including the effect of a third large body may at times be required (as described further in \ref{r4bpSection}). Adding such a periodic forcing effect to a 2 DOF Hamiltonian system often results in a non-autonomous, time-periodic 2.5 DOF Hamiltonian system on the 5D \emph{extended phase space} $\mathbb{R}^{4} \times \mathbb{T}$. The equations of motion in this case are given by Eq. \eqref{H_EOM} along with time-periodic Hamiltonian function \eqref{perturbed_H} 
\begin{equation} \label{H_EOM} \dot x = \frac{\partial H_{\varepsilon}}{\partial p_{x}} \quad \dot y = \frac{\partial H_{\varepsilon}}{\partial p_{y}} \quad \quad \dot p_{x} = -\frac{\partial H_{\varepsilon}}{\partial x} \quad \dot p_{y} = -\frac{\partial H_{\varepsilon}}{\partial y} \quad \quad \dot \theta_{p} = \Omega_{p} \end{equation}
\begin{equation}  \label{perturbed_H} H_{\varepsilon}(x,y,p_x,p_{y}, \theta_{p})= H_{0}(x,y,p_x,p_{y})+ H_{1}(x,y,p_x,p_{y}, \theta_{p}; \varepsilon)\end{equation}
where $\theta_{p} \in \mathbb{T}$ is an angle representing the phase of the periodic perturbation, $H_{0}$ is the 2 DOF Hamiltonian of Section \ref{2dofSection},  $H_{1}:\mathbb{R}^{4} \times \mathbb{T} \rightarrow \mathbb{R}$ is the Hamiltonian perturbation by the time-periodic effect and satisfies $H_{1}(x,y,p_x,p_{y}, \theta_{p}; 0)=0$, and $\varepsilon > 0$ and $\Omega_{p}$ are the perturbation parameter and perturbation frequency, respectively. $\varepsilon$ signifies the strength of the perturbation, with $\varepsilon=0$ being the unperturbed 2 DOF system, and $\Omega_{p}$ is a constant frequency which is assumed to be known a priori. The perturbation from $H_{1}$ is $2\pi / |\Omega_{p}|$ periodic. Note that the Hamiltonian function $H_{\varepsilon}$ will no longer be constant along trajectories. 

\subsubsection{Example: planar concentric circular restricted 4-body problem} \label{r4bpSection}

The planar concentric circular restricted 4-body problem \cite{blazevski2012} (CCR4BP) is a 2.5 DOF Hamiltonian system that results from a periodic third-body perturbation of the PCR3BP. It describes the motion of a spacecraft influenced by  gravity from \emph{three} large masses $m_{1}$, $m_{2}$, and $m_{3}$ with $m_{1} >> m_{2}, m_{3}$. $m_{2}$ and $m_{3}$ are assumed to revolve around $m_{1}$ in coplanar, concentric circles of radii $r_{12}$ and $r_{13}$, where $m_{2}$ has no effect on the motion of $m_{3}$ nor vice versa. Due to Kepler's third law, the angular velocities $\Omega_{2}$ and $\Omega_{3}$ of the revolution of $m_{2}$ and $m_{3}$ around $m_{1}$ depend on the masses and the orbital radii, as $\Omega_{i} = \sqrt{\mathcal{G}(m_{1}+m_{i})/r_{1i}^{3}}$ for $i = 2,3$ where $\mathcal{G}$ denotes the universal gravitational constant. In the planar CCR4BP, the circular orbits of $m_{2}$ and $m_{3}$ as well as the spacecraft trajectory are assumed to all lie in the same plane. 

Defining $\mu = \frac{m_{2}}{m_{1}+ m_{2}}$ and  $\varepsilon = \frac{m_{3}}{m_{1}+ m_{2}}$, and normalizing mass, length, and time units similarly to the $m_{1}$-$m_{2}$ PCR3BP---so that $\mathcal{G}(m_{1}+m_{2})$, $r_{12}$, and $\Omega_{2}$ all become $1$---the planar CCR4BP equations of motion can be written in the same rotating coordinate system usually used for the $m_{1}$-$m_{2}$ PCRTBP. $m_{1}$ and $m_{2}$ lie on the rotating frame $x$-axis, and the $m_{1}$-$m_{2}$ barycenter is taken as the frame origin. With these units and coordinate frame, the angle between the  $x$-axis and the vector from $m_{1}$ to $m_{3}$ becomes the perturbation phase angle $\theta_{p}$ as defined in Section \ref{2_5dofSection}. The equations of motion are then given by Eq. \eqref{H_EOM} with $\dot \theta_{p} = \Omega_3 - 1$ and the 2.5 DOF time-periodic Hamiltonian (see Blazevski and Ocampo \cite{blazevski2012} for a derivation)
\begin{equation}  \label{ccr4bpH}   H_{\varepsilon}(x,y,p_x,p_{y}, \theta_{p})= \frac{p_{x}^{2}+p_{y}^{2}}{2} + p_{x}y -p_{y}x  - \frac{1-\mu}{r_{1}} - \frac{\mu}{r_{2}}  - \frac{\varepsilon}{r_{3}} + \varepsilon \frac{x \cos \theta_{p}}{r_{13}^{2}} + \varepsilon \frac{y \sin \theta_{p}}{r_{13}^{2}}   \end{equation}
Here, $(x_{3}, y_{3}) = (-\mu+r_{13} \cos(\theta_{p}), r_{13} \sin(\theta_{p}) )$ give the position of $m_{3}$, while $r_{1} = \sqrt{(x+\mu)^{2} + y^{2}}$, $r_{2} = \sqrt{(x-1+\mu)^{2} + y^{2}}$, and $r_{3} = \sqrt{(x-x_{3})^{2} + (y-y_{3})^{2}}$ are the distances from the spacecraft to $m_{1}$, $m_{2}$, and $m_{3}$, respectively.  Note that when $\varepsilon=0$, Eq. \eqref{ccr4bpH} is just the $m_{1}$-$m_{2}$ Hamiltonian of Eq. \eqref{pcr3bpH}; indeed, the PCR3BP and CCR4BP fit into the framework of Section \ref{2_5dofSection} and Eq. \eqref{perturbed_H} by taking $H_{0}$ as in Eq. \eqref{pcr3bpH}, and
\begin{equation}  \label{ccr4bpH1}   H_{1}(x,y,p_x,p_{y}, \theta_{p}, \varepsilon)=  \varepsilon \left[ - \frac{1}{r_{3}} +  \frac{x \cos \theta_{p}}{r_{13}^{2}} +  \frac{y \sin \theta_{p}}{r_{13}^{2}}  \right]  \end{equation}

\subsection{Stroboscopic Maps, NHIMs, Tori, and Subharmonic Periodic Orbits} \label{stroboscopic}

Suppose that we have a family of unstable periodic orbits of a 2 DOF Hamiltonian flow. This family will form a 2D cylindrical NHIM, denoted $\Xi_0 \subset \mathbb{R}^{4}$. As mentioned in Section \ref{problemStateSection} for the case of maps, NHIMs persist under sufficiently small perturbations of the dynamics \citep{fenichel1971persistence}. However, to apply this persistence result to our case of time-periodic perturbations of 2 DOF Hamiltonian flows on $\mathbb{R}^{4}$, the original and perturbed systems must be defined on the same phase space. This is not the case for the perturbed, 2.5 DOF systems, whose phase space is $\mathbb{R}^{4} \times \mathbb{T}$ rather than $\mathbb{R}^{4}$, due to the addition of the perturbation phase angle $\theta_p \in \mathbb{T}$.

In order to pass from the extended phase space $\mathbb{R}^{4} \times \mathbb{T}$ to a dynamical system on $\mathbb{R}^{4}$ more amenable to perturbative analysis, one can use a \emph{stroboscopic map}. Recall from Section \ref{2_5dofSection} that the 2.5 DOF system phase angle $\theta_p$ has a constant frequency of $\Omega_p$, so that its period is $T_p = |2\pi/\Omega_p|$. The stroboscopic map $F_{\varepsilon}: \mathbb{R}^{4} \times \mathbb{T} \rightarrow \mathbb{R}^{4} \times \mathbb{T}$ is thus defined as the time-$T_p$ mapping of extended phase space points by the 2.5 DOF equations of motion \eqref{H_EOM} and \eqref{perturbed_H} with perturbation parameter $\varepsilon$. With this definition, given any $\theta_0 \in [0, 2\pi)$, the map $F_{\varepsilon}$ will map points from the phase space section $\Sigma_{\theta_0}=\{(x, y, p_{x}, p_{y}, \theta_{p}) \in \mathbb{R}^{4} \times \mathbb{T}: \theta_{p} = \theta_0\}$ into itself; this is because integrating $\theta_p$ by the dynamics $\dot \theta_p = \Omega_p$ for the stroboscopic mapping time $T_p= |2\pi/\Omega_p|$ will simply result in a full revolution of $\theta_p$ by $\Omega_p T_p = 2\pi$ or $-2\pi$, back to its initial angular value. 

Since $\Sigma_{\theta_0}$ is mapped into itself by $F_\varepsilon$ for any $\theta_0$, one can consider the dynamics of $F_\varepsilon$ restricted to any $\Sigma_{\theta_0}$; in essence, each $\Sigma_{\theta_0}$ is a Poincar\'e section for the 2.5 DOF system, with $F_{\varepsilon}$ being its corresponding Poincar\'e map. Moreover, each $\Sigma_{\theta_0}$ is clearly diffeomorphic to $\mathbb{R}^{4}$ through a simple projection $(x, y, p_x, p_y, \theta_0) \rightarrow (x, y, p_x, p_y)$ that drops the last non-varying coordinate. Thus, after fixing some value of $\theta_0$ and making a slight abuse of notation, one can in fact consider $F_{\varepsilon}$ to be a symplectic map from $\mathbb{R}^{4}$ into itself, as desired. Furthermore, while this discussion and stroboscopic map definition invoked the period $T_p$ of the perturbation from the 2.5 DOF system, one can just as well define a time-$T_p$ mapping $F_0: \mathbb{R}^{4} \rightarrow \mathbb{R}^{4}$ of points of $\mathbb{R}^{4}$ by the \emph{unperturbed}, 2 DOF system's equations of motion. We thus get a family of symplectic stroboscopic maps $F_\varepsilon:\mathbb{R}^{4} \rightarrow \mathbb{R}^{4}$ depending on a perturbation parameter $\varepsilon \geq 0$, with $F_{\varepsilon = 0}$ corresponding to the unperturbed system---just as in the problem setting described at the beginning of Section \ref{settingSection}. 

With the family of stroboscopic maps defined in the vein of Section \ref{settingSection}, we now characterize the NHIM $\Xi_0$ of 2 DOF flow-periodic orbits in this context as well. First of all, since $\Xi_0$ was a NHIM of the 2 DOF flow, it will also be a NHIM of the unperturbed map $F_0$, as the dynamics of $F_0$ are defined in terms of that same flow so that the same invariance and normal hyperbolicity properties will hold. Furthermore, since each constituent periodic orbit of $\Xi_0$ is invariant under the 2 DOF flow, it will also be invariant under $F_0$. As flow-periodic orbits are diffeomorphic to 1D tori, this means that each unstable periodic orbit of the 2 DOF flow corresponds to an unstable (whiskered) 1D invariant torus of $F_0$. Since $\Xi_0$ was by definition entirely foliated by unstable periodic orbits, it is hence entirely foliated by whiskered invariant tori of $F_0$, just as required in the problem setting of Section \ref{settingSection}. Furthermore, if the equations of motion Eqs. \eqref{H_EOM}-\eqref{perturbed_H} are differentiable at $\varepsilon = 0$ with respect to $\varepsilon$ in a neighborhood of $\Xi_{0}$, then the maps $F_{\varepsilon}$ will be as well.

Finally, with both the family of symplectic stroboscopic maps $F_{\varepsilon}$ as well as the $F_0$-invariant NHIM $\Xi_0$ of 2 DOF flow-periodic orbits understood to satisfy the assumptions of Section \ref{settingSection}, we now consider the properties of the $F_0$-invariant tori formed by those same flow-periodic orbits inside $\Xi_0$. Let $\bold{x}_{0} \in \Xi_0$ be a point on some periodic orbit of period $T$ under the 2 DOF flow; denoting the time-$t$ flow map of the 2 DOF system as $\phi(\bold{x}, t): \mathbb{R}^{4} \times \mathbb{R} \rightarrow \mathbb{R}^{4}$, and defining $K_0: \mathbb{T} \rightarrow \mathbb{R}^{4}$ as $K_0(\theta) = \phi(\bold{x}_{0}, \frac{\theta}{2\pi} T)$, we then have that
\begin{equation} \label{flowPoTorusParam} F_{0}(K_{0}(\theta)) = F_0\left(\phi\left(\bold{x}_{0}, \frac{\theta}{2\pi} T\right)\right) = \phi\left(\bold{x}_{0}, \frac{\theta}{2\pi} T + T_p \right) = K_{0} \left(\theta +  \frac{2\pi}{T}  T_p\right) \end{equation}
where the middle equality holds because $F_0$ is the time-$T_p$ flow map of the 2 DOF system. Noting that the image of $K_0$ is the entire $F_0$-invariant torus formed by the periodic orbit---for which it thus serves as a torus parameterization---and comparing Eq. \eqref{flowPoTorusParam} to Eq. \eqref{torusEquation}, it is clear that this torus has a rotation number of $\omega = 2\pi T_p / T$ under $F_0$. This last expression thus relates the orbital periods $T$ of 2 DOF flow-periodic orbits to the rotation numbers $\omega$ of the corresponding $F_0$-invariant tori. Now, assume that the flow-periodic orbits in at least some portion of the NHIM have monotonically varying periods $T$ under the 2 DOF flow. Then, the rotation numbers $\omega = 2\pi T_p / T$  of the corresponding $F_0$-invariant tori will satisfy the twist condition of Section \ref{settingSection}. Thus, the existence of tori with resonant rotation numbers, as defined in the same section, will also be guaranteed, and the remainder of the discussion of Sections \ref{settingSection}--\ref{problemStateSection} will hold as well. 

In summary, we have shown that by considering stroboscopic maps of 2.5 DOF Hamiltonian systems generated through periodic perturbations of 2 DOF Hamiltonian systems, if the 2 DOF system contains a family of unstable periodic orbits, the setting of Section \ref{settingSection}---which this paper will assume henceforth---is indeed attained with only minor additional assumptions on orbit periods and differentiability with respect to $\varepsilon$. These last assumptions generally hold true in practice; for example, both conditions hold when studying the 2.5 DOF planar CCR4BP perturbation of the PCR3BP, as long as the orbit of the perturbing third body does not intersect the PCR3BP unstable periodic orbit family considered in $(x,y)$ space. Also, as an added benefit, the dimension-reducing nature of stroboscopic maps makes certain calculations and visualizations simpler. Thus, in the remainder of this paper, we will present our methods only for the case of the 4D symplectic maps of Section \ref{settingSection}, but without any loss of generality to 2.5 DOF Hamiltonian flows as well.

\begin{remark} 
As a final clarifying discussion, we also present an alternative description of resonant tori as well as subharmonic periodic orbits for the flow-derived stroboscopic map setting. Namely, recall from Section \ref{settingSection} that resonant tori are those with $\frac{\omega}{2\pi}$ rational. Given that $\omega = 2\pi T_p / T$ here, resonant $F_0$-invariant tori thus correspond to orbits with $T_p/T$ rational---in other words, a resonance between the 2 DOF flow-periodic orbit's period and the perturbation period. In this case, if $\frac{\omega}{2\pi} = {T_p}/{T} = p/q$, then any subharmonic map-periodic orbits of Section \ref{settingSection} that persist into the $\varepsilon>0$ case will have periods of $q$ stroboscopic map iterations of $F_{\varepsilon}$---equivalent to time $qT_p$ under the corresponding 2.5 DOF system's flow. 
\end{remark}

\section{An Efficient Method for Computing Subharmonic Periodic Orbits and Floquet Vectors} \label{spoSection}

In this section, we develop and implement an efficient quasi-Newton algorithm for the simultaneous computation of unstable subharmonic periodic orbits, as well as their Floquet directions and multipliers, for the perturbative families of 4D symplectic maps $F_{\varepsilon}$ described in Section \ref{problemSection}. We present the analytical details and derivation of the method, as well as the considerations required for its numerical implementation. This method is heavily inspired by the parameterization method of \cite{kumar2022} for computing invariant tori with their center, stable, and unstable directions, and much of the discussion and proofs  in this section follow a very similar structure to those of that paper. However, the adaptation to handle periodic orbits instead of tori requires a number of significant modifications as well, which will be highlighted throughout this section.

\subsection{The Parameterization Method for Invariant Manifolds} \label{paramsectiongeneral}

The parameterization method is a general technique for the computation of many kinds of invariant objects in dynamical systems. \cite{haroetal} describe several applications. The idea is that given a map $F: \mathbb{R}^{d} \rightarrow \mathbb{R}^{d}$, if we know that there is an $F$-invariant object diffeomorphic to some model manifold $\mathcal{M}$, then we can solve for a function $W:\mathcal M \rightarrow \mathbb{R}^{d}$ and a diffeomorphism $f: \mathcal M \rightarrow \mathcal M$ such that the invariance equation
\begin{equation}  \label{invariancequation}   F(W(s)) = W(f(s)) \end{equation}
holds for all $s \in \mathcal M$. $W$ is referred to as the parameterization of the invariant manifold, and $f$ as the internal dynamics on $\mathcal M$. Eq. \eqref{invariancequation} means that $F$ maps the image $W(\mathcal M)$ into itself, so that $W(\mathcal M)$ is the invariant object in the full space $\mathbb{R}^{d}$. 

\subsection{Parameterization Method-Style Equations for Subharmonic Periodic Orbits and Floquet Vectors} \label{quasiNewton}

Assume that for some $\varepsilon>0$, we aim to compute a subharmonic periodic orbit (SPO, from this point onwards) of $F_{\varepsilon}$ corresponding to a $q$-iteration unperturbed-map SPO (of $F_0$) that is expected to persist. As described in Section \ref{problemStateSection}, we thus wish to find points $X_{\varepsilon}(k) \in \mathbb{R}^{4}$ of the $F_{\varepsilon}$-SPO satisfying the equation 
\begin{equation} \label{invariance} F_\varepsilon(X_\varepsilon(k)) = X_\varepsilon(k+1 \mod q) \end{equation}
for all $k = 0, 1, \dots, q-1$. Eq. \eqref{invariance} can be interpreted similarly to the framework of Section \ref{paramsectiongeneral} with $\mathcal M = \{0, 1, \dots, q-1\}$ and $f(s) = s+1 \mod q$; it also bears similarities to the torus invariance equation of Eq. \ref{torusEquation}, which will allow for the adaptation of torus parameterization methods from \cite{kumar2022} to SPOs. 

In addition, recall from Section \ref{problemStateSection} that we also sought to solve Eq. \eqref{floquetEquation} for the Floquet directions and multipliers of the SPO given by $X_{\varepsilon}$. For our quasi-Newton method, we will thus add an equation of the same form, but with slightly different assumptions on the Floquet matrix that will nevertheless then enable easy solution of Eq. \eqref{floquetEquation} as well. In particular, we will seek matrices $P_{\varepsilon}(k), \Lambda_{\varepsilon}(k) \in \mathbb{C}^{4 \times 4}$ such that
\begin{equation} \label{bundleEquations} DF_\varepsilon(X_\varepsilon(k)) P_{\varepsilon}(k) = P_\varepsilon(k+1 \mod q) \Lambda_{\varepsilon}(k) \end{equation}
for all $k = 0, 1, \dots, q-1$. Furthermore, we will mandate that each $\Lambda_{\varepsilon}(k)$ has the \emph{near-diagonal} form 
\begin{equation} \label{lambdaForm} \Lambda_{\varepsilon}(k) = \begin{bmatrix}
\lambda_1 &  T   & 0 & 0 \\ 0 &  \lambda_2   & 0 & 0 \\ 0 & 0  & \lambda_s(k) & 0 \\ 0 &  0 & 0 & \lambda_u(k) \end{bmatrix}  \end{equation}
as opposed to the fully-diagonal $\bar \Lambda_{\varepsilon}$ of Section \ref{problemStateSection}. Here, $T, \lambda_{1}, \lambda_{2} \in \mathbb{C}$ and $\lambda_{s}(k), \lambda_{u}(k) \in \mathbb{R}$ are to be found. 

In the above $\Lambda_{\varepsilon}(k)$, we adopt the convention that $\lambda_{s}(k)$ and $\lambda_{u}(k)$ represent  stable and unstable SPO Floquet multipliers corresponding to directions transverse to the 2D NHIM $\Xi_{\varepsilon}$ that contains the SPO (recall Section \ref{problemStateSection}). $\lambda_{1}$ and $\lambda_{2}$ on the other hand will represent multipliers of the dynamics tangent to $\Xi_{\varepsilon}$. 
Recall that the $\bar P_{\varepsilon}(k) \in \mathbb{R}^{4 \time 4}$ defined in Section \ref{problemStateSection} contained Floquet directions of the SPO at the point $X_{\varepsilon}(k)$, with its columns being eigenvectors of the monodromy matrix $DF_{\varepsilon}^{q} (X_{\varepsilon}(k))$. In a similar vein, it is not hard to see that columns 1, 3, and 4 of $P_{\varepsilon}(k)$ above will also be eigenvectors of $DF_{\varepsilon}^{q} (X_{\varepsilon}(k))$ corresponding to eigenvalues $\lambda_{1}^{q}$, $\prod_{k=0}^{q-1} \lambda_{s}(k)$, and $\prod_{k=0}^{q-1} \lambda_{u}(k)$, respectively. The fourth eigenvalue of $DF_{\varepsilon}^{q} (X_{\varepsilon}(k))$ will be $\lambda_{2}^{q}$; however, column 2 of $P_{\varepsilon}$ will \emph{not} be its corresponding eigenvector, due to the off-diagonal term $T$ in Eq. \eqref{lambdaForm}.\ Instead, column 2 will lie in the span of the eigenvectors corresponding to $\lambda_{1}^{q}$ and $\lambda_{2}^{q}$. Nevertheless, given a solution of Eqs. \eqref{invariance}-\eqref{bundleEquations}, it will be very easy to determine this last eigenvector as well. 

As will be explained at the end of Section \ref{Xstep}, solving simultaneously for $X_{\varepsilon}$, $P_{\varepsilon}$, and $\Lambda_{\varepsilon}$ is actually more efficient than solving for $X_{\varepsilon}$ alone. Similarly to the torus parameterization method of \cite{kumar2022} from which it is adapted, the quasi-Newton method we will present for solving Eq. \eqref{invariance}-\eqref{bundleEquations} uses the near-diagonal form of $\Lambda_{\varepsilon}$ to decouple the linear system of equations we get in each differential correction step. The method will require only algebraic operations, index shifts, and the solving of 1D equations for scalar-valued sequences. 
Finally, note that Eq. \eqref{bundleEquations} is actually underdetermined; we can change the scales of the stable and unstable directions at each $k$; we will take advantage of this in Section \ref{constantLambda} to make $\Lambda_{\varepsilon}$ constant. 

\begin{remark}
Note that given any invertible matrix $V \in \mathbb{C}^{4 \times 4}$, if $X_{\varepsilon}$, $P_{\varepsilon}$, $\Lambda_{\varepsilon}$ are a solution of Eqs. \eqref{invariance}-\eqref{bundleEquations}, then the matrices $\tilde P_{\varepsilon}(k) = P_{\varepsilon}(k) V$ and $\tilde \Lambda_{\varepsilon}(k) = V^{-1} \Lambda_{\varepsilon}(k) V$ will also yield solutions of Eq. \eqref{bundleEquations}. This will allow us to change the form of $\Lambda_{\varepsilon}$ between that of Eq. \eqref{lambdaForm}, the desired fully-diagonal form of Section \ref{problemStateSection}, and others as well that will occur throughout the quasi-Newton procedure. 
\end{remark}

\subsection{Preliminaries: Solution Methods for Common Equations and Procedures} \label{prelimSection}

Throughout the discussion to follow, a number of basic procedures are repeatedly used. Therefore, before presenting the details of the quasi-Newton method initialization and stepping algorithms, we discuss some of these results and procedures that will be required later in Sections \ref{continuationSection}-\ref{Pstep}.

\subsubsection{Fixed-point iteration for equations of form $\lambda_{a}(k) u(k) - \lambda_{b}(k) u(k+1 \mod q) = b(k)$} \label{fixedPointIter}

Throughout the following sections, we will repeatedly encounter equations of the form 
\begin{equation} \label{cohomHyp} \lambda_{a}(k) u(k) - \lambda_{b}(k) u(k+1 \mod q) = b(k) \end{equation}
where $b(k), \lambda_{a}(k)$, and $\lambda_{b}(k)  \in \mathbb{C}$ are known for all $k = 0, 1, \dots, q-1$, and the $u(k) \in \mathbb{C}$ are to be found. The $\lambda_a(k)$ and $\lambda_{b}(k)$ will be equal to $\lambda_1$, $\lambda_2$, $\lambda_s(k)$, $\lambda_u(k)$, or 1. While one can write explicit formulas for the $u(k)$ here (see Section \ref{cohomHypFormulaSection}), if $| \lambda_a(k)/\lambda_b(k)| > 1$ or $| \lambda_a(k)/\lambda_b(k)| <1$ for all $k = 0, 1, \dots, q-1$ (as usually occurs), it is more numerically stable to use fixed point iteration instead. 
For this, rewrite Eq. \eqref{cohomHyp} as
\begin{equation} \label{xi3contract} u(k) = \left[\lambda_{a}(k-1 \mod q)u(k-1 \mod q) -  b(k-1 \mod q)\right] / \lambda_{b}(k-1 \mod q) \stackrel{\text{def}}{=} [A(u)](k)\end{equation}    
if $| \lambda_a(k)/\lambda_b(k)| <1$ for all $k = 0, 1, \dots, q-1$. If instead $| \lambda_a(k)/\lambda_b(k)| > 1$, then rewrite Eq, \eqref{cohomHyp} as
\begin{equation} \label{xi4contract} u(k) = \lambda_{a}^{-1}(k) \left[ b(k) +  \lambda_{b}(k) u(k+1 \mod q) \right]  \stackrel{\text{def}}{=}  [B(u)](k) \end{equation}  
We define $A$ and $B$ above as maps which send any finite sequence $\{u(k)\}_{k = 0, \dots, q-1}$ to the new finite sequences $A(u)$ and $B(u)$ with $k$\textsuperscript{th} terms given by the middle expressions of Eq. \eqref{xi3contract}-\eqref{xi4contract} for $k = 0, \dots, q-1$. 
It is then easy to show (see \ref{contractionProof}) that if $| \lambda_a(k)/\lambda_b(k)| <1$ for all $k = 0, 1, \dots, q-1$, $A$ is a contraction under the $\ell^\infty$ norm, and similarly for $B$  if all $| \lambda_a(k)/\lambda_b(k)| >1$ instead. Thus, to find $u$, let $u_0(k) = 0$ for all $k = 0, \dots, q-1$, and repeatedly iterate $u_{n+1} = A(u_{n})$ (if all $| \lambda_a(k)/\lambda_b(k)| < 1$) or $u_{n+1} = B(u_{n})$ (if all $| \lambda_a(k)/\lambda_b(k)| > 1$), starting at $n=0$. By the contraction mapping theorem  \cite{chicone2006}, the iteration will converge to the desired solution sequence $u$ of Eq. \eqref{xi3contract} or  \eqref{xi4contract}, and thus also of Eq. \eqref{cohomHyp}.

\subsubsection{Explicit formulas for equations of form $\lambda_{a}(k) u(k) - \lambda_{b}(k) u(k+1 \mod q) = b(k)$} \label{cohomHypFormulaSection}

The fixed-point iteration method of solving Eq. \eqref{cohomHyp} presented in the previous section relies on the assumption that all $| \lambda_a(k)/\lambda_b(k)| >1$ or all $| \lambda_a(k)/\lambda_b(k)| < 1$ for $k = 0, 1, \dots, q-1$. However, even if $\lambda_a(k)$ and $\lambda_a(k)$ are not all both 1 (a case handled in the following subsection), it can be that $| \lambda_a(k)| = |\lambda_b(k)| = 1$ for all $k = 0, 1, \dots, q-1$. This can occur, for example, if all $\lambda_a(k) = \lambda_1 $, $\lambda_b(k) = \lambda_2$, and $\lambda_1$ and $\lambda_2$ are a complex conjugate pair of elliptic eigenvalues. As all cases of this type involve constant $\lambda_a$ and $\lambda_{b}$ (independent of $k$), we will simply write $\lambda_a(k)$ and $\lambda_b(k)$ as $\lambda_a$ and $\lambda_{b}$ from this point onwards.

To address such cases, it is possible to write an explicit formula for the solution of Eq. \eqref{cohomHyp}. For this, one simply takes the relation $u(k) = \lambda_{a}^{-1} \left[ b(k) +  \lambda_{b} u(k+1 \mod q) \right] $, evaluates it at $k+1 \mod q$ to get $u(k+1 \mod q) = \lambda_{a}^{-1} \left[ b(k +1 \mod q) +  \lambda_{b} u(k+2 \mod q) \right] $, substitutes this into the expression for $u(k)$, and continues similarly with $u(k+2 \mod q)$ and so forth $q$ times to find a relation that yields 
\begin{equation} \label{cohomHypSoln} u(k) = \frac{\lambda_{a}^{-1}}{1- (\lambda_b/\lambda_a)^q }\sum_{i=0}^{q-1} ({\lambda_{b}}/{\lambda_{a}})^{i} b(k+i \mod q)  \end{equation} 
In practice, it is more efficient to evaluate $u(0)$ using the above formula, and then recursively calculate the remaining $u(k)$ using the relation $u(k+1 \mod q) = \left[\lambda_{a}u(k) -  b(k)\right] / \lambda_{b}$. Note that if $|\lambda_b/\lambda_a| > 1$, the above formula is highly numerically unstable, so that in such cases the methods of Section \ref{fixedPointIter} are used instead. Also note that Eq. \eqref{cohomHypSoln} is undefined if $(\lambda_b/\lambda_a)^{q} = 1$, which will occur if $\lambda_a = \lambda_b = 1$. We address this next.

\subsubsection{``Cohomological equations'': the case $u(k) -  u(k+1 \mod q) = b(k)$} \label{cohomSection}

The final case of Eq. \eqref{cohomHyp} is when $\lambda_a(k) = \lambda_b(k) = 1$ for all $k$. In this case, we get an equation of form
\begin{equation} \label{cohomological}u(k) - u(k+1 \mod q) =  b(k) \end{equation}
This equation is analogous to the cohomological equations $u(\theta) - u(\theta+\omega) = b(\theta)$ involved in computing invariant tori and their Floquet directions, e.g. in \cite{kumar2022}. 
Notice that similar to the torus case, in which $b$ must have zero average, summing both sides of Eq. \eqref{cohomological} for $k = 0, 1, \dots, q-1$ yields the necessary condition $\sum_{k=0}^{q-1} b(k) =0$ here as well. Note that if $u(k)$ is a solution of Equation \eqref{cohomological}, then so is $u(k)+C$ for any $C \in \mathbb{C}$, making the solution non-unique. Thus, to solve for $u(k)$, one can set $u(0) = 0$ arbitrarily, and then recursively find $u(1), \dots, u(q-1)$ using Equation \eqref{cohomological} for $k = 0, \dots, q-2$. Equation \eqref{cohomological} will then automatically also be satisfied for $k=q-1$ due to the condition $\sum_{k=0}^{q-1} b(k) =0$.

\subsubsection{Procedure for making $\Lambda_{\varepsilon}$ independent of $k$} \label{constantLambda}

Suppose that one has a solution $X_{\varepsilon}$, $P_{\varepsilon}$, and $\Lambda_{\varepsilon}$ of Eqs. \eqref{invariance}-\eqref{bundleEquations} with each $\Lambda_{\varepsilon}(k)$ being of the form given by Eq. \eqref{lambdaForm}. In these $\Lambda_{\varepsilon}(k)$ matrices, all entries are thus constant (independent of $k$) except for the $\lambda_{s}(k)$ and $\lambda_{u}(k)$. However, through a rescaling of the third and fourth columns of $P_{\varepsilon}(k)$, it is possible to find a new solution $X_{\varepsilon}$ (unchanged), $\tilde P_{\varepsilon}$ (new), and $\tilde \Lambda_{\varepsilon}$ (new) such that the $\tilde \Lambda_{\varepsilon}(k)$ will be constant for all $k$. As such solutions can help ensure better numerical stability during numerical continuation, we now discuss how to ``make $\Lambda_{\varepsilon}$ constant''. This procedure will be used in Sections \ref{initCols34} and \ref{Pstep}. 

Denote columns 3 and 4 of $P_{\varepsilon}(k)$ as $\bold{v}_{s}(k)$ and $\bold{v}_{u}(k)$, respectively. Now, set $\bar \lambda_s = \exp \left[ \frac{1}{q}\sum_{0}^{q-1} \log (\lambda_s(k))  \right]$ and $\bar \lambda_u = \exp \left[ \frac{1}{q}\sum_{0}^{q-1} \log (\lambda_u(k))  \right]$ and let $a_s(k),a_u(k) \in \mathbb{R}$, $k = 0, 1, \dots, q-1$ be the solutions to
\begin{gather} \label{scalevs} \log(a_s(k))  - \log(a_s(k+1 \mod q)) = - [\log (\lambda_s(k)) - \log (\bar \lambda_s) ] \\
 \label{scalevu}  \log(a_u(k)) - \log(a_u(k+1 \mod q)) = -[ \log (\lambda_u(k)) - \log (\bar \lambda_u) ] \end{gather}
The aforementioned values of $\bar \lambda_s, \bar \lambda_u \in \mathbb{R}$ ensure that the LHS of both Eqs. \eqref{scalevs}--\eqref{scalevu}  will have zero sum over all $k$. Thus, letting $u(k) = \log(a_s(k))$, Eq. \eqref{scalevs} becomes an equation of form Eq. \eqref{cohomological} which can be solved for all $u(k)$ by the methods of Section \ref{cohomSection}. This gives $a_{s}(k)=e^{u(k)}$. We can solve Eq. \eqref{scalevu} for $a_{u}(k)$ in the exact same manner. Finally, one should replace columns 3 and 4 of each $P_{\varepsilon}(k)$ by $ \bold{\tilde v}_{s} (k)= a_s(k) \bold{v}_{s}(k)$ and $ \bold{\tilde v}_{u} (k)= a_u(k) \bold{v}_{u}(k)$ respectively to get the desired $\tilde P_{\varepsilon}$, and replace $\lambda_{s}(k)$ and $\lambda_{u}(k)$ in each $\Lambda_{\varepsilon}(k)$ by $\bar \lambda_s$ and $\bar \lambda_{u}$ to get $\tilde \Lambda_{\varepsilon}$. We prove that the resulting $\tilde P_{\varepsilon}$ and $\tilde \Lambda_{\varepsilon}$ indeed satisfy Eq. \eqref{bundleEquations} in \ref{rescaleProof}. 

\subsection{Initialization for Continuation by $\varepsilon$} \label{continuationSection}

To compute the desired $F_{\varepsilon}$-subharmonic periodic orbit points $X_{\varepsilon}(k)$ and matrices $P_{\varepsilon}$, $\Lambda_{\varepsilon}$ solving Eqs. \eqref{invariance}-\eqref{bundleEquations} for some desired perturbation $\varepsilon = \varepsilon_{f} > 0$, we will start from the corresponding solution of the unperturbed case $\varepsilon= 0$ and numerically continue by $\varepsilon$ until the desired $F_{\varepsilon}$-SPO and matrices are found for $\varepsilon = \varepsilon_{f}$. The quasi-Newton method we will present in Sections \ref{Xstep}-\ref{Pstep} will enable this continuation procedure: choose a number of continuation steps $n$, take an SPO $X_{\varepsilon}$ and matrices $P_{\varepsilon}$, $\Lambda_{\varepsilon}$ from the $\varepsilon = 0$ system, and use them to help generate an initial guess for the quasi-Newton method to solve for the SPO and matrices in the $\varepsilon = \varepsilon_{f}/n$ system. Similarly, for $i = 1, \dots, n-1$, use the solution from the $\varepsilon_{f} i/n$ system to generate an initial guess for the solution in the $\varepsilon_{f} (i+1)/ n$ system. Once $i=n-1$, we will have the SPO and matrices for $\varepsilon = \varepsilon_{f}$. First, however, one must find the $\varepsilon = 0$ solution to initialize the continuation.

As described in Section \ref{problemStateSection}, we assume that the points $X_{0}(k)$, $k=0, 1, \dots, q-1$, corresponding to an $F_0$-SPO that persists into the perturbed map $F_{\varepsilon}$, are already known, e.g., from a Melnikov-type analysis. However, we still need $P_{0}(k)$ and $\Lambda_{0}(k)$ solving Eqs. \eqref{bundleEquations} for the $\varepsilon = 0$ case, with $\Lambda_{0}$ of the form of Eq. \eqref{lambdaForm}. Also, for later steps in the procedure, it will be desirable for all the $\lambda_s(k)$ and $\lambda_u(k)$ from $\Lambda_0(k)$ to be taken positive. We now describe how to construct such a solution $X_{0}$,  $P_{0}$, $\Lambda_{0}$ to Eqs. \eqref{invariance}--\eqref{bundleEquations} for $\varepsilon = 0$. 

\subsubsection{Columns 3 and 4 of $P_0$ and the multipliers $\lambda_s, \lambda_u$} \label{initCols34}

Recall from Section \ref{quasiNewton} that the $\lambda_s(k)$ and $\lambda_u(k)$ are stable and unstable Floquet multipliers respectively corresponding to stable and unstable Floquet directions of the SPO at $X_0(k)$. These directions,  which will form columns 3 and 4 of $P_{0}(k)$, must be eigenvectors of the SPO monodromy matrix $DF_{0}^{q}(X_{0}(k))$. As $X_0(k)$ lies on an unstable resonant $F_{0}$-invariant torus, $DF_{0}^{q}(X_{0}(k))$ will have real eigenvalues 1, 1, $\lambda$, and $\lambda^{-1}$ with $|\lambda| > 1$. However, if $q$ is large, it may be numerically difficult to find eigenvectors for $\lambda$ and $\lambda^{-1}$ directly, as the matrix $DF_{0}^{q}(X_{0}(k))$ will numerically have very large entries, with $\lambda$ very large and $\lambda^{-1}$ very small. 

Denote the desired stable/unstable monodromy matrix eigenvectors at $X_0(k)$, corresponding to eigenvalues $\lambda^{-1}$ and $\lambda$, as $\bold{v}_{s}(k)$ and $\bold{v}_{u}(k)$, respectively. To find $\bold{v}_{s}(k)$ and $\bold{v}_{u}(k)$, rather than handling the matrix $DF_{0}^{q}(X_{0}(k))$, one can instead use an iterative approach. 
For this, first set $\bold{v}_{s,0}(0)$ and $\bold{v}_{u,0}(0)$ equal to arbitrary 4D unit vectors. Then, define $\bold{v}_{s,0}(k)$ and $\bold{v}_{u,0}(k)$ for $k=1, \dots, q-1$ by the recursive relations
\begin{equation}  \label{evecIterS} \bold{v}_{s,i}(k) = \frac{DF_0(X_0(k))^{-1} \bold{v}_{s,i}(k+1 \mod q)}{\| DF_0(X_0(k))^{-1} \bold{v}_{s,i}(k+1 \mod q) \|} \end{equation} 
\begin{equation} \label{evecIterU} \bold{v}_{u,i}(k) = \frac{DF_0(X_0(k-1 \mod q)) \bold{v}_{u,i}(k-1 \mod q)}{\| DF_0(X_0(k-1 \mod q)) \bold{v}_{u,i}(k-1 \mod q) \|}  \end{equation} 
with $i=0$. This yields initial sequences $\bold{v}_{s,0}(k)$ and $\bold{v}_{u,0}(k)$, $k = 0, 1, \dots, q-1$ to start the iteration. Now, given vector sequences $\bold{v}_{s,i}(k)$ and $\bold{v}_{u,i}(k)$, define the next iteration $\bold{v}_{s,i+1}(k)$ and $\bold{v}_{u,i+1}(k)$ by setting
\begin{equation} \label{evecIterS0}  \bold{v}_{s,i+1}(0) = \frac{DF_0(X_0(0))^{-1} \bold{v}_{s,i}(1 \mod q)}{\| DF_0(X_0(0))^{-1} \bold{v}_{s,i}(1) \|} \end{equation}
\begin{equation} \label{evecIterU0}  \bold{v}_{u,i+1}(0) = \frac{DF_0(X_0(q-1)) \bold{v}_{u,i}(q-1)}{\| DF_0(X_0(q-1)) \bold{v}_{u,i}(q-1 ) \|} \end{equation}
and then defining $\bold{v}_{s,i+1}(k)$, $\bold{v}_{u,i+1}(k)$ for $k=1, \dots, q-1$ using the above $ \bold{v}_{s,i+1}(0) $, $ \bold{v}_{u,i+1}(0) $ and the recursion of Eqs. \eqref{evecIterS}--\eqref{evecIterU} (with $i+1$ in place of $i$). This procedure will generate sequences of unit vectors $\bold{v}_{s,i}(k)$ and $\bold{v}_{u,i}(k)$ such that if $\lambda$ and $\lambda^{-1}$ are positive, then the limits $\lim_{i\rightarrow \infty} \bold{v}_{s,i}(k)$ and $\lim_{i\rightarrow \infty} \bold{v}_{u,i}(k)$ will both exist and will respectively converge to valid $\bold{v}_{s}(k)$ and $\bold{v}_{u}(k)$ for each $k = 0, 1, \dots, q-1$. Furthermore, taking limits of Eqs. \eqref{evecIterS}--\eqref{evecIterU0} as $i \rightarrow \infty$, evaluating Eq. \eqref{evecIterU} at $k+1$, and rearranging yields that 
\begin{equation} \label{evecSConfirm} DF_0(X_0(k)) \bold{v}_{s}(k) = {\| DF_0(X_0(k))^{-1} \bold{v}_{s}(k+1 \mod q) \|^{-1}} \bold{v}_{s}(k+1 \mod q)\end{equation}
\begin{equation} \label{evecUConfirm}  {DF_0(X_0(k)) \bold{v}_{u}(k)} = {\| DF_0(X_0(k)) \bold{v}_{u}(k) \|}  \bold{v}_{u}(k+1 \mod q) \end{equation}
for all $k$. Recalling that $\bold{v}_{s}(k)$, $\bold{v}_{u}(k)$ are columns 3 and 4 of $P_{0}(k)$, and comparing Eqs. \eqref{evecSConfirm}--\eqref{evecUConfirm} to columns 3 and 4 of Eq. \eqref{bundleEquations}, one can conclude that with $\bold{v}_{s}(k)$ and $\bold{v}_{u}(k)$ thus defined, $\lambda_s(k) = \| DF_0(X_0(k))^{-1} \bold{v}_{s}(k+1 \mod q) \|^{-1}$ and $\lambda_u(k) = \| DF_0(X_0(k)) \bold{v}_{u}(k) \|$---which are both positive for all $k = 0, 1, \dots, q-1$, as desired. 

While the previous iteration yields unit-length stable and unstable eigenvectors of $DF_{0}^{q}(X_{0}(k))$ when the stable/unstable eigenvalues $\lambda^{-1}$ and $\lambda$ of $DF_{0}^{q}(X_{0}(k))$ are positive, the iteration will fail to converge if $\lambda < 0$;  failure of the iteration in fact \emph{implies} that $\lambda < 0$. In such a case, note that if $\lambda<0$ and $\lambda^{-1}<0$ are eigenvalues of the $q$-iteration monodromy matrix $DF_{0}^{q}(X_{0}(k))$, then the $2q$-iteration monodromy matrix $DF_{0}^{2q}(X_{0}(k)) = [DF_{0}^{q}(X_{0}(k))]^{2}$ will have \emph{positive} eigenvalues $\lambda^{2}$ and $\lambda^{-2}$. Thus, if one considers the SPO of interest to be a $2q$-iteration-long $F_0$-periodic orbit rather than a $q$-iteration orbit, its monodromy matrix will have only positive eigenvalues---and procedure of Eqs. \eqref{evecIterS}-\eqref{evecIterU0} will converge to valid choices of $\bold{v}_{s}(k)$, $\bold{v}_{u}(k)$ with $\lambda_s(k), \lambda_u(k) > 0$. Hence, given an SPO with $X_0(k)$, $k = 0, 1, \dots, q-1$ for which the previous iteration fails to converge, define a new length-$2q$ sequence $\tilde X_0(k)$, $k = 0, 1, \dots, 2q -1$ with values defined by $\tilde X_0(k) = X_0(k \mod q)$. Considering $\tilde X_{0}$ rather than $X_0$ (and with $2q$ playing the role of $q$) will then allow for the procedure of Eqs. \eqref{evecIterS}-\eqref{evecIterU0}, as well as all further steps in this paper, to proceed as desired. 

The above methodology yields valid columns 3 and 4 of each $P_{\varepsilon}(k)$ given by $\bold{v}_{s}(k)$ and $\bold{v}_{u}(k)$, as well as their corresponding Floquet multipliers $\lambda_{s}(k), \lambda_{u}(k)>0$, that will satisfy columns 3 and 4 of Eq. \eqref{bundleEquations} for all $k = 0, 1, \dots, q-1$. However,  the $\lambda_{s}(k)$ and $\lambda_{u}(k)$ thus found will depend on $k$. As a final step, to improve the numerical stability of later computations, one should rescale $\bold{v}_{s}(k)$ and $\bold{v}_{u}(k)$ to ensure constant $\lambda_{s}$ and $\lambda_{u}$; the method of Section \ref{constantLambda} can be used for this. Making a slight abuse of notation,  in the next Section \ref{initCols12}, we will refer to these rescaled vectors and multipliers as $\bold{v}_{s}(k)$, $\bold{v}_{u}(k)$, $\lambda_s$, and $\lambda_u$ as well. 

\subsubsection{Columns 1 and 2 of $P_0$ and $\Lambda_0$} \label{initCols12}

With columns 3 and 4 of $P_{0}(k)$ and $\Lambda_0(k)$ all found as just described, our focus now turns to columns 1 and 2. Starting with column 1, recall from Section \ref{settingSection} that the SPO for $\varepsilon = 0$ lies on a resonant $F_{0}$-invariant torus of rotation number $\omega = 2\pi p/q$ ($p,q \in \mathbb{Z}$) parameterized by some function $K_0(\theta): \mathbb{T} \rightarrow \mathbb{R}^{4}$ satisfying Eq. \eqref{torusEquation}. Differentiating Eq. \eqref{torusEquation} for $K = K_{0}$ with respect to $\theta$  yields that $DF_{0}(K_{0}(\theta))DK_{0}(\theta) = DK_{0}(\theta + \omega)$. Now, since the points $X_0(k)$ correspond to $X_0(k) = K_0(\theta_{0}+k\omega)$ for some $\theta_{0} \in \mathbb{T}$ and $k = 0, 1, \dots, q-1$, the differentiated Eq. \eqref{torusEquation} then implies that (denote $\theta_{k} = \theta_{0}+ k\omega$ henceforth)
\begin{equation} \label{col1Lambda} DF_{0}(X_{0}(k))DK_{0}(\theta_{k} ) = DK_{0}(\theta_{k+1 \mod q} ) \end{equation}
Now, if we set column 1 of each $P_{0}(k)$ to be $DK_{0}(\theta_{k} )$, and $\lambda_1 =1$ in Eq. \eqref{lambdaForm} for all $\Lambda_0(k)$, then Eq. \eqref{col1Lambda} implies that column 1 of Eq. \eqref{bundleEquations} will automatically be satisfied. Thus, we set $\lambda_1 =1$ and column 1 of each $P_{0}(k)$ in this manner. Also note that if $F_0$ is the stroboscopic map of a 2 DOF Hamiltonian flow, then each $DK_{0}(\theta_{k} )$ is a multiple of the flow vector at $X_0(k)$ by a constant scaling factor; if this is the case, then one can use these flow vectors in place of $DK_{0}(\theta_{k} )$ as well, as the former may be easier to compute. 

With columns 1, 3, and 4 of $P_{0}(k)$ all determined along with $\lambda_{1}$, $\lambda_{s}$, and $\lambda_{u}$ of $\Lambda_0$, the last step is to find $\lambda_{2}$, $T$, and column 2 of $P_{0}(k)$. This will require some extra calculations, again leveraging the fact that the $F_0$-SPO lies on a resonant $F_0$-invariant torus parameterized by $K_0$, with $X_0(k) = K_0(\theta_{k})$. The first step in their computation is to find $A(k), B(k), C(k)$, and $D(k) \in \mathbb{R}$ for each $k = 0, 1, \dots, q-1$ such that
\begin{align} \label{abcd} \begin{split} DF_0(X_0(k)) \frac{J^{-1} DK_0(\theta_k)}{ \|DK_0(\theta_k)\|^{2}} = A(k) &DK_0(\theta_{k+1 \mod q}) + B(k) \frac{J^{-1} DK_0(\theta_{k+1 \mod q})}{ \|DK_0(\theta_{k+1 \mod q})\|^{2}} \\ 
&+ C(k) \bold{v}_{s}(k+1 \mod q) + D(k) \bold{v}_{u}(k+1 \mod q) 
\end{split} \end{align} 
where $  J= \begin{bmatrix}
0_{2 \times 2}   & I_{2 \times 2}   \\ -I_{2 \times 2}  &  0_{2 \times 2} \end{bmatrix} $ is the matrix of the symplectic form in the usual Euclidean metric on $\mathbb{R}^{4}$, and $\bold{v}_{u}(k)$, $\bold{v}_{u}(k)$ are those found at the end of Section \ref{initCols34}. All the quantities in \eqref{abcd} are known except $A(k)$, $B(k)$, $C(k),$ and $D(k)$. We can therefore consider Eq. \eqref{abcd} as a system of linear equations for $A, B, C,$ and $D$ which can be solved for each $k$. One will find that $B(k) = 1$; this occurs as a result of symplectic geometric considerations (see Eq. \eqref{BisOne}). After this, we solve for $f_{1}(k), f_{2}(k) \in \mathbb{R}$, $k = 0, 1, \dots, q-1$ such that
\begin{equation} \label{f1} -C(k) =  \lambda_{s}f_{1}(k) - f_{1}(k+1 \mod q)   \end{equation}
\begin{equation} \label{f2} -D(k) =  \lambda_{u}f_{2}(k) - f_{2}(k+1 \mod q) \end{equation}  
which can be done using the contraction map iteration method of Section \ref{fixedPointIter}. Next, set vectors $\bold{v}_{c}(k)$ as 
\begin{equation} \label{sympconj} \bold{v}_{c}(k) = \frac{J^{-1} DK_0(\theta_k)}{ \|DK_0(\theta_k)\|^{2}} + f_{1}(k) \bold{v}_{s}(k)  + f_{2}(k) \bold{v}_{u}(k) \end{equation}
These vectors will satisfy a relation very close to column 2 of Eq. \eqref{bundleEquations}. In particular, one will have that
\begin{equation} \label{almostCol2} DF_0(X_0(k)) \bold{v}_{c}(k) = A(k) DK_0(\theta_{k+1 \mod q}) + \bold{v}_{c}(k+1 \mod q)  \end{equation}
which is proven in \ref{sympConjProof}. Now, set $ T = \frac{1}{q}\sum_{k=0}^{q-1} A(k) $. Denoting the desired column $2$ of each $P_0(k)$ as $\bold{v}_{2}(k)$, to find these $\bold{v}_{2}$, one should first find $a(k) \in \mathbb{R}$, $k = 0, 1, \dots, q-1$ satisfying
\begin{equation} \label{Tkill} -[A(k) - T]  = a(k) - a(k + 1 \mod q) \end{equation}
As the sum of the LHS over all $k$ is zero by definition of $T$, Eq. \eqref{Tkill} can be solved using the method of Section \ref{cohomSection}. The resulting $a(k)$ then yield $\bold{v}_{2}$ through $\bold{v}_{2}(k) = \bold{v}_{c}(k) + a(k)DK_0(\theta_k) $, and will satisfy
\begin{equation} \label{finalCol2} DF_0(X_0(k)) \bold{v}_{2}(k) = T DK_0(\theta_{k+1 \mod q}) + \bold{v}_{2}(k+1 \mod q)  \end{equation}
as is also proven in \ref{sympConjProof}. Setting column 2 of each $P_0(k)$ to the $\bold{v}_{2}(k)$ thus found, as well as $\lambda_2 =1$ and $ T = \frac{1}{q}\sum_{k=0}^{q-1} A(k) $ in $\Lambda_0$, one can see that Eq. \eqref{finalCol2} implies that column 2 of Eq. \eqref{bundleEquations} is satisfied as well. 

\medskip
In summary, the $DK_{0}(\theta_{k} )$, $\bold{v}_{2}(k)$, $\bold{v}_{s}(k)$, and $\bold{v}_{u}(k)$ defined in the previous discussions provide columns 1, 2, 3, and 4 respectively of each $P_{0}(k)$ for $k = 0, 1, \dots, q-1$, while the nonzero entries of $\Lambda_0(k)$---that is, $\lambda_{1}=\lambda_{2} =1$, $T$, $\lambda_s$, and $\lambda_{u}$---are also given by the above procedures. Thus, given an $F_{0}$-SPO with known points $X_0(k)$ solving Eq. \eqref{invariance}, we can get a full solution $X_0$, $P_0$, $\Lambda_0$ to Eqs. \eqref{invariance}-\eqref{bundleEquations} for $\varepsilon = 0$, with each $\Lambda_0(k)$ being of the form Eq. \eqref{lambdaForm}. If this SPO is expected to persist  for $\varepsilon > 0$, the $\varepsilon = 0$ solution can then be used to start a numerical continuation procedure to compute corresponding $F_{\varepsilon}$-solutions $X_\varepsilon$, $P_\varepsilon$, $\Lambda_\varepsilon$ as well. We describe the quasi-Newton method which enables this continuation in the following sections. 

\subsection{Summary of Steps for Quasi Newton-Method for SPO and Floquet Directions/Multipliers}

With the  initialization process for finding an $\varepsilon = 0$ solution $X_0$, $P_0$, $\Lambda_0$ for Eqs. \eqref{invariance}--\eqref{bundleEquations}  fully described, we will now develop our quasi-Newton method for solving Eqs. \eqref{invariance}--\eqref{bundleEquations} for $\varepsilon > 0$ as well. As described at the beginning of Section \ref{continuationSection}, this will enable continuation of persisting SPOs of $F_{\varepsilon =0}$ into the perturbed maps $F_{\varepsilon}$ with $\varepsilon > 0$. Before presenting the details of the method, we give a brief overview. 
Assume we have an approximate solution $(X_{\varepsilon},P_{\varepsilon},\Lambda_{\varepsilon})$ for Eq. \eqref{invariance}--\eqref{bundleEquations}. Then, we will
\begin{enumerate}
\item Compute $E(k) = F_{\varepsilon}(X_{\varepsilon}(k)) - X_{\varepsilon}(k+1 \mod q)$, $E_{red}(k) = P_{\varepsilon}^{-1}(k+1 \mod q)DF_{\varepsilon}(X_{\varepsilon}(k)) P_{\varepsilon}(k) -  \Lambda_{\varepsilon}(k)$
\item Solve $-P_{\varepsilon}^{-1}(k+1 \mod q)E(k) = \Lambda_{\varepsilon}(k)\xi(k) -   \xi(k+1 \mod q)$ for $\xi(k) \in \mathbb{C}^{4}$ using Eq. \eqref{xi1}-\eqref{xi4} and set $X_{\varepsilon}(k)$ equal to either $X_{c}(k) = X_{\varepsilon}(k) + P_{\varepsilon}(k) \xi(k)$ or $\text{Re}\, X_{c}(k)$ (details given in Section \ref{Xstep}). 
\item Recompute $DF_{\varepsilon}(X_{\varepsilon}(k))$ and $E_{red}(k)$ using the newly corrected points $X_{\varepsilon}(k)$. 
\item Solve $-E_{red}(k) = \Lambda_{\varepsilon}(k) Q(k) - Q(k + 1 \mod q) \Lambda_{\varepsilon}(k) - \Delta \Lambda(k)$ for $Q(k), \Delta \Lambda \in \mathbb{C}^{4 \times 4}$ using Eqs. \eqref{E_LC}-\eqref{E_UU}. Set $P_{c}(k) = P_{\varepsilon}(k) + P_{\varepsilon}(k) Q(k)$ and $\Lambda_{c}(k) = \Lambda_{\varepsilon}(k) + \Delta \Lambda(k)$ (details given in Section \ref{Pstep}).  
\item Using a Schur decomposition, transform the aforementioned $\Lambda_{c}(k)$ into a $\Lambda_{\varepsilon}(k)$ of the form of Eq. \eqref{lambdaForm}, modifying the $P_{c}(k)$ accordingly to get corresponding $P_{\varepsilon}(k)$  as well. 
\item Return to step 1 and repeat correction until $E$ and $E_{red}$ are within tolerance. 
\end{enumerate} 
$\varepsilon$ does not change during each quasi-Newton step. Thus, throughout the following discussion we will omit the subscripts $\varepsilon$ on $F_{\varepsilon}$, $X_{\varepsilon}$, $P_{\varepsilon}$, and $\Lambda_{\varepsilon}$ for notational convenience, denoting them as $F_{}$, $X_{}$, $P_{}$, and $\Lambda_{}$ instead.

\subsection{Quasi-Newton Step for Correcting $X$} \label{Xstep}

We seek to solve Eqs. \eqref{invariance} and \eqref{bundleEquations} for $X$, $P$, and $\Lambda$. All the entries of $\Lambda$ are equal to 0 as shown in Eq. \eqref{lambdaForm} except for $\lambda_{1}$, $\lambda_{2}$, $\lambda_s(k)$, $\lambda_u(k)$, and $T$. We will now derive an iterative step that, given an approximate solution $(X,P,\Lambda)$ of Eqs. \eqref{invariance} and \eqref{bundleEquations}, produces a much more accurate one. Define the errors
\begin{equation} \label{Edef} E(k) = F(X(k)) - X(k+1 \mod q) \end{equation}
\begin{equation} \label{Ereddef} E_{red}(k) = P^{-1}(k+1 \mod q)DF(X(k)) P(k) -  \Lambda(k) \end{equation} 
We then need to find corrections $\Delta X$, $\Delta P$, and $\Delta \Lambda$ to cancel $E$ and $E_{red}$. We start with $\Delta X$; write $\Delta X(k) = P(k) \xi(k)$. We will solve for $\xi(k) \in \mathbb{C}^{4}$, $k = 0, 1, \dots, q-1$ satisfying
	\begin{equation} \label{xiEquation} \eta(k) \stackrel{\text{def}}{=} -P^{-1}(k+1 \mod q)E(k) = \Lambda(k)\xi(k) -   \xi(k+1 \mod q) \end{equation}  
\begin{claim*}
For $E$ and $E_{red}$ sufficiently small, if the $\xi(k)$ solve Eq. \eqref{xiEquation}, then adding $\Delta X(k) = P(k) \xi(k)$ to each $X(k)$ reduces the error $E$ quadratically.
\end{claim*}

\begin{proof}
Substitute $X(k)+\Delta X(k)$ into the RHS of Eq. \eqref{Edef}. Assuming that $\Delta X$ is small enough (true for $E$ sufficiently small), we can expand Eq. \eqref{Edef} in Taylor series to get 
\begin{align} \label{deriveDeltaX} \begin{split}
E_{new}&(k) = F(X(k) + \Delta X(k)) - [X(k+1 \mod q) + \Delta X(k+1 \mod q)] \\
=&F(X(k)) +DF(X(k))\Delta X(k) + \mathcal O(\Delta X(k)^{2}) - [X(k+1 \mod q) + \Delta X(k+1 \mod q)] \\
=&E(k) +DF(X(k)) \Delta X(k) - \Delta X(k+1 \mod q) + \mathcal O(\Delta X(k)^{2})\\ 
\end{split} \end{align}  
$\Delta X(k) = P(k) \xi(k)$, and Eq.  \eqref{Ereddef} implies $DF(X(k)) P(k) = P(k+1 \mod q) \left[ \Lambda(k) + E_{red}(k) \right]$. Thus,
\begin{align} \begin{split} \label{xiWithQuadratic} E_{new}&(k) = E(k) + DF(X(k)) P(k) \xi(k) - P(k+1 \mod q)  \xi(k+1 \mod q) + \mathcal O(\xi(k)^{2}) \\
& =E(k) + P(k+1 \mod q) \left[ \Lambda(k)\xi(k) + E_{red}(k)\xi(k) - \xi(k+1 \mod q)  \right]  + \mathcal O(\xi(k)^{2}) \\
&=  P(k+1 \mod q) E_{red}(k)\xi(k) + \mathcal O(\xi(k)^{2})  \\
\end{split} \end{align} 
where the last line follows from Eq. \eqref{xiEquation}. The $\xi(k)$ solving Eq. \eqref{xiEquation} will be similar in magnitude to the $E(k)$, so $E_{red}(k) \xi(k)$ will be quadratically small, comparable to $E_{red}(k)E(k)$. Hence, as long as the $E(k)$ (and hence the $\xi(k)$ and $\Delta X(k)$) are small enough that the Taylor expansion in Eq. \eqref{deriveDeltaX} is valid, and the $\mathcal  O(\xi^{2})$ terms of the Taylor expansion are small, the new errors $E_{new}$ will be quadratically smaller than $E$. 
\end{proof}

To solve Eq. \eqref{xiEquation}, let $\xi(k) = \begin{bmatrix} \xi_{1}(k) & \xi_{2}(k)  & \xi_{3}(k) & \xi_{4}(k) \end{bmatrix}^{T}$ and $\eta(k) = \begin{bmatrix} \eta_{1}(k) & \eta_{2}(k)  & \eta_{3}(k) & \eta_{4}(k) \end{bmatrix}^{T}$ for all $k = 0, 1, \dots, q-1$. As each $\Lambda(k)$ is nearly diagonal, we can write Eq. \eqref{xiEquation}  component-wise as 
		\begin{gather}  \label{xi1} \eta_{1}(k)-T \xi_{2}(k)= \lambda_{1} \xi_{1}(k) -   \xi_{1}(k+1 \mod q) \\
	  \label{xi2} \eta_{2}(k) = \lambda_{2} \xi_{2}(k) -   \xi_{2}(k+1 \mod q) \\
	  \label{xi3} \eta_{3}(k) = \lambda_{s}(k)\xi_{3}(k) -   \xi_{3}(k+1 \mod q) \\
	  \label{xi4} \eta_{4}(k) = \lambda_{u}(k)\xi_{4}(k) -   \xi_{4}(k+1 \mod q) 
	  \end{gather}  
	  
Eqs. \eqref{xi1}--\eqref{xi4} are all of the form whose solution was discussed in Section \ref{prelimSection}. Eqs. \eqref{xi3}--\eqref{xi4} for $\xi_{3}$ and $\xi_{4}$ admit a straightforward solution using the method of Section \ref{fixedPointIter}. As for $\xi_{1}$ and $\xi_{2}$, one first solves Eq. \eqref{xi2} for the $\xi_{2}$, and then uses this to evaluate the LHS of Eq. \eqref{xi1} and solve it for $\xi_{1}$. 
Except for during the first quasi-Newton step after increasing $\varepsilon$ from $\varepsilon=0$ (for which $\lambda_{1}=\lambda_{2}=1$, as discussed in Section \ref{initCols12}), generally both $\lambda_{1}$ and $\lambda_{2}$ will be different from 1---thus facilitating the solution of Eqs. \eqref{xi1}-\eqref{xi2} using the methods of Sections \ref{fixedPointIter}-\ref{cohomHypFormulaSection}. In this first quasi-Newton step after $\varepsilon=0$, on the other hand, one should apply the procedure of Section \ref{cohomSection} to Eqs. \eqref{xi1}--\eqref{xi2} even if the sum $\sum_{k=0}^{q-1} \eta_{2}(k)$ is not zero; later quasi-Newton steps will be able to correct the errors further. In such (rare) cases of $\lambda_{1}=\lambda_{2}=1$, after solving Eq. \eqref{xi2} for a preliminary solution $\bar \xi_2$ using the method of Section \ref{cohomSection}, the sum of the LHS of Eq. \eqref{xi1} can be made zero before solving for $\xi_{1}$ by setting each final $\xi_2(k) = \bar\xi_2(k) + \frac{1}{qT}\sum_{k=0}^{q-1} [\eta_{1}(k)-T\bar \xi_2(k)]$; recall from Section \ref{cohomSection} that if $\bar\xi_2(k)$ are a solution to Eq. \eqref{xi2}, then so are $\bar\xi_2(k)+C$ for any $C \in \mathbb{C}$.

 Finally, with all four components of $\xi$ solved, we compute $X_{c}(k) = X(k) + P(k) \xi(k)$. Recall that $P(k) \in \mathbb{C}^{4 \times 4}$ and $\xi(k) \in \mathbb{C}^{4}$ may have non-real entries; hence, it can be that $X_{c}(k) \in \mathbb{C}^{4}$. While one could allow $X(k)$ to take values in $\mathbb{C}^{4}$ and simply set each corrected $X(k)$ equal to $X_c(k)$, in practice we have found that setting each corrected $X(k)$ equal to $\text{Re} \, X_{c}(k)$ works as well, as the imaginary parts $\text{Im} \,  X_{c}(k)$ are usually very small; this allows the $X(k)$ to remain in $\mathbb{R}^{4}$, which may make the evaluation of $F(X(k))$ and $DF(X(k))$ simpler.  Either way, this concludes the $X$ correction part of the quasi-Newton step. 
 
\begin{remark}
There are methods of numerically solving for $\Delta X$ without using $P$ or $\Lambda$, such as standard multi-shooting methods \citep{Guckenheimer2007}. For our 4D phase space, these methods involve solving Eq. \eqref{deriveDeltaX} (with quadratic terms dropped, and $E_{new}=0$) simultaneously for all $\Delta X(k)$. This requires solving a $4q \times 4q$ linear system at each correction step. 
However, by using $P$ and the nearly diagonal $\Lambda$, we decouple the equations in a manner that lends itself to non-matrix-based solution methods, and avoid this large dimensional system. 
\end{remark}

\subsection{Quasi-Newton Step for Correcting $P$ and $\Lambda$} \label{Pstep}

Using the newly-corrected $X(k)$, we first recompute $DF(X(k))$ and then $E_{red}(k)$ for $k = 0, 1, \dots, q-1$ using Eq. \eqref{Ereddef}. Finding $\Delta P(k)$ and $\Delta \Lambda(k)$ to cancel $E_{red}$ then follows a similar methodology as $\Delta X$. Let $\Delta P (k) = P(k) Q(k)$; we will now solve for $Q(k), \Delta \Lambda(k) \in \mathbb{C}^{4 \times 4}$ satisfying
\begin{equation}  \label{qEquation} -E_{red}(k) = \Lambda(k) Q(k) - Q(k + 1 \mod q) \Lambda(k) - \Delta \Lambda(k)\end{equation}  
\begin{claim*}
For $E_{red}$ sufficiently small, if the $Q(k)$ and $\Delta \Lambda(k)$ solve Eq. \eqref{qEquation} for all $k = 0, 1, \dots, q-1$, then adding $\Delta P(k) = P(k) Q(k)$ and $\Delta \Lambda(k)$ respectively to each $P(k)$ and $\Lambda(k)$ reduces $E_{red}$ quadratically.
\end{claim*}
\begin{proof}
Substitute $P(k)+P(k)Q(k)$ and $\Lambda(k)+\Delta \Lambda(k)$ into Eq. \eqref{bundleEquations} to define
\begin{align} \begin{split} 
\mathcal{E}(k)=DF(K(&\theta)) [P(k)+P(k)Q(k)] \\
&- [P(k+1 \mod q)+P(k+1 \mod q)Q(k+1 \mod q)] [\Lambda(k)+\Delta \Lambda(k)] \\
\end{split} \end{align}  
Using $E_{red}(k) = P^{-1}(k+1 \mod q)DF(X(k)) P(k) -  \Lambda(k)$, we then find that
\begin{align} \begin{split} \label{bundleDerive}
P(k+&1  \mod q)^{-1} \mathcal{E}(k) =E_{red}(k)+P(k+1 \mod q)^{-1}DF(X(k))P(k)Q(k) \\
&\quad \quad \quad \quad \quad \quad \, \, \, \quad \quad \quad \quad \quad \quad \quad \quad - Q(k+1 \mod q) [\Lambda(k)+\Delta \Lambda(k)] - \Delta \Lambda(k) \\
  &=E_{red}(k)+[\Lambda(k)+E_{red}(k) ]Q(k) - Q(k+1 \mod q) [\Lambda(k)+\Delta \Lambda(k)] - \Delta \Lambda(k) \\
    &=E_{red}(k) Q(k) - Q(k+1 \mod q) \Delta \Lambda(k) 
\end{split} \end{align}  
where the last line follows from the one before it due to Eq. \eqref{qEquation}. Evaluating Eq. \eqref{Ereddef} with $P(k)+P(k)Q(k)$ and $\Lambda(k)+\Delta \Lambda(k)$ in place of $P(k)$ and $\Lambda(k)$ and denoting the result as $E_{red,new}$, we have 
\begin{align} \begin{split} E_{red,new}(k)&=[P(k+1 \mod q)+P(k+1 \mod q)Q(k+1 \mod q)]^{-1}\mathcal{E}(k) \\
&=[I+Q(k+1 \mod q)]^{-1}P(k+1 \mod q)^{-1}\mathcal{E}(k) \\
&=[I+Q(k+1 \mod q)]^{-1} [E_{red}(k) Q(k) - Q(k+1 \mod q) \Delta \Lambda(k)]  \\
\end{split} \end{align}  
$Q$ and $\Delta \Lambda$ here will be similar in magnitude to $E_{red}$. Hence, if $E_{red}$ is small, then $E_{red,new}$ will be quadratically smaller of similar order as $E_{red}^{2}$.  
\end{proof}

Since each $\Lambda(k)$ is nearly diagonal, the equations for the different entries of $Q(k)$ and $\Delta \Lambda(k)$ following from Eq. \eqref{qEquation} are almost completely decoupled from each other. Write 
\begin{equation} \begin{gathered} \label{Ecomponents} E_{red}(k) = \begin{bmatrix}
E_{LL}(k) &  E_{LC}(k)    & E_{LS}(k)  & E_{LU}(k)  \\ E_{CL}(k) &  E_{CC}(k)    & E_{CS}(k)  & E_{CU}(k) \\ E_{SL}(k) & E_{SC}(k)    & E_{SS}(k)  & E_{SU}(k) \\ E_{UL}(k) &  E_{UC}(k)    & E_{US}(k)  & E_{UU}(k) \end{bmatrix} \\ 
Q(k) = \begin{bmatrix}
Q_{LL}(k)&  Q_{LC}(k)    & Q_{LS}(k)  & Q_{LU}(k)  \\ Q_{CL}(k) &  Q_{CC} (k)   & Q_{CS}(k)  & Q_{CU}(k) \\ Q_{SL}(k) & Q_{SC} (k)   & Q_{SS} (k) & Q_{SU}(k) \\ Q_{UL}(k) &  Q_{UC}(k)    & Q_{US}(k)  & Q_{UU}(k) \end{bmatrix} \quad \Delta \Lambda(k) = \begin{bmatrix}
\Delta \lambda_1 &  \Delta T   & 0 & 0 \\ \Delta S &  \Delta \lambda_2   & 0 & 0 \\ 0 & 0  & \Delta \lambda_s(k) & 0 \\ 0 &  0 & 0 & \Delta \lambda_u(k) \end{bmatrix} \end{gathered} \end{equation}
Note that unlike the case for tori, the first column of $Q$ and (1,1), (2,1), and (2,2) entries of $\Delta \Lambda$ are nonzero here. We can then write Eq. \eqref{qEquation} entry by entry and rearrange terms slightly to get 16 scalar equations 
\begin{equation}  
\label{E_LL} \Delta \lambda_1 -E_{LL}(k) -T Q_{CL}(k)= \lambda_1 Q_{LL}(k) - \lambda_1 Q_{LL}(k+1 \mod q) 
\end{equation}  
\begin{equation}  
\label{E_LC} \Delta T -E_{LC}(k)-T Q_{CC}(k) + T Q_{LL}(k+1 \mod q) = \lambda_1 Q_{LC}(k) - \lambda_2 Q_{LC}(k+1 \mod q)  
\end{equation}  
\begin{equation}  
 \label{E_LS} -E_{LS}(k) -T Q_{CS}(k)= \lambda_1 Q_{LS}(k) - \lambda_{s}(k) Q_{LS}(k+1 \mod q) 
\end{equation}  
\begin{equation}  
 \label{E_LU} -E_{LU}(k) -T Q_{CU}(k)= \lambda_1 Q_{LU}(k) - \lambda_{u}(k) Q_{LU}(k+1 \mod q)  
\end{equation}  
\begin{equation}  
\label{E_CL} \Delta S -E_{CL}(k) = \lambda_2 Q_{CL}(k) - \lambda_1 Q_{CL}(k+1 \mod q) 
\end{equation}  
\begin{equation}  
 \label{E_CC} \Delta \lambda_2 -E_{CC}(k) + T Q_{CL}(k+1 \mod q) = \lambda_2 Q_{CC}(k) - \lambda_2 Q_{CC}(k+1 \mod q)  
\end{equation}  
\begin{equation}  
 \label{E_CS} -E_{CS}(k) = \lambda_2 Q_{CS}(k) - \lambda_{s}(k) Q_{CS}(k+1 \mod q) 
\end{equation}  
\begin{equation}  
 \label{E_CU} -E_{CU}(k) = \lambda_2 Q_{CU}(k) - \lambda_{u}(k) Q_{CU}(k+1 \mod q)  
\end{equation}  
\begin{equation}  
\label{E_SL} -E_{SL}(k) = \lambda_{s}(k) Q_{SL}(k) - \lambda_1 Q_{SL}(k+1 \mod q) 
\end{equation}  
\begin{equation}  
  \label{E_SC} -E_{SC}(k) + T Q_{SL}(k+1 \mod q) = \lambda_{s}(k) Q_{SC}(k) - \lambda_2 Q_{SC}(k+1 \mod q) 
\end{equation}  
\begin{equation}  
 \label{E_SS} \Delta \lambda_{s}(k) -E_{SS}(k) = \lambda_{s}(k) Q_{SS}(k) - \lambda_{s}(k) Q_{SS}(k+1 \mod q)  
\end{equation}  
\begin{equation}  
 \label{E_SU} -E_{SU}(k) = \lambda_{s}(k) Q_{SU}(k) - \lambda_{u}(k) Q_{SU}(k+1 \mod q)  
\end{equation}  
\begin{equation}  
\label{E_UL} -E_{UL}(k) = \lambda_{u}(k) Q_{UL}(k) - \lambda_1 Q_{UL}(k+1 \mod q) 
\end{equation}  
\begin{equation}  
  \label{E_UC} -E_{UC}(k) + T Q_{UL}(k+1 \mod q) = \lambda_{u}(k) Q_{UC}(k) - \lambda_2 Q_{UC}(k+1 \mod q) 
\end{equation}  
\begin{equation}  
 \label{E_US} -E_{US}(k) = \lambda_{u}(k) Q_{US}(k) - \lambda_{s}(k) Q_{US}(k+1 \mod q) 
\end{equation}  
\begin{equation}  
 \label{E_UU} \Delta \lambda_{u}(k) -E_{UU}(k) = \lambda_{u}(k) Q_{UU}(k) - \lambda_{u}(k) Q_{UU}(k+1 \mod q) 
\end{equation}  
First of all, we solve Eqs.  \eqref{E_CS}, \eqref{E_CU}, \eqref{E_SL}, \eqref{E_SU},  \eqref{E_UL}, and \eqref{E_US} using the method of Section \ref{fixedPointIter}; the resulting $Q_{CS}$, $Q_{CU}$, $Q_{SL}$ and $Q_{UL}$ then enable the solution of Eqs. \eqref{E_LS}, \eqref{E_LU}, \eqref{E_SC} and \eqref{E_UC} by the same method as well. The solutions of Eqs. \eqref{E_SS} and \eqref{E_UU} are non-unique; we simply choose the $Q_{SS}(k)=Q_{UU}(k)=0$ with $\Delta \lambda_{s}(k)=E_{SS}(k)$ and $\Delta \lambda_{u}(k)=E_{UU}(k)$ for all $k = 0, 1, \dots, q-1$ (see Remark \ref{keepLambdaConstant} at the end of this subsection for another potential solution). Finally, we are left with Eqs. \eqref{E_LL}, \eqref{E_LC}, \eqref{E_CL}, and \eqref{E_CC}. 

We start with Eq. \eqref{E_CL}. Note that if $\lambda_{1}, \lambda_{2} \neq 1$, then Eq. \eqref{E_CL} is underdetermined; one can choose any value of $\Delta S$ and then find a valid solution for the $Q_{CL}(k)$ as well. However, there is a choice of $\Delta S$ that is best for numerical stability. To see this, sum both sides of Eq. \eqref{E_CL} over all $k$ and rearrange terms to get
\begin{equation} \label{Q_CL_average}   \frac{1}{q} \sum_{k=0}^{q-1} Q_{CL}(k) = \frac{1}{\lambda_2- \lambda_1} \left[\Delta S - \frac{1}{q} \sum_{k=0}^{q-1}  E_{CL}(k) \right] \end{equation}
Now, recall that in the $\varepsilon=0$ case, $\lambda_1=\lambda_2=1$. While this is almost never the case for $\varepsilon > 0$ (except for during the first quasi-Newton step after increasing $\varepsilon$ from 0), generally $\lambda_1$ and $\lambda_2$ do remain \emph{near} 1. Thus, the $\frac{1}{\lambda_2 - \lambda_1}$ term above will be quite large in magnitude, making the average of $Q_{CL}$ large as well unless the RHS is zero. Thus, to avoid large numerical values of $Q_{CL}$ (and $\Delta P$), one should set $\Delta S = \frac{1}{q} \sum_{k=0}^{q-1}  E_{CL}(k)$, and then solve for $Q_{CL}$ by the method of Section \ref{fixedPointIter} or \ref{cohomHypFormulaSection}. This choice of $\Delta S$ also enables the solution of Eq. \eqref{E_CL} for $Q_{CL}$ if $\lambda_1=\lambda_2=1$, in which case the method of Section \ref{cohomSection} is used instead. 

Now, the resulting $Q_{CL}(k)$ can be back-substituted into Eqs. \eqref{E_LL} and \eqref{E_CC}. After dividing them through by $\lambda_1$ and $\lambda_2$, respectively, both equations have the form of Eq. \eqref{cohomological}. Thus, the sum of the LHS of both across all $k$ must be made zero; for this, set $\Delta \lambda_1 = \frac{1}{q} \sum_{k=0}^{q-1} [E_{LL}(k) +T Q_{CL}(k)]$ and $\Delta \lambda_2 = \frac{1}{q} \sum_{k=0}^{q-1} [E_{CC}(k) - T Q_{CL}(k+1 \mod q)]$. Then, apply the method of Section \ref{cohomSection} to solve Eqs. \eqref{E_LL} and \eqref{E_CC} for $Q_{LL}$ and $Q_{CC}$. These in turn should be back-substituted into Eq. \eqref{E_LC}, in which one should set $\Delta T = \frac{1}{q} \sum_{k=0}^{q-1} [E_{LC}(k)+T Q_{CC}(k) - T Q_{LL}(k+1 \mod q)]$ for similar reasons as the $\Delta S$ case of Eq. \eqref{E_CL}. This then allows one to solve Eq. \eqref{E_LC} for $Q_{LC}$ using the methods of Section \ref{fixedPointIter} or \ref{cohomHypFormulaSection} (if $\lambda_{1}, \lambda_2 \neq 1$) or Section \ref{cohomSection} (in the rare $\lambda_1 = \lambda_2 = 1$ case). This completes the solution of $Q(k)$ and $\Lambda(k)$ for all $k$. 

Once $Q$ and $\Delta \Lambda$ are known, we set $P_c(k) = P(k) + P(k) Q(k)$ and $\Lambda_c(k) = \Lambda(k) + \Delta \Lambda(k)$. While these $P_c$ and $\Lambda_c$ should quadratically reduce the error $E_{red}$ when substituted into Eq. \eqref{Ereddef}, as desired, note that the (2,1) entry $\Delta S$ of $\Delta \Lambda$ was nonzero. Hence, even if all $\Lambda(k)$ were in the desired form  of Eq. \eqref{lambdaForm}, the corrected $\Lambda_c(k)$ generally will not be. We thus need an extra step to get a new $\Lambda(k)$ of the required form.

\begin{remark} \label{keepLambdaConstant}
If the $\lambda_{s}(k)$ and $\lambda_{u}(k)$ are constant (independent of $k$), we can choose the non-unique solutions of Eqs. \eqref{E_SS} and \eqref{E_UU} such that they remain constant. In particular, choose $\Delta \lambda_s(k) = \frac{1}{q}\sum_{i=0}^{q-1} E_{SS}(i)$ and $\Delta \lambda_u(k) = \frac{1}{q}\sum_{i=0}^{q-1} E_{UU}(i)$ across all $k$, and solve for $Q_{SS}$ and $Q_{UU}$ using the method of Section \eqref{cohomSection} (after dividing through by $\lambda_s$ or $\lambda_{u}$). Our experience was that this choice of solution negatively affected the numerical stability of our method, however; thus, we did not keep $\lambda_{s}$ and $\lambda_{u}$ constant in our implementation.
\end{remark}

\subsubsection{Transforming $\Lambda_c(k)$ to find $\Lambda(k)$ } \label{schurSection}

To get new $P(k)$ and $\Lambda(k)$ which quadratically reduce the error $E_{red}$ to a level similar to $P_c$ and $\Lambda_c$, but with the $\Lambda(k)$ still having the form of Eq. \eqref{lambdaForm}, one can leverage the (complex) \emph{Schur decomposition} \cite{Horn_Johnson_1985}: given an arbitrary square matrix $A \in \mathbb{C}^{n \times n}$, $A$ can be expressed as $A = V U V^{-1}$ for some upper triangular matrix $U \in \mathbb{C}^{n \times n }$ and some unitary matrix $V \in \mathbb{C}^{n \times n}$. MATLAB, Julia, and many other programming languages have functions built-in or part of well-known libraries to compute this matrix decomposition. 

To apply the Schur decomposition to our case, first note that each $\Lambda_c(k)$ has the block-diagonal form
\begin{equation} 
\Lambda_c(k) = \begin{bmatrix}
  \lambda_{1,c} &  T_{c} & 0 & 0 \\
\Delta S & \lambda_{2,c} & 0 & 0 \\
  0 & 0 & \lambda_{s,c}(k) & 0 \\
  0 & 0 & 0 & \lambda_{u,c}(k)
\end{bmatrix} = \left[ \begin{array}{c|c}
  A_{2 \times 2} & \mathbf{0}_{2 \times 2} \\
  \hline
  \mathbf{0}_{2 \times 2} & B_{2\times2}
\end{array} \right]
\end{equation}
with $A$ and $B$ being defined as the top left and bottom right $2\times 2$ blocks of $\Lambda_c$ as shown. Now, find a Schur decomposition of the $2\times2$ block $A = V_{1} U V_{1}^{-1}$ so $U \in \mathbb{C}^{2\times2}$ is upper triangular. Then, define
\begin{equation} 
V = \left[ \begin{array}{c|c}
  V_{1} & \mathbf{0}_{2 \times 2} \\
  \hline
  \mathbf{0}_{2 \times 2} & I_{2\times2}
\end{array} \right]
\end{equation}
Finally, set each new $P(k) = P_{c}(k) V$ and $\Lambda(k) = V^{-1} \Lambda_{c}(k) V$. The resulting $\Lambda(k)$ will all have top-left block $V_{1}^{-1} A V_{1} = U$, which is upper triangular as required by Eq. \eqref{lambdaForm}, while the rest of the matrix will remain identical to $\Lambda_c$ and continue to match Eq. \eqref{lambdaForm}. Moreover, with the new $P$ and $\Lambda$, the new $E_{red}(k)$ will be
\begin{equation} \label{newEred}
P^{-1}(k+1 \mod q)DF(X(k)) P(k) -  \Lambda(k) = V^{-1} \left[ P_{c}^{-1}(k+1 \mod q)DF(X(k)) P_{c}(k) -  \Lambda_{c}(k)  \right] V
\end{equation}
so that if the expression in brackets on the RHS of Eq. \eqref{newEred} is quadratically smaller than the old $E_{red}(k)$---as was indeed achieved in Section \ref{Pstep}---the new $E_{red}$ error will also be quadratically smaller than the old one. This hence concludes the quasi-Newton correction step for $P(k)$ and $\Lambda(k)$. 

\subsection{After the Quasi-Newton Step} \label{afterStep}

After completing a  quasi-Newton correction step for $P$ and $\Lambda$ as described in Section \ref{Pstep}, one should recompute $E(k)$ and $E_{red}(k)$ using the new $X(k)$, $P(k)$, $\Lambda(k)$ and Eqs. \eqref{Edef}--\eqref{Ereddef}. If these errors are not yet within the desired tolerance, one should go back to the quasi-Newton step for correcting the torus parameterization $X(k)$ from Section \ref{Xstep} and repeat the entire method until the $E(k)$ and $E_{red}(k)$ are within tolerance. In practice, we use the supremum norm $\| E \| = \max_{k = 0, 1, \dots, q-1} \|E(k)\|_{\infty}$ and $\| E_{red} \| = \max_{k = 0, 1, \dots, q-1} \| E_{red} (k)\|_{\infty}$ to measure the size of these errors, which we have found to work well. 

Once the quasi-Newton method has converged to a solution $X_{\varepsilon}, P_{\varepsilon}, \Lambda_{\varepsilon}$ for an intermediate continuation parameter $\varepsilon = \varepsilon_{i} < \varepsilon_{f}$, for numerical stability purposes, we have found that it is beneficial  to apply the procedure of Section \ref{constantLambda} to construct a new solution $X_{\varepsilon}$, $\tilde P_{\varepsilon}$, $\tilde \Lambda_{\varepsilon}$ in which the stable and unstable multipliers contained in $\tilde \Lambda_{\varepsilon}$ are independent of $k$. The resulting $\tilde \Lambda_{\varepsilon}(k)$ will be independent of $k$ as well, which will ensure better numerical behavior in the following continuation step. 

Finally, if we have found a solution $X_{\varepsilon_{f}}, P_{\varepsilon_{f}}, \Lambda_{\varepsilon_{f}}$ to Eqs. \eqref{invariance}-\eqref{bundleEquations} (with $\lambda_1, \lambda_2 \neq 1$ and $\Lambda_{\varepsilon_{f}}$ independent of $k$) for the final desired value of $\varepsilon = \varepsilon_{f}$, we can now find a solution $X_{\varepsilon_{f}}, \bar P_{\varepsilon_{f}}, \bar \Lambda_{\varepsilon_{f}}$ to Eq. \eqref{floquetEquation} with $ \bar \Lambda_{\varepsilon_{f}}$ fully diagonal as well. For this, first find the matrix $V_{D}$ that diagonalizes $\Lambda_{\varepsilon_{f}} = V_{D} D V_{D}^{-1}$ with $D$ diagonal. Then, simply set all $\bar P_{\varepsilon_{f}}(k)  = P_{\varepsilon_{f}}(k) V_{D}$ and $\bar \Lambda_{\varepsilon_{f}} = V_{D}^{-1}  \Lambda_{\varepsilon_{f}} V_{D} = D$; as described at the end of Section \ref{quasiNewton}, the resulting $X_{\varepsilon_{f}}, \bar P_{\varepsilon_{f}}$, and $\bar \Lambda_{\varepsilon_{f}}$ will satisfy Eq. \eqref{bundleEquations} and thus Eq. \eqref{floquetEquation} as well. This fully-diagonal $\bar \Lambda_{\varepsilon_{f}}$ and corresponding $\bar P_{\varepsilon_{f}}$ are useful for computing stable/unstable manifolds (and their invariant submanifolds) for the SPO $X_{\varepsilon_{f}}$. An example of this will be described in Section \ref{parambigsection} for separatrices of certain SPOs. 

\subsection{Some Final Remarks} \label{remarksSection}

In similar quasi-Newton methods for computing invariant tori (e.g., \cite{kumar2022}), the (1,1), (2,1), and (2,2) terms of $\Delta \Lambda$ are always set to zero, so that the corresponding entries of $\Lambda$ remain as 1, 0, and 1 respectively throughout the entire correction procedure. This is made possible by an equation similar to Eq. \eqref{col1Lambda} coupled with symplectic geometric considerations akin to those of Section \ref{initCols12} for the $\varepsilon=0$ case. However, for the subharmonic periodic orbits of this paper, the resonant $F_0$-invariant tori from which they emerge generically break down for all $\varepsilon >0$. Instead, the SPO gets non-unity Floquet multipliers $\lambda_1$ and $\lambda_2$, whose computation requires nonzero $\Delta \lambda_1$ and $\Delta \lambda_2$; these non-unity multipliers ameliorate the zero-denominators problem faced by torus parameterization methods for rational $\omega$. Nevertheless, to better handle the computations when $\lambda_1$ and $\lambda_2$ are \emph{near} 1---which is usually the case, and occurs even more pronouncedly whenever the SPO's $\lambda_1$-$\lambda_2$ Floquet multiplier pair is going through a stability transition---we set $\Delta S  \neq 0 $ and $\Delta T \neq 0$ as well. In fact, if an SPO is transitioning from a real multiplier pair $\lambda_1, \lambda_2 \in \mathbb{R}$ at the previous continuation $\varepsilon$ value to a complex elliptic pair at the current $\varepsilon$, one \emph{must} take $\Delta S \neq 0$ to capture this behavior and converge. 

To see why, first note that Eq. \eqref{cohomHyp} and \eqref{cohomological} both admit only real solutions $u(k)$ if $\lambda_a(k), \lambda_b(k), b(k) \in \mathbb{R}$ for all $k$ (barring the artificial addition of a non-real constant to a solution of Eq. \eqref{cohomological}). Now, suppose that at some stage of a continuation procedure, all $P(k)$ have all real entries, and $\lambda_1$, $\lambda_2$, $T$, $\lambda_s(k)$, and $\lambda_{u}(k)$ all have real nonzero values; this is always the case, for example,  during the first quasi-Newton step after $\varepsilon = 0$ when $\lambda_1 = \lambda_2 = 1$. Then, if one applies the procedure of Section \ref{Pstep} to Eqs. \eqref{E_LL}-\eqref{E_UU}---with the exception of taking $\Delta S = 0$---the solutions will all be real. Then, all $P_c(k) = P(k)+P(k)Q(k)$ and $\Lambda_c(k) = \Lambda(k) + \Delta \Lambda(k)$ will also have real entries---and since $\Delta S=0$, the new $\Lambda_c(k)$ will already be of the form Eq. \eqref{lambdaForm}, precluding the need for the Schur procedure of Section \ref{schurSection}. Thus, $P_c$ and $\Lambda_c$ themselves will form the new $P$ and $\Lambda$ used in the next step, with all real entries---as was the case at the beginning of the previous step. However, this means that the multipliers $\lambda_1$ and $\lambda_2$ will be real no matter how many quasi-Newton iterations are applied---which will prevent convergence if the SPO at hand in fact has elliptic multipliers. Thus, it is imperative to take $\Delta S \neq 0$, which makes the top-left block of $\Lambda_c(k)$ non-triangular and forces the need for the Schur procedure---which can yield a complex $\Lambda(k)$ from such a real $\Lambda_c(k) $. 

The quasi-Newton method just presented also works in the non-generic case that the resonant $F_0$-invariant torus containing the SPO does not break down for $\varepsilon > 0$; this can happen, for example, if the maps $F_{\varepsilon}$ are all flow maps of a 2 DOF autonomous Hamiltonian system, rather than a 2.5 DOF perturbation, so that the unstable flow-periodic orbits for $\varepsilon = 0$ persist as flow-periodic orbits for $\varepsilon > 0$ as well. In such a case, the SPO will lie on an $F_{\varepsilon}$-invariant torus even for $\varepsilon > 0$; nevertheless, the quasi-Newton method has been found to work in such settings as well with only a minor change. Note that such an SPO will have $\lambda_1 = \lambda_2 = 1$ for all $\varepsilon$, due to an equation exactly analogous to Eq. \eqref{col1Lambda}. Thus, to compute SPOs in such systems, the main change to the quasi-Newton method is that one should simply take $\Delta \lambda_1 = \Delta \lambda_2 = \Delta S = 0$ throughout the procedure, so that $\lambda_1$ and $ \lambda_2$ will remain at 1. It will be necessary to use the method of Section \ref{cohomSection} more frequently, and it will not always be possible to force the sum of the RHS of Eq. \eqref{cohomological} to zero. Nevertheless, the method converges in practice, likely due to similar vanishing lemmas as those of \cite{fontichLlaveSire, kumar2022} for tori.

Finally, we note that while convergence of this new quasi-Newton method is not yet rigorously proven, it has been used successfully for numerical continuation of subharmonic periodic orbits in a number of cases, some of which are presented in Section \ref{numResults} for various examples from the CCR4BP. We found that the method worked in many cases where the stability of the $\lambda_1$-$\lambda_2$ multiplier pair changed from elliptic to hyperbolic or vice versa as $\varepsilon$ changed. The method did fail in some other cases as those multipliers (or their $q$\textsuperscript{th} powers) approached 1 during continuation by $\varepsilon$; further investigation is required to determine whether this is due to a failure of the method to ``cross'' a stability transition, or if this indicates a fold bifurcation of the SPOs being studied such that no corresponding SPO exists past a certain $\varepsilon$ value. While the method presented is not yet designed to handle such bifurcations where $\varepsilon$ takes a maximum along the solution family, one could potentially use a ``hybrid'' algorithm to continue orbits past such points, using standard multi-shooting to solve for $X_{\varepsilon}(k)$ alongside the quasi-Newton step of Section \ref{Pstep} to compute $P_{\varepsilon}$ and $\Lambda_{\varepsilon}$.

\section{Parameterization Method for Separatrices} \label{parambigsection}

Suppose that for some desired value of $\varepsilon = \varepsilon_{f}$, as described in Section \ref{problemStateSection} (see also the end of Section \ref{afterStep}), we have a solution $X(k)$, $P(k)$, $\Lambda$ to Eqs. \eqref{invarianceEquation}--\eqref{floquetEquation} for $k = 0, 1, \dots, q-1$ with $\Lambda \in \mathbb{C}^{4 \times 4}$ diagonal and independent of $k$ (we again will omit subscripts $\varepsilon$ in this section). Let the diagonal entries of $\Lambda$ be $\lambda_1$, $\lambda_2$, $\lambda_s$, and $\lambda_u$; the $q$\textsuperscript{th} powers of these entries will be the eigenvalues of the relevant SPO's monodromy matrix  $DF^{q}(X(k))$ with columns of $P(k)$ being their corresponding eigenvectors, as also explained in Section \ref{problemStateSection}. 

Now, in accordance with the conventions established in Section \ref{quasiNewton}, let $\lambda_s$ and $\lambda_u$ represent stable and unstable SPO Floquet multipliers inherited and continued from the $\varepsilon=0$ SPO $X_0$. As in the $\varepsilon = 0$ case, the columns of $P(k)$ corresponding to $\lambda_s$ and $\lambda_u$ will be transverse to the 2D cylindrical NHIM $\Xi_{\varepsilon}$ containing the SPO. However, $\lambda_1$ and $\lambda_2$ will then be multipliers corresponding to the linearized dynamics of $F$ \emph{restricted to} $\Xi_{\varepsilon}$. Thus, to understand the effect of the SPO $X(k)$ on these internal NHIM dynamics, one should study the effect of these Floquet multipliers. In particular, if $\lambda_1$ and $\lambda_2$ are real and hyperbolic (without loss of generality, say $0 < \lambda_1 <1$ and $\lambda_{2} = \lambda_{1}^{-1}$), then their corresponding Floquet directions will give rise to 1D stable/unstable manifolds of the SPO \emph{inside} $\Xi_{\varepsilon}$ for the dynamics restricted to this NHIM. We will refer to these 1D stable/unstable manifolds inside $\Xi_{\varepsilon}$ as \emph{separatrices}. Intersections of separatrices between different SPOs destroy invariant tori inside NHIMs \cite{kumar2023aas}; hence, we would like to compute and study them.

Let $\bold{v}_{1}(k) \in \mathbb{R}^{4}$ be the Floquet direction at $X(k)$ corresponding to $\lambda_1$, and similarly for $\bold{v}_{2}(k) \in \mathbb{R}^{4}$ and $\lambda_2$; these will respectively be the first and second columns of each $P(k)$. Then, $\bold{v}_{1}(k)$ and $\bold{v}_{2}(k)$ will form linear approximations to the stable/unstable separatrices at each SPO point $X(k)$. However, the traditional method of computing globalized stable/unstable manifolds---which involves taking points along a small segment of the linear manifold approximation near each $X(k)$ and applying $F$ repeatedly to them---will not work here. Recall that these separatrices are contained in a NHIM; repeatedly applying $F$ to points along the separatrices' linear approximations will thus result in these points quickly diverging \emph{away} from the NHIM, in the direction of strong transverse expansion corresponding to $\lambda_{u}$. Hence, one needs to be able to compute these separatrices as far as possible from the base points $X(k)$ \emph{without} repeatedly applying $F$. 

To address this problem, we compute high-order Taylor polynomials which will approximate the separatrices accurately in a larger domain of validity than their linear approximations. The algorithm used here is a parameterization method \cite{haroetal} largely adapted from those of \citep{kumar2022} for stable/unstable manifolds of invariant tori and \cite{kumar2021journal} for period maps of periodic orbits; it also bears strong similarities to the recent method of \cite{kumar2025jnls} which was also extended from those works. Since we have $q$ SPO points $X(k)$, each of which has one attached ``weak stable'' $\bold{v}_{1}(k)$ and ``weak unstable'' $\bold{v}_{2}(k)$ direction at that point, there will be $q$ 1D curves to compute for each of the stable and unstable separatrices. In the framework of Section \ref{paramsectiongeneral}, we have $\mathcal M = \{0, 1, \dots, q-1\} \times \mathbb{R}$ and $f(k, s) = (k+1 \mod q, \lambda s)$, where $\lambda$ is the weakly stable $\lambda_{1}$ or unstable $\lambda_{2}$ entry of $\Lambda$, depending on which separatrix we are trying to compute. With this, the equation to solve for the parameterizations $W: \{0, 1, \dots, q-1\} \times \mathbb{R} \rightarrow \mathbb{R}^{4}$ of the stable or unstable separatrix is 

\begin{equation}  \label{invariancequationfinal}   F(W(k, s)) - W(k + 1 \mod q, \lambda s) =0, \quad\quad\quad (k, s) \in \{0, 1, \dots, q-1\} \times \mathbb{R} \end{equation}

\subsection{Order-by-Order Method to Solve for $W$} \label{paramsection}

We express $W$ as Taylor series that depends on $k = 0, 1, \dots, q-1$, with form
\begin{equation}  \label{series} W(k, s) = X(k) + \sum_{d \geq 1} W_{d}(k)s^{k}  \end{equation}
Here, $s=0$ corresponds to the points $X(k)$ of the SPO whose separatrix we are trying to compute. The $s^{0}$ term of each $W$ is $X(k)$, and each linear term $W_{1}(k)$ is the stable $\bold{v}_{1}(k)$ or unstable $\bold{v}_{2}(k)$ Floquet direction known from the first or second column of $P(k)$. Hence we need to solve for the higher-order coefficients $W_{d}(k) \in \mathbb{R}^{4}$, $k = 0, 1, \dots, q-1$ for $d \geq 2$. 

Denote $W_{<d} (k, s) = X(k) + \sum_{j = 1}^{d-1} W_{j}(k)s^{j}  $. Assume we have solved for all $W_{j}(k)$ for $j <d$, so that the Taylor expansions of $F(W_{<d}(k, s)) - W_{<d}( k+1 \mod q, \lambda s)$ have only $s^{d}$ and higher order terms for all $k=0, 1, \dots, q-1$. Then, starting with $d=2$, the recursive method to solve for the $W_{d}(k)$ is:
    \begin{enumerate}
        	\item Find $E_{d}(k)= [F(W_{<d}(k, s)) - W_{<d}(k+1 \mod q,  \lambda s)]_{d}$, where $[\cdot]_{d}$ denotes the $s^{d}$ Taylor coefficient of the term inside brackets. We describe methods for this in Section \ref{jettransport}. 
	\item Find $W_{d} (k)$, $k=0, 1, \dots, q-1$ such that $W_{<d} (k,s)+W_{d} (k)s^{d}$ cancels the error $E_d (k)s^{d}$ in Eq. \eqref{invariancequationfinal}, thus satisfying Eq. \eqref{invariancequationfinal} up to order $s^{d}$. The equation to solve for $W_{d}(k)$ is 
	\begin{equation}  \label{correctionwk} DF(X(k)) W_{d}(k) - \lambda^{d} W_{d}(k+1 \mod q) = -E_{d}(k) \end{equation}
	Substituting $W_{d}(k) = P(k) V_d(k)$, defining $\eta_d(k) = -P(k+1 \mod q)^{-1}E_{d}(k) $, and recalling that $X$, $P$, $\Lambda$ solve Eq. \eqref{floquetEquation}, we have Eq. \eqref{correctionwk} is equivalent to $\Lambda V_{d}(k) - \lambda^{d} V_{d}(k+1 \mod q) = \eta(k)$. Equivalently,
	\begin{equation}  \label{v1}  \lambda_1 V_{d,1}(k) - \lambda^{d} V_{d,1}(k+1 \mod q) = \eta_{d,1}(k)  \end{equation} 
		\begin{equation}  \label{v2}  \lambda_2 V_{d,2}(k) - \lambda^{d} V_{d,2}(k+1 \mod q) = \eta_{d,2}(k)  \end{equation} 
	\begin{equation}  \label{v3}  \lambda_s V_{d,3}(k) - \lambda^{d} V_{d,3}(k+1 \mod q) = \eta_{d,3}(k)   \end{equation} 
	\begin{equation}  \label{v4}  \lambda_u V_{d,4}(k) - \lambda^{d} V_{d,4}(k+1 \mod q) = \eta_{d,4}(k)   \end{equation} 
where $V_d(k) = [V_{d,1}(k), V_{d,2}(k), V_{d,3}(k), V_{d,4}(k)]^T , \eta_d(k) = [\eta_{d,1}(k), \eta_{d,2}(k), \eta_{d,3}(k), \eta_{d,4}(k)]^T \in \mathbb{R}^4$. Eqs. \eqref{v1}-\eqref{v4} can be solved by the method of Section \ref{fixedPointIter}, since $ |\lambda_1 \lambda^{-d}| \neq 1$, $| \lambda_2 \lambda^{-d}| \neq 1$, $| \lambda_s \lambda^{-d}| \neq 1$, and $| \lambda_u \lambda^{-d}| \neq 1$ for all $d \geq 2$. This gives $V_d(k)$ and thus $W_{d}(k) = P(k) V_d(k)$. 

	\item Set $W_{<d+1} (k,s) = W_{<d} (k, s) + W_d(k) s^{d}$ and return to step 1.
    \end{enumerate}
    The recursion is stopped when we are satisfied with the degree $d$ of $W$. Note that the Floquet matrix $\Lambda$ allowed us to  decouple the equations in Step 2, simplifying the solution of Eq. \eqref{correctionwk}. We now prove that Eq. \eqref{correctionwk} indeed yields $W_{d}(k)$ cancelling the order $s^d$ error.

\begin{claim*} If $W_{d}$ solves Eq. \eqref{correctionwk}, then for $j \leq d$ (using the $[\cdot]_{d}$ notation defined earlier),
\begin{equation}  \label{jcoeff} \left[  F(W_{<d}(k, s)+W_{d}(k)s^{d}) - \left(W_{<d}(k+1 \mod q, \lambda s)+W_{d}(k+1 \mod q)( \lambda s)^{d}\right) \right]_{j}= 0 \end{equation} 
\end{claim*} 
\begin{proof} Recall that $ F(W_{<d}(k, s)) - W_{<d}(k+1 \mod q,  \lambda s)=E_{d}(k)s^{d} + \mathcal O(s^{d+1}) $ by assumption. Expanding Eq. \eqref{jcoeff} in Taylor series and keeping only $s^{d}$ and lower order terms gives
\begin{align}  \begin{split}
\Big[F(W_{<d}&(k, s)) + \Big. DF(W_{<d}(k, s)) W_{d}(k)s^{d} - \\
& \quad \quad \quad \left. \left(W_{<d}(k+1 \mod q, \lambda s)+W_{d}(k+1 \mod q)( \lambda s)^{d}\right) \right]_{j} \\
=&[E_{d}(k)s^{d} + DF(W_{<d}(k,s)) W_{d}(k)s^{d} - \lambda^{d} W_{d}(k+1 \mod q)s^{d} ]_{j} \\
=&\begin{cases} 
      0 &\text{if $j<d$}, \\
      E_{d}(k) + DF(X(k)) W_{d}(k) - \lambda^{d} W_{d}(k+1 \mod q) = 0 &\text{if $j=d$}
   \end{cases} \\
\end{split} \end{align}
where the $j=d$ case of the last line follows from the preceding line by dividing $s^{d}$ out from the quantity inside $[.]_{j}$, and then taking $s \rightarrow 0$. 
\end{proof}

\subsubsection{Computing $E_{d}(k)$: automatic differentiation and jet transport} \label{jettransport}

In step 1 of the order-by-order parameterization method to find $W$, we computed the $s^{d}$ coefficients
\begin{equation} \label{Ekdef} E_{d}(k)= [F(W_{<d}(k, s)) - W_{<d}(k+1 \mod q,  \lambda s)]_{d}\end{equation}
Each $W_{<d}(k, s)$ is a degree $d-1$ polynomial and $\lambda$ is a constant, so the $s^{d}$ term of $W_{<d}(k+1 \mod q,  \lambda s)$ is just 0. However, computing the Taylor expansion of $F(W_{<d}(k, s)) $ can be more complicated if $F$ is a highly nonlinear map, and even more so if  $F$ is a stroboscopic map defined by integrating points for a fixed time $T_p$ by some Hamiltonian flow  Eqs. \eqref{H_EOM}-\eqref{perturbed_H}. For such cases, we use the tools of automatic differentiation \cite{haroetal} and/or jet transport \cite{perezpalau2015}, which are also sometimes referred to as differential algebra in the literature \citep{dast,Berz1998}. For the sake of self-containedness, the following discussion of these two methods is largely identically reproduced from the author's previous papers \cite{kumar2022, kumar2025jnls}.

Automatic differentiation is an efficient and recursive technique for evaluating operations on Taylor series. For instance, let $f(s)$ and $g(s)$, $s \in \mathbb{R}$, be two series; we can use their known coefficients to compute $d(s) = f(s) / g(s)$ as a Taylor series as well. Let subscript $j$ denote the $s^{j}$ coefficient of a series; since $ f(s) = d(s) g(s)$, we find that $f_{i} = \sum_{j=0}^{i} d_{j} g_{i-j} =  \left( \sum_{j=0}^{i-1} d_{j} g_{i-j}(s) \right)+ d_{i} g_{0} $, which implies that
\begin{equation}  
  \label{autodiff}  d_{i} = \frac{1}{g_{0}} \left( f_{i} - \sum_{j=0}^{i-1} d_{j} g_{i-j} \right)\end{equation}  
Starting with $d_{0} = f_{0}/ g_{0}$, Eq. \eqref{autodiff} allows us to recursively compute $d_{i}$, $i \geq 1$. Similar formulas also exist for recursively evaluating many other functions and operations on Taylor series, including $f(s)^{\alpha}$, $\alpha \in \mathbb{R}$; see \cite{haroetal} for more examples. Most importantly, in all automatic differentiation formulas, the output series $s^{i}$ coefficient depends only on the $s^{i}$ and lower order coefficients of the input series. Hence, truncation of Taylor series for the purpose of computer implementation does not affect the accuracy of the computed coefficients.

In the case that $F$ is a map defined through perhaps complicated but explicit mathematical formulas, automatic differentiation can be applied to evaluate $F(W_{<d}(k, s))$ directly and find the desired degree-$d$ Taylor coefficient. If $F$ is the stroboscopic map of a flow, though, this direct evaluation no longer applies; more steps are required. Nevertheless, automatic differentiation still remains useful for flow maps. Its utility in this case is that we can substitute Taylor series such as $W_{<d}(k, s)$ for $(x, y, p_x, p_y)$ in the equations of motion \eqref{H_EOM}, which gives us series in $s$ for $(\dot x, \dot y,\dot p_x, \dot p_y)$. Thus, after overloading the required operators (e.g. arithmetic and power) to accept Taylor series arguments, we can use numerical integration routines with the series as well.\footnote{In fact, any method of overloading the basic operations to accept polynomial arguments could be used in combination with numerical integration here, including methods other than automatic differentiation. As the implementation of this paper used automatic differentiation, this will be the focus of discussion here.} 
More precisely, consider a Taylor series-valued function of time $V(s,t) = \sum_{j=0}^{\infty}V_{j}(t)s^{j}:\mathbb{R}^{2} \rightarrow \mathbb{R}^{4}$, where $V_{j}(t)$ are its time-varying Taylor coefficients. Write $V_{x}(s,t)$, $V_{y}(s,t)$, ${V}_{p_{x}}(s,t)$, and $V_{p_y}(s,t)$ for the $x$, $y$, $p_x$, and $p_y$ components of $V(s,t)$; similarly write $V_{j,x}(t)$, $V_{j,y}(t)$, ${V}_{j,p_{x}}(t)$, and $V_{j,p_y}(t)$ for the components of $V_{j}(t)$. Substituting $V$ in Eq. \eqref{H_EOM} yields a system of differential equations 
\begin{equation} \label{vxdot}  \frac{d}{dt}{V_{x}(s,t)} = \sum_{j=0}^{\infty}\dot V_{j,x}(t)s^{j} = \frac{\partial H_{\varepsilon}}{\partial p_{x}}\Big(V_{x}(s,t),V_{y}(s,t), V_{p_x}(s,t),V_{p_y}(s,t),\theta_{p}\Big)  \end{equation}
\begin{equation}  \label{vydot} \frac{d}{dt}{V_{y}(s,t)} = \sum_{j=0}^{\infty}\dot V_{j,y}(t)s^{j} = \frac{\partial H_{\varepsilon}}{\partial p_{y}}\Big(V_{x}(s,t),V_{y}(s,t), V_{p_x}(s,t),V_{p_y}(s,t),\theta_{p}\Big)  \end{equation}
\begin{equation}  \label{vpxdot} \frac{d}{dt}{V_{p_x}(s,t)} = \sum_{j=0}^{\infty}\dot V_{j,p_x}(t)s^{j} = -\frac{\partial H_{\varepsilon}}{\partial x}\Big(V_{x}(s,t),V_{y}(s,t), V_{p_x}(s,t),V_{p_y}(s,t),\theta_{p}\Big)  \end{equation}
\begin{equation} \label{vpydot} \frac{d}{dt}{V_{p_y}(s,t)} = \sum_{j=0}^{\infty}\dot V_{j,p_y}(t)s^{j} = -\frac{\partial H_{\varepsilon}}{\partial y}\Big(V_{x}(s,t),V_{y}(s,t), V_{p_x}(s,t),V_{p_y}(s,t),\theta_{p}\Big)  \end{equation}
\begin{equation} \label{thetapdot} \dot \theta_p = \Omega_p \end{equation}

Assume that $H_{\varepsilon}$ and its partials are algebraic functions that are suitable for use with automatic differentiation techniques; see, for instance, the CCR4BP Hamiltonian Eq. \eqref{ccr4bpH}. Hence, if the $V_{j,x}(t)$, $V_{j,y}(t)$, ${V}_{j,p_{x}}(t)$, $V_{j,p_y}(t)$, and $\theta_{p}$ are known for $j \in \mathbb{N}$ and some $t \in \mathbb{R}$, automatic differentiation allows us to simplify the RHS of each of Eq. \eqref{vxdot}-\eqref{vpydot} to a series in $s$. Then, for each of Eq. \eqref{vxdot}-\eqref{vpydot} and $j \in \mathbb{N}$, the $s^{j}$ coefficient $\dot V_{j,x}(t)$, $\dot V_{j,y}(t)$, ${\dot V}_{j,p_{x}}(t)$, or $\dot V_{j,p_y}(t)$ from the LHS must be equal to the $s^{j}$ coefficient of the RHS. In other words, $\dot V_{j,x}(t)$, $\dot V_{j,y}(t)$, ${\dot V}_{j,p_{x}}(t)$, and $\dot V_{j,p_y}(t)$, $j \in \mathbb{N}$, are functions of $\theta_{p}$, $ V_{j,x}(t)$, $ V_{j,y}(t)$, ${ V}_{j,p_{x}}(t)$, and $ V_{j,p_y}(t)$, $j \in \mathbb{N}$. This is effectively a system of differential equations for the time-varying Taylor coefficients of $V(s,t)$. Solving Eq. \eqref{vxdot}-\eqref{thetapdot}   for the various initial conditions $V(s,0) = W_{<d}(k, s)$, $k=0, 1, \dots, q-1$, and initial $\theta_{p}$ equal to the value $\theta_0$ fixed in Section \ref{stroboscopic}, we can compute $F(W_{<d}(k, s))=V(s,T_p)$ for all desired $k$, which are precisely the Taylor series we needed. 

In summary, if $F$ is a map given by explicit formulas, automatic differentiation can be used to directly find the $s^{d}$ Taylor coefficients $E_{d}(k)$ of $F(W_{<d}(k, s))$ for all $k = 0, 1, \dots, q-1$. On the other hand, if $F$ is defined as the stroboscopic map of a flow, we consider the Taylor coefficients of $W_{<d}(k, s)$ as initial state variables to be numerically integrated coefficient by coefficient; propagating by time $T_p$, we get the Taylor coefficients of $F(W_{<d}(k, s))$, and the $s^{d}$ coefficient of this gives us $E_{d}(k)$. This approach for numerical integration of Taylor series is called \emph{jet transport}; see \cite{perezpalau2015} for more details. Truncated Taylor series can be used with jet transport, since the automatic differentiation techniques used to evaluate time derivatives work with truncated series without loss of accuracy; in fact, during the $E_{d}(k)$ calculation, one \emph{should} truncate all series to order $s^d$ for the automatic differentiation and jet transport steps, to optimize computational time and storage. Note that for degree-$d$ truncated series and our $4$-dimensional phase space, there are $4(d + 1)$ coefficients, which is the required dimension for the numerical integration. 

\begin{remark}
If $W(k, s)$, $k = 0, 1, \dots, q-1$ is a solution of Equation \eqref{invariancequationfinal}, then so is $W(k, \alpha s)$ for any $\alpha \in \mathbb{R}$. Sometimes, the jet transport integration may struggle to converge due to fast-growing coefficients $W_{j}(k)$ of $W(k, s)$; conversely, fast-shrinking $W_{j}(k)$ can lead to numerical errors in computing $W(k,s)$. In either case, scaling $W(k, s)$ to some $W(k, \alpha s)$ can help. To do this, simply multiply $W_{1}(k)$ by $\alpha$ and then restart the order-by-order algorithm of Section \ref{paramsection}; $\alpha$ should be chosen so that the $W_{j}(k)$ neither grow too rapidly nor shrink to zero. Such an $\alpha$ can be found by running a preliminary calculation of $W(k, s)$, and fitting an exponential growth rate to the resulting coefficients;  alternatively, simple trial and error also often works. 
\end{remark}

\subsection{Fundamental Domains of Validity for Separatrix Parameterizations $W(k,s)$} \label{funDomainSection}

The $d$ degree Taylor parameterizations $W_{\leq d}(k,s)$ of the stable/unstable separatrices of $X(k)$ under the map $F$ will be much more accurate than their linear approximations by $\bold{ v}_{1}(k)$ or $\bold{ v}_{2}(k)$. Nevertheless, they will still be inexact due to series truncation error; furthermore, even the exact infinite series $W(k,s)$ satisfying Eq. \eqref{invariancequationfinal} would only be valid for $s$ within some radius of convergence. Hence, one must determine the values of $s \in \mathbb{R}$ for which $W_{\leq d}(k,s)$ accurately represents curves on the stable/unstable separatrix. 

Fix an error tolerance, say $E_{tol} = 10^{-5}$ or $10^{-6}$. We now find what \cite{haroetal} calls the fundamental domain of $W_{\leq d}(k, s)$: the largest set $\mathcal D = \{0, 1, \dots, m-1\} \times (-D,D)$ such that for all $(k,s) \in \mathcal{D}$, the error in invariance Eq. \eqref{invariancequationfinal} is less than $E_{tol}$. In other words, we seek the largest $D \in \mathbb{R}^{+}$ such that for all $s$ with $|s| < D$, 
\begin{equation} \max_{k=0, 1, \dots, q-1} \left\|F(W_{\leq d}(k, s)) - W_{\leq d}(k+1 \mod q, \lambda s)\right\| < E_{tol} \end{equation}
The simplest way to find $D$ is to fix $k$ to some value, and then use bisection to find the largest $D_{k}$ such that $\left\|F(W_{\leq d}(k, s)) - W_{\leq d}(k+1 \mod q, \lambda s)\right\| < E_{tol} $ for all $s \in (-D_{k},D_{k})$; starting the bisection with endpoints $s=0$ and $s=10$ generally works well. After doing this for each value of $k=0,1, \dots, q-1$, the aforementioned $D$ will be the minimum of all the $D_{k}$. 

Finally, since $F$ should not be repeatedly applied to separatrix points for globalization here, one should represent the given separatrix curve by simply evaluating each $W_{\leq d}(k,s)$ for a dense grid of values of $s \in [-D, D]$; in fact, since globalization is not being carried out, one can evaluate each $W_{\leq d}(k,s)$ for a dense grid of values of $s \in [-D_{k}, D_{k}]$ instead, which will give longer curves than using $D = \min_{k=0,1,\dots,q-1} D_{k}$. Applying this entire procedure, starting with the steps of Section \ref{paramsection}, for both $\lambda=\lambda_1$ and $\lambda_2$, the resulting parameterizations and evaluated points hence conclude the computation of the desired separatrices. 

\section{Numerical Implementation and Results in the CCR4BP} \label{numResults}

The methodology developed in Sections \ref{spoSection} and \ref{parambigsection} is general, and can be applied to any family of symplectic maps $F_{\varepsilon}: \mathbb{R}^{4} \rightarrow \mathbb{R}^{4}$ satisfying the assumptions of Section \ref{settingSection}. The quasi-Newton method and separatrix parameterization method of this paper have successfully been implemented and used in studies of subharmonic periodic orbits and their induced dynamics in the Jupiter-Europa-Ganymede \cite{kumar2023aas} and Uranus-Titania-Oberon \cite{kumar2024aas} planar CCR4BP models, for which stroboscopic maps were studied (recall Section \ref{stroboscopic}). After a brief discussion of celestial mechanics-specific terminology, and some details on the methods' computational implementations, this section will present examples of SPOs and separatrices from those studies---thus illustrating the effectiveness of this paper's methods for applied numerical and dynamical studies. 

\subsection{Celestial Terminology: Mean Motion Resonances and Secondary Resonances} 

Recall that the planar CCR4BP, described in Section \ref{r4bpSection}, is a 2.5 DOF periodic Hamiltonian perturbation of the 2 DOF PCR3BP (Section \ref{pcr3bpSection}). In this celestial mechanics context, the original NHIMs of unstable periodic orbits from the unperturbed PCR3BP often occur inside phase space regions called \emph{mean motion resonances} (MMRs); any persisting invariant tori of $F_{\varepsilon}$ corresponding to these $F_0$-orbits are thus also generally called ``resonant tori'', even if their rotation number $\omega$ under the CCR4BP stroboscopic map is not in fact a rational multiple of $2\pi$. Hence, to distinguish the more general ``resonant tori'' of CCR4BP MMRs from the ``resonant tori'' defined in Section \ref{settingSection} (those with rational $\frac{\omega}{2\pi}$), we will henceforth refer to the latter as \emph{secondary resonant} tori in this celestial context. As expected, secondary resonant tori from the PCR3BP generically do not persist into the CCR4BP for any $\varepsilon > 0$, but a finite number of their constituent subharmonic periodic orbits do; in \cite{kumar2023aas} and \cite{kumar2024aas}, these SPOs are referred to as \emph{secondary resonant periodic orbits}. We will use this term and SPO interchangeably in Section \ref{jupUrSection}. 

While we refer the reader to \cite{kumar2022} for full details, to give a more intuitive physical interpretation of the above discussion, the MMR orbits of the PCR3BP are in resonance with the revolution of the large mass $m_{2}$ included in that model. Secondary resonant orbits inside this MMR are then \emph{also} in resonance with the revolution of the CCR4BP perturbing mass $m_{3}$. Thus, while secondary resonant orbits remain in resonance with $m_{2}$, they are also in resonance with $m_{3}$, and can be thought of as being contained in a ``resonance inside a resonance''---whose boundaries inside the MMR's 2D NHIM are formed by separatrices.

\subsection{Notes About Computational Implementation} 

The studies that will be discussed in Section \ref{jupUrSection} were carried out on a 2021-era Mac laptop with an Apple M1 Pro CPU. All algorithms were implemented in the Julia programming language. The calculation of CCR4BP SPOs through numerical continuation, starting from the PCR3BP and leveraging the quasi-Newton method of Section \ref{spoSection}, generally took less than 2 seconds per SPO. The algorithm implementation did not require the use of any special packages beyond OrdinaryDiffEq.jl \cite{RackauckasQing} for ODE propagation. Numerical integration for stroboscopic map evaluation used OrdinaryDiffEq.jl's built-in DP8 (order 8/5/3 Dormand-Prince Runge-Kutta) adaptive step size integration algorithm. 

The separatrix computations also were carried out in Julia; we went up to truncation order $d=20$ in our series computations. The parameterization method, automatic differentiation, and jet transport described in Section \ref{parambigsection} were implemented leveraging the TaylorSeries.jl \cite{Benet2019}, TaylorIntegration.jl \cite{perezTaylorIntegration}, and OrdinaryDiffEq.jl \cite{RackauckasQing} packages for automatic differentiation and jet transport. The TaylorSeries.jl package already defines a truncated Taylor1 variable type, with built in automatic differentiation routines to operate on them. The OrdinaryDiffEq.jl library, though not originally developed for jet transport, can handle Taylor1 variables as initial conditions when loaded alongside the TaylorIntegration.jl package, propagating them exactly as described in Section \ref{jettransport} on jet transport. Note that the DP8 integrator was used here as well, rather than the Taylor integrator of TaylorIntegration.jl. 

\subsection{Secondary Resonant Orbits and Separatrices in CCR4BP Models of Real Systems}  \label{jupUrSection}

In the paper \cite{kumar2023aas}, we successfully applied the methods of Sections \ref{spoSection}-\ref{parambigsection} to the computation of secondary resonant periodic orbits, their Floquet multipliers \& directions, and finally their separatrices for the stroboscopic map of the Jupiter-Europa-Ganymede CCR4BP. Europa and Ganymede are both moons of Jupiter which revolve in near-circular orbits around Jupiter; in \cite{kumar2023aas}, we focused on studying the perturbative effect of Jupiter's moon Europa on a family of Jupiter-Ganymede PCR3BP ($\mu = 7.8037 \times 10^{-5}$) unstable periodic orbits contained in Ganymede's 4:3 MMR. These PCR3BP orbits are shown on the left of Figure \ref{fig:43pcr3bp}, with Europa's orbit also displayed for reference (but without its gravity taken into account yet). All computations were carried out in the Jupiter-Ganymede co-rotating reference frame described in Sections \ref{pcr3bpSection} and \ref{r4bpSection}, which is also used for visualization in many of the following figures. Europa's mass ratio $\varepsilon$, as defined in Section \ref{r4bpSection}, is considered the perturbing parameter (also denoted as $\mu_{3}$ in some of the figures). 

\begin{figure}
\includegraphics[width=0.499\columnwidth]{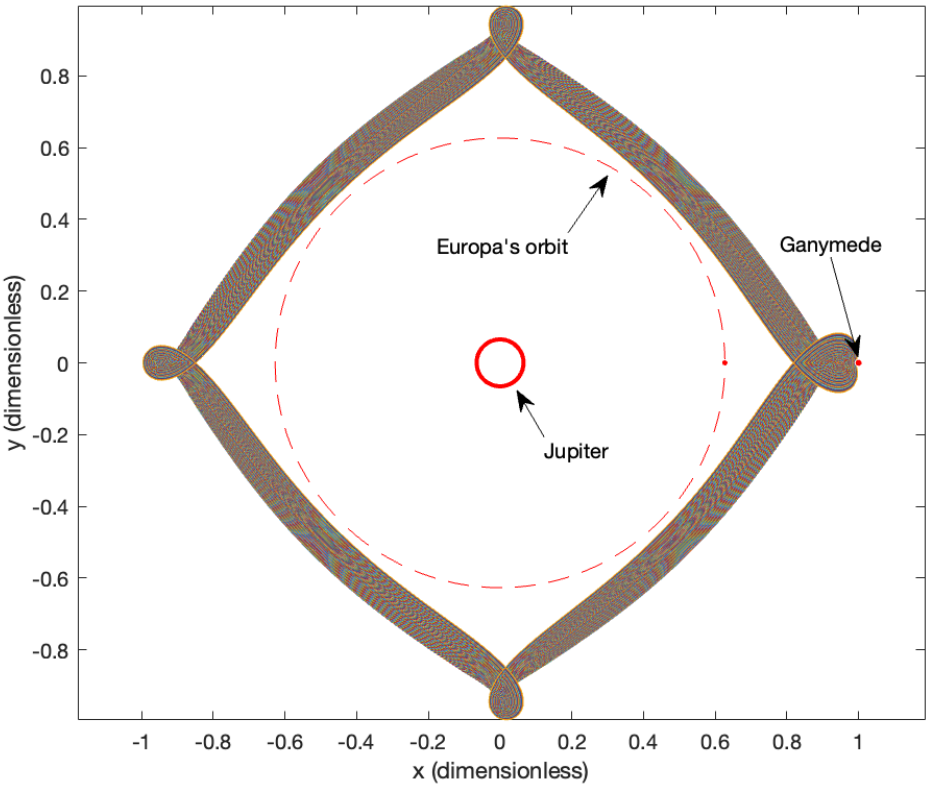}
\includegraphics[width=0.499\columnwidth]{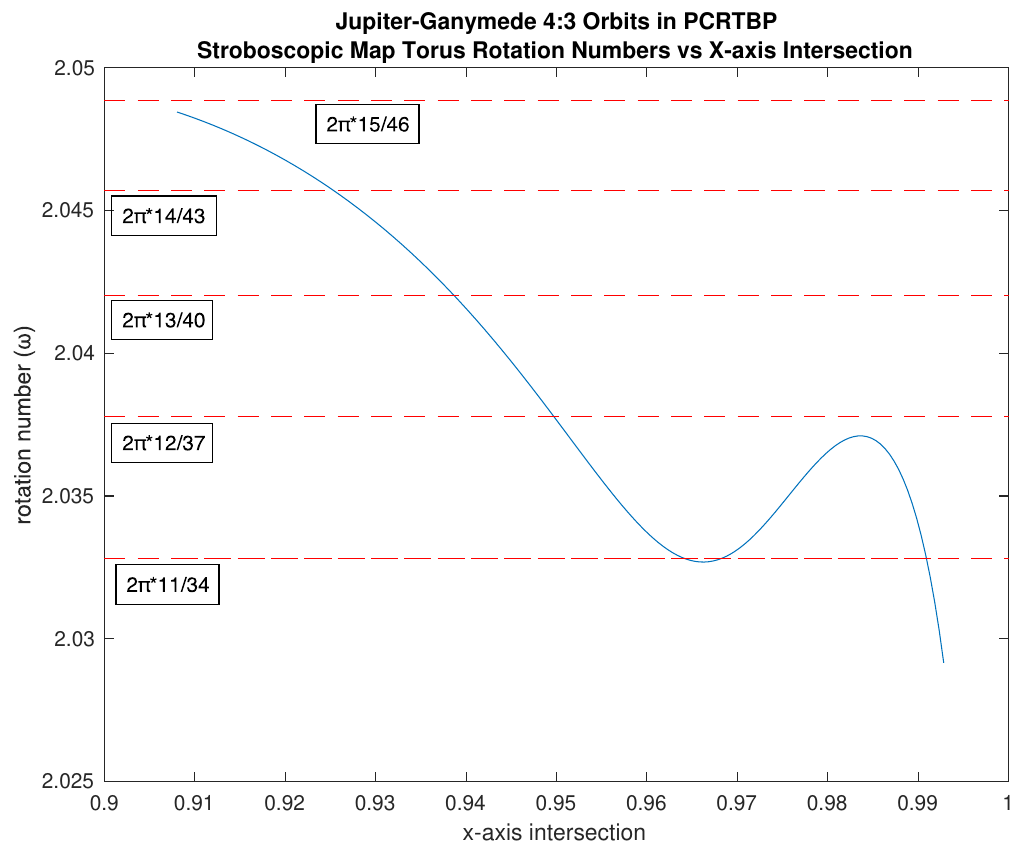}
\vspace{-16pt}
\caption{\label{fig:43pcr3bp} (L) Jupiter-Ganymede PCRTBP 4:3 MMR unstable periodic orbits, (R) Orbit $\omega$ vs $x$-intercept plot \cite{kumar2023aas}} 
\vspace{-16pt}
\end{figure}

In the unperturbed Jupiter-Ganymede PCR3BP with $\varepsilon = 0$, this 4:3 MMR unstable flow-periodic orbit family has a range of periods $T$ which can be converted to stroboscopic map torus rotation numbers using the relation $\omega = 2\pi T_p / T$ of Section \ref{stroboscopic}; in normalized units, $T_p \approx 6.1966$ for the Europa perturbation period (and hence stroboscopic mapping time). On the right of Fig. \ref{fig:43pcr3bp}, one can see the plot of $\omega$ versus the PCR3BP orbit's rightmost $x$-intercept. 
While we were able to numerically continue many tori with $\omega > 2.04047$ from the Jupiter-Ganymede PCR3BP into the physical-mass (planar) Jupiter-Europa-Ganymede CCR4BP---i.e., from $\varepsilon = 0$ to $\varepsilon = 2.5265 \times 10^{-5}$ for Europa---this was not the case for lower $\omega$. Thus, to investigate the dynamics in this portion of the 4:3 Ganymede MMR unstable orbit family, we sought to instead compute secondary resonant periodic orbits (SPOs) and their separatrices in the physical-mass CCR4BP.

Using the CCR4BP's symmetry to identify the initial phases of unperturbed PCR3BP SPOs likely to persist into the CCR4BP, in \cite{kumar2023aas} we used the quasi-Newton method of Section \ref{spoSection} to numerically continue secondary resonant SPOs and their Floquet directions for $\frac{\omega}{2\pi}$ ratios of 11/34, 34/105, 23/71, 35/108, 12/37, and 25/77. Since writing \cite{kumar2023aas}, we also computed SPOs at ratios 37/114 and 45/139, which are also included in some of the following figures. All of these ratios correspond to $\omega < 2.04047$. We used a tolerance of $10^{-7}$ in Eqs. \eqref{invariance}-\eqref{bundleEquations} and continuation step sizes $\Delta \varepsilon$ of $5 \times 10^{-7}$ to $10^{-6}$ for the computation of these SPOs, all of which were successfully continued under the CCR4BP stroboscopic map until the desired $\varepsilon$ value---in contrast with the tori in this range of $\omega < 2.04047$, none of which survived to $\varepsilon = 2.5265 \times 10^{-5}$. Some of these SPOs experienced stability changes in the $\lambda_1$-$\lambda_2$ Floquet multiplier pair as $\varepsilon$ changed; for instance, the 23/71 SPO which had hyperbolic $\lambda_1$-$\lambda_2$ at $\varepsilon = 3 \times 10^{-6}$ changed to elliptic  $\lambda_1$-$\lambda_2$ at $\varepsilon = 4 \times 10^{-6}$, and vice versa. The quasi-Newton method handled such  $\lambda_1$-$\lambda_2$ transitions without loss of accuracy. 

Figure \ref{fig:srPOs} shows all orbits (both tori and SPOs) found for $\omega < 2.04047$, with $\varepsilon = 0$ shown on top, $\varepsilon =  8.0 \times 10^{-6}$ ($\approx 31.7$ percent of Europa's actual mass ratio)  in the middle, and the real Europa mass ratio $\varepsilon = 2.5265 \times 10^{-5}$ on bottom. 
\begin{figure}[t]
\begin{centering}
\includegraphics[width=0.93\columnwidth]{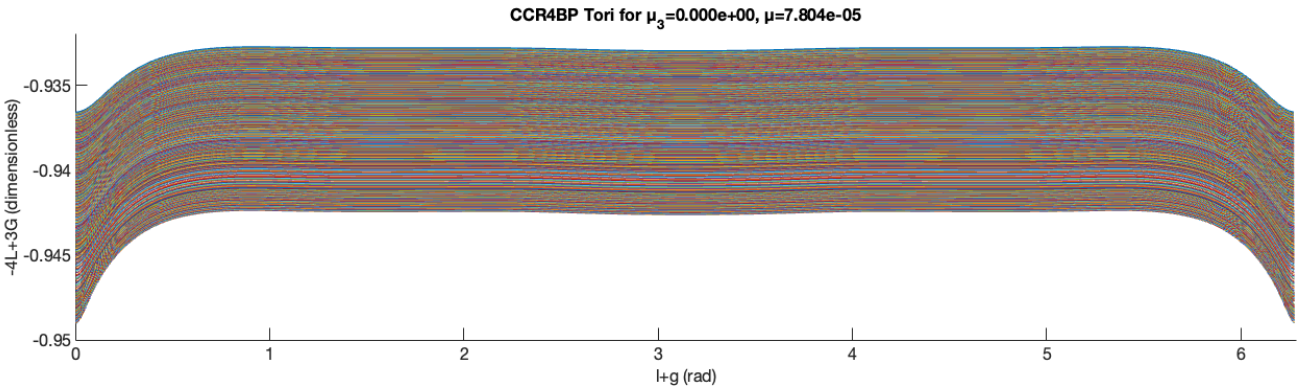}
\includegraphics[width=0.93\columnwidth]{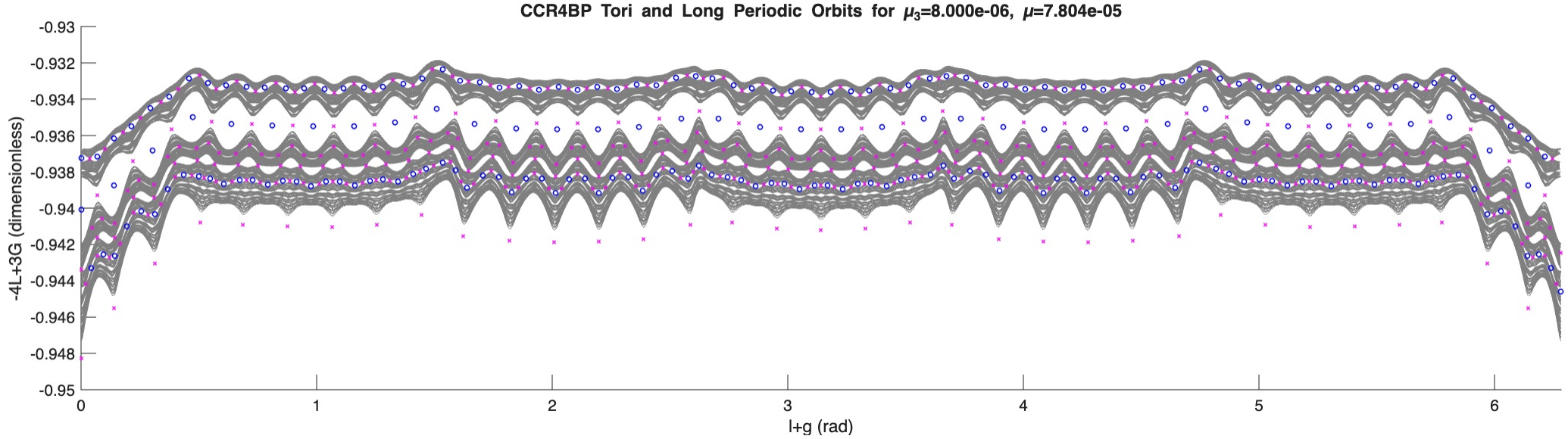}
\includegraphics[width=0.93\columnwidth]{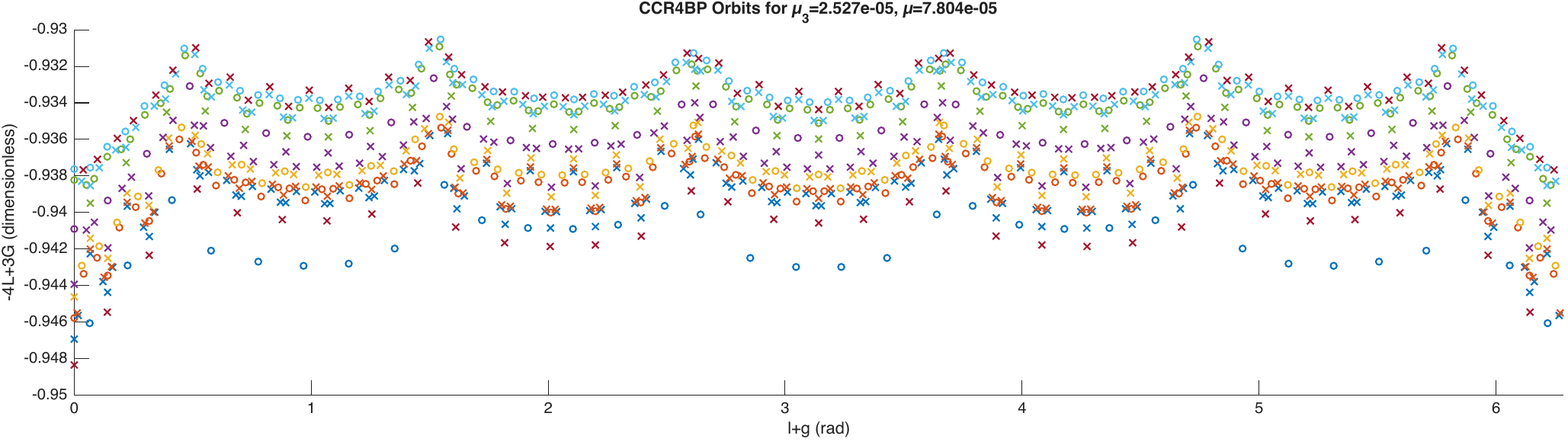}
\vspace{-6pt}
\caption{\label{fig:srPOs}  Ganymede 4:3 MMR unstable orbits of CCR4BP map, $\varepsilon = 0.0, 8.0 \times 10^{-6}, 2.5265 \times 10^{-5}$ (from top to bottom) \cite{kumar2023aas}}
\vspace{-6pt}
\end{centering}
\end{figure}
These plots are all in ``action-angle-like'' coordinates (see \cite{kumar2023aas} for details) of the unstable 4:3 MMR orbit NHIM, in order to highlight the appearance of secondary resonances. Indeed, in the intermediate-perturbation middle plot where tori and isolated SPOs coexist, one can see the secondary resonant periodic orbits at the ``necks'' and ``centers'' of pendulum-shaped regions bounded by persisting tori, as expected from perturbation theory. In the middle and bottom plots, SPOs with $\lambda_1,\lambda_2$ hyperbolic (elliptic) are marked by $\times$'s ($\circ$'s), respectively. While such orbits are simply represented by these isolated points for the CCR4BP stroboscopic map, they correspond to continuous 1D orbits of the CCR4BP flow as well; two such flow-periodic orbits, for $\frac{\omega}{2\pi}=11/34$ and $12/37$, are plotted in Figs. \ref{fig:1134orbit} and \ref{fig:1237orbit} respectively. 
\begin{figure}[h!]
\begin{centering}
\includegraphics[width=0.47\columnwidth]{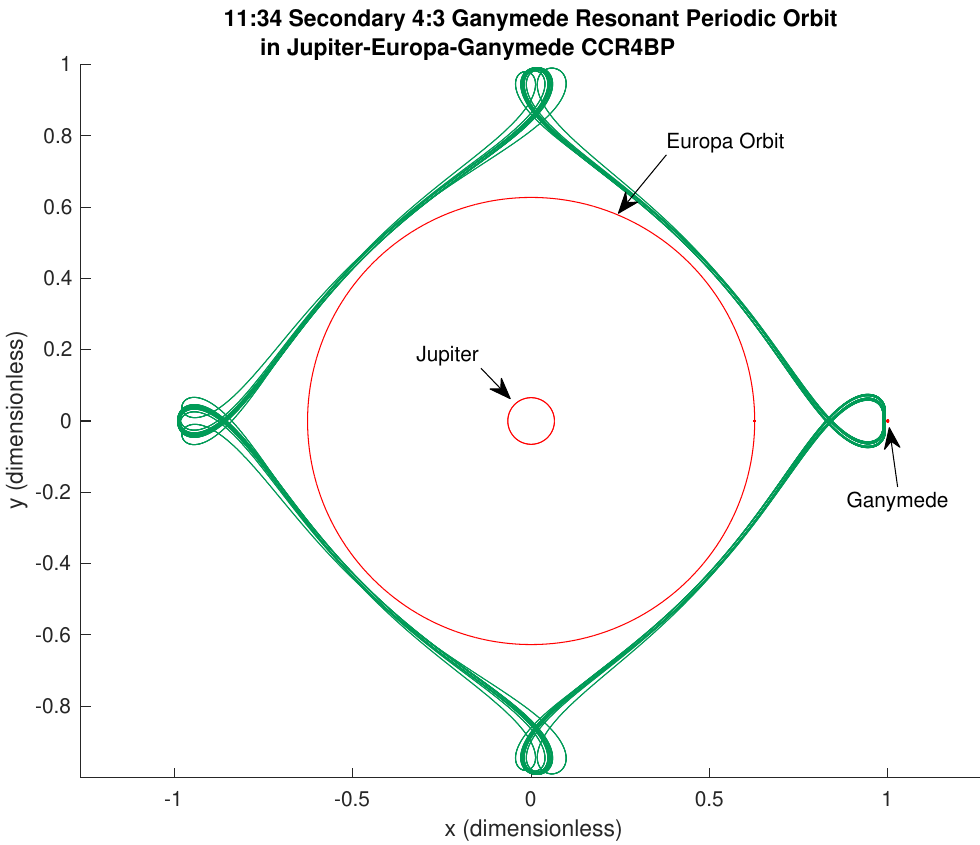}
\includegraphics[width=0.48\columnwidth]{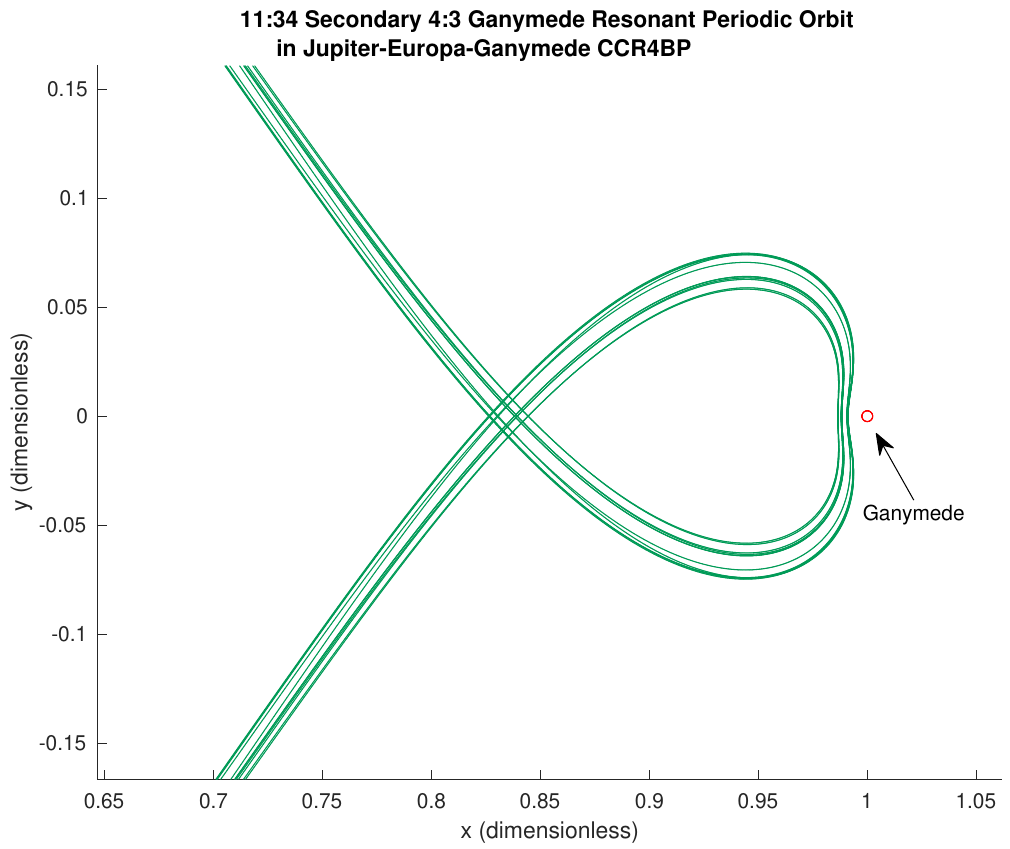}
\vspace{-10pt}
\caption{\label{fig:1134orbit} 4:3 Ganymede MMR 11/34 flow-SPO, including zoomed view on right, Jupiter-Europa-Ganymede CCR4BP}
\end{centering}
\vspace{-3pt}
\end{figure}
\begin{figure}
\includegraphics[width=0.503\columnwidth]{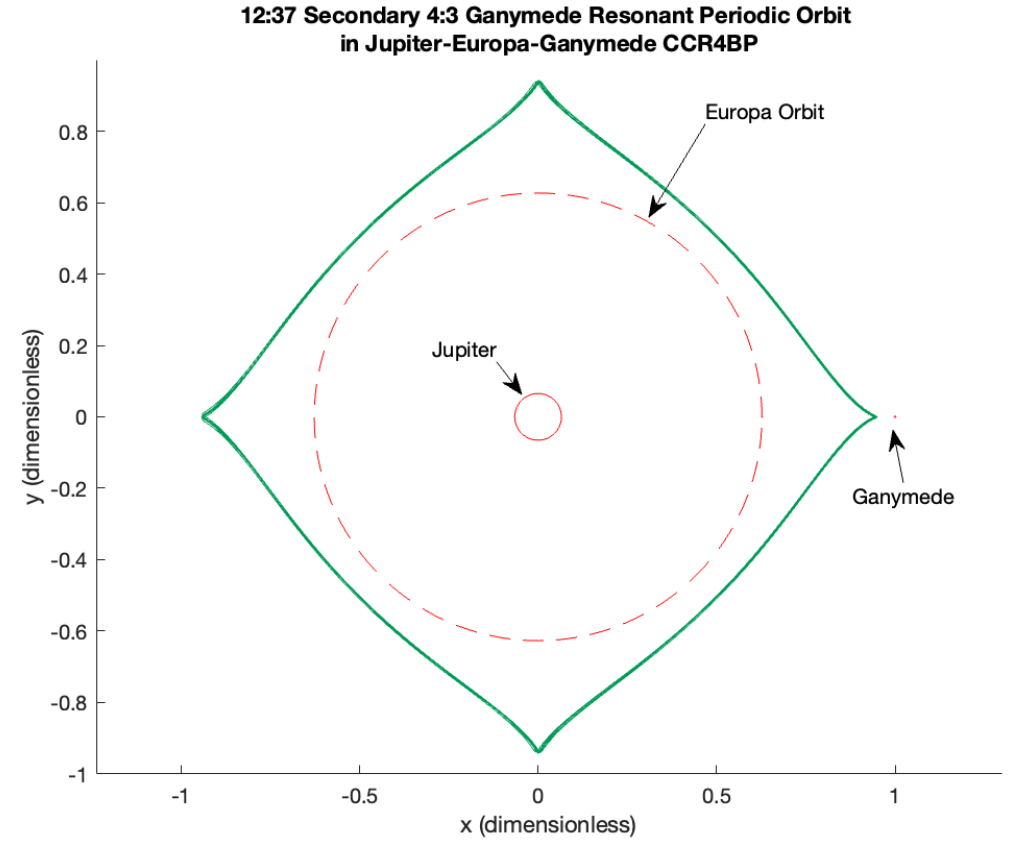}
\includegraphics[width=0.497\columnwidth]{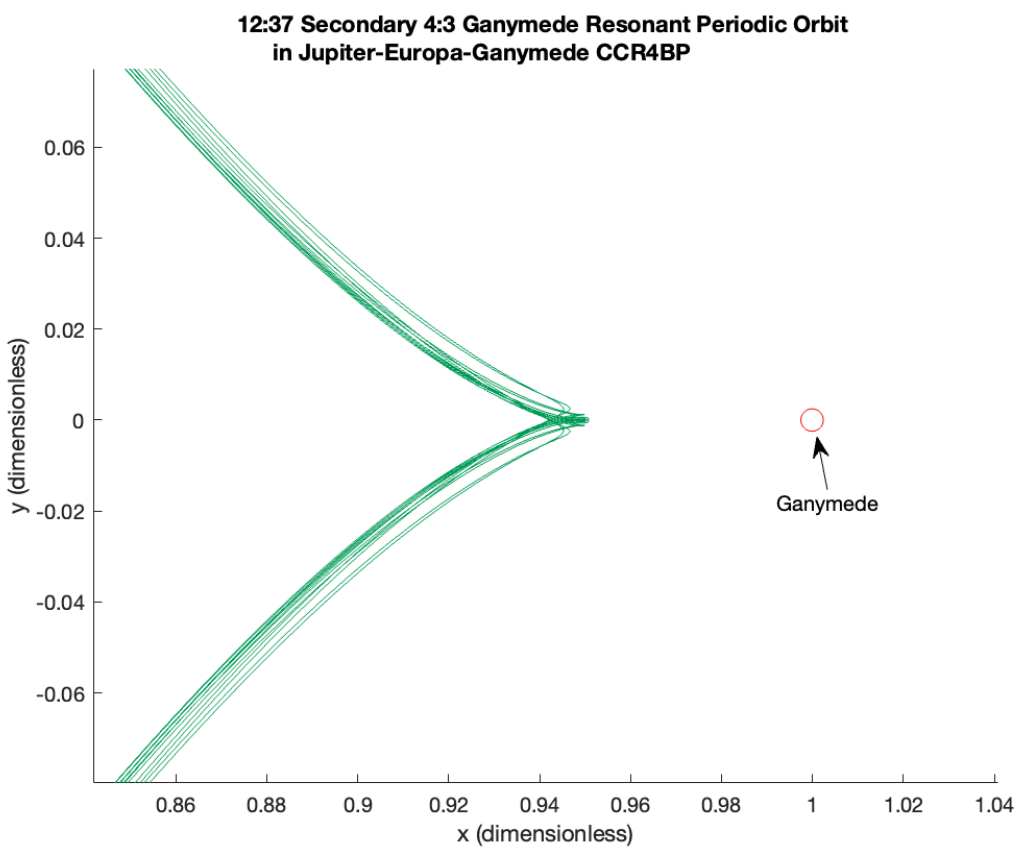}
\vspace{-19pt}
\caption{\label{fig:1237orbit} 4:3 Ganymede MMR 12/37 flow-SPO, including zoomed view on right, Jupiter-Europa-Ganymede CCR4BP \cite{kumar2023aas}}
\vspace{-3pt}
\end{figure}

Finally, with the secondary resonant SPO periodic orbits computed in the CCR4BP for all the aforementioned ratios $\frac{\omega}{2\pi} = p/q$, as well as the SPOs' Floquet directions and multipliers, we took advantage of these Floquet directions to compute separatrices emanating from those SPOs with $\lambda_1$, $\lambda_{2}$ hyperbolic. Following the parameterization method of Section \ref{parambigsection} and plotting the resulting separatrices' points in the same ``action-angle-like'' coordinates also used in Fig. \ref{fig:srPOs}, the resulting curves are displayed in Fig. \ref{fig:separatrices}; the SPO points are shown as well. The ability of the parameterization method to accurately capture the nonlinear shape of these separatrices farther away from their base SPOs helped confirm that separatrices of consecutive SPOs intersect in this part of the 4:3 MMR NHIM---thereby destroying all tori for this range of $\omega$, and providing the dynamical cause due to which the earlier torus continuations had failed. We refer the interested reader to the full paper \cite{kumar2023aas} for more details and analysis of this result and its practical significance. 

\begin{figure}[h!]
\begin{centering}
\includegraphics[width=0.99\columnwidth]{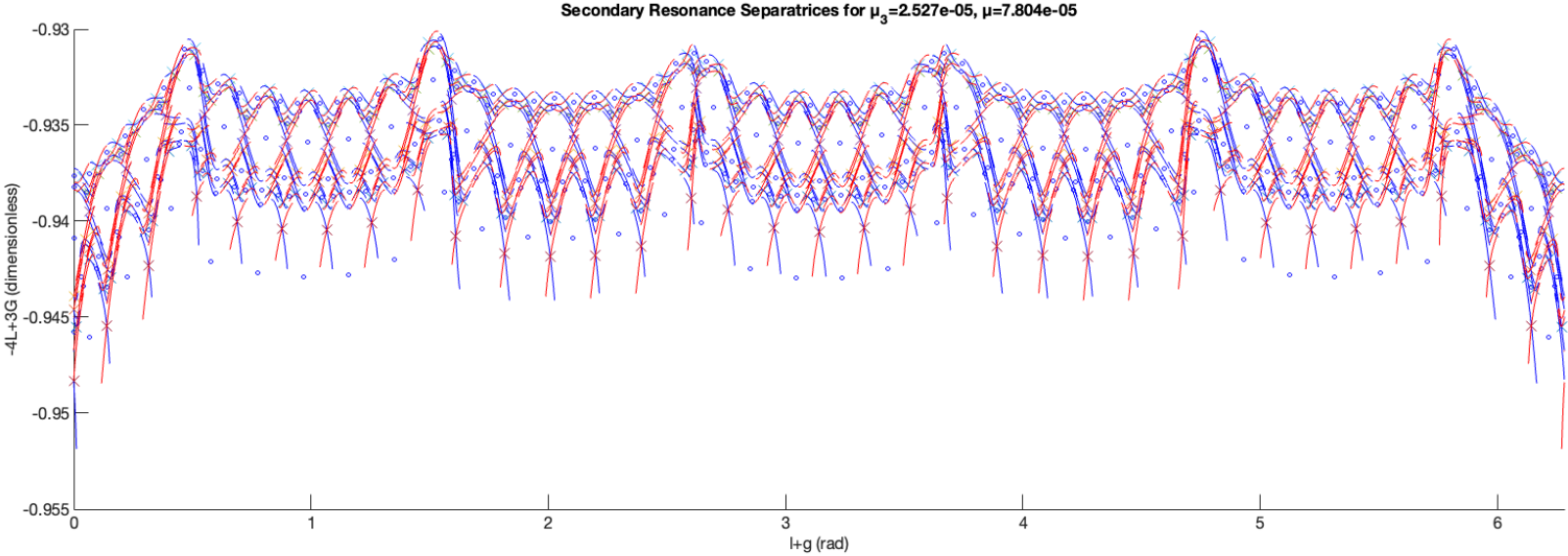}
\vspace{-12pt}
\caption{\label{fig:separatrices} Ganymede 4:3 MMR unstable secondary resonant periodic orbits and separatrices for stroboscopic map of physical-mass Jupiter-Europa-Ganymede CCR4BP. Plot in ``action-angle-like'' coordinates. Stable separatrices in blue, unstable in red.}
\end{centering}
\vspace{-12pt}
\end{figure}

While \cite{kumar2023aas} studied orbits of the 4:3 Ganymede MMR in the Jupiter-Europa-Ganymede system, in \cite{kumar2024aas} we shifted focus to the Uranian system, finding phenomena similar to the Jovian case in Oberon's 6:5 MMR as well. While we refer the reader to the full paper \cite{kumar2024aas} for details, in summary, once again we started with a family of (here, 6:5 Oberon) MMR unstable periodic orbits from the (Uranus-Oberon) PCR3BP, and studied the effect of an additional perturbing moon (here, Titania) on those orbits. Modeling the latter situation by a CCR4BP with $\mu = 3.5433 \times 10^{-5}$, we used the quasi-Newton methods of \cite{kumar2022} and Section \ref{spoSection} to respectively continue tori and SPOs (with their Floquet directions), across a range of stroboscopic map $\omega$ values, from $\varepsilon = 0$ to Titania's real mass ratio $\varepsilon = 3.9168 \times 10^{-5}$. Here, the secondary resonance ratios $\frac{\omega}{2\pi}$ for which real-mass CCR4BP SPOs were computed were 25/69, 21/58, 17/47, 30/83, 13/36, 22/61, and 9/25. 
Fig.\ \ref{fig:65oberon} shows the 6:5 Oberon MMR PCR3BP unstable periodic orbit family on left, and the successfully-continued 9/25 secondary resonant periodic orbit of the Uranus-Titania-Oberon CCR4BP flow on right. 

\begin{figure}[t]
\includegraphics[width=0.509\columnwidth]{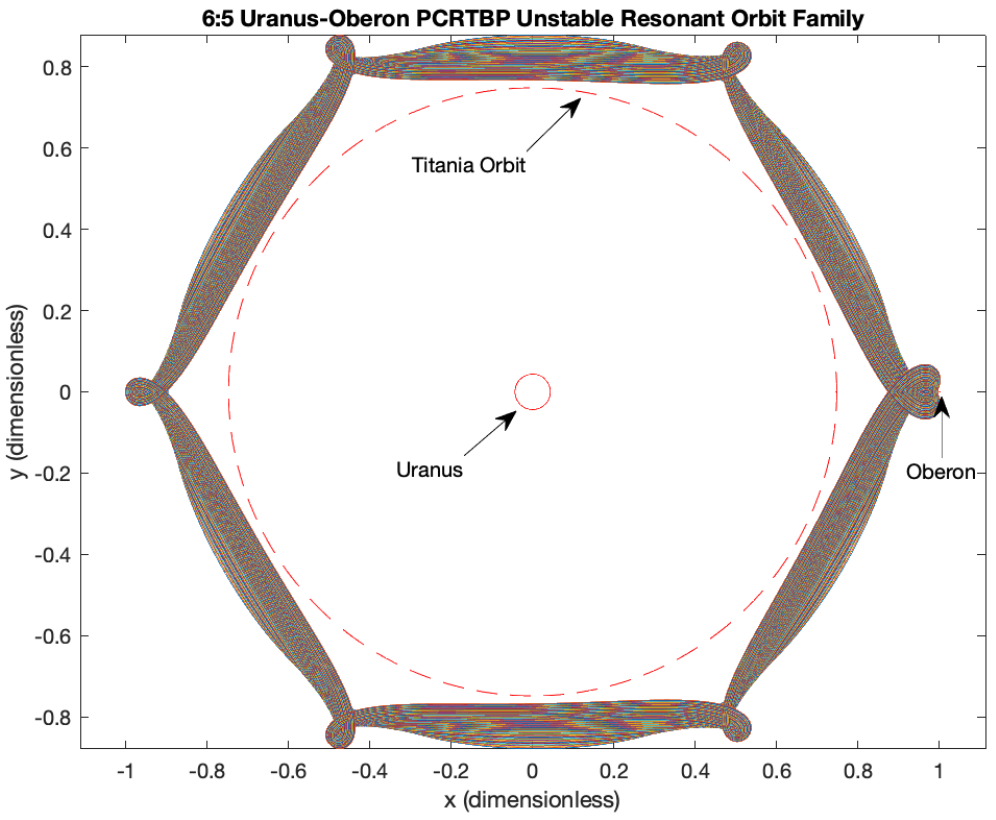}
\includegraphics[width=0.491\columnwidth]{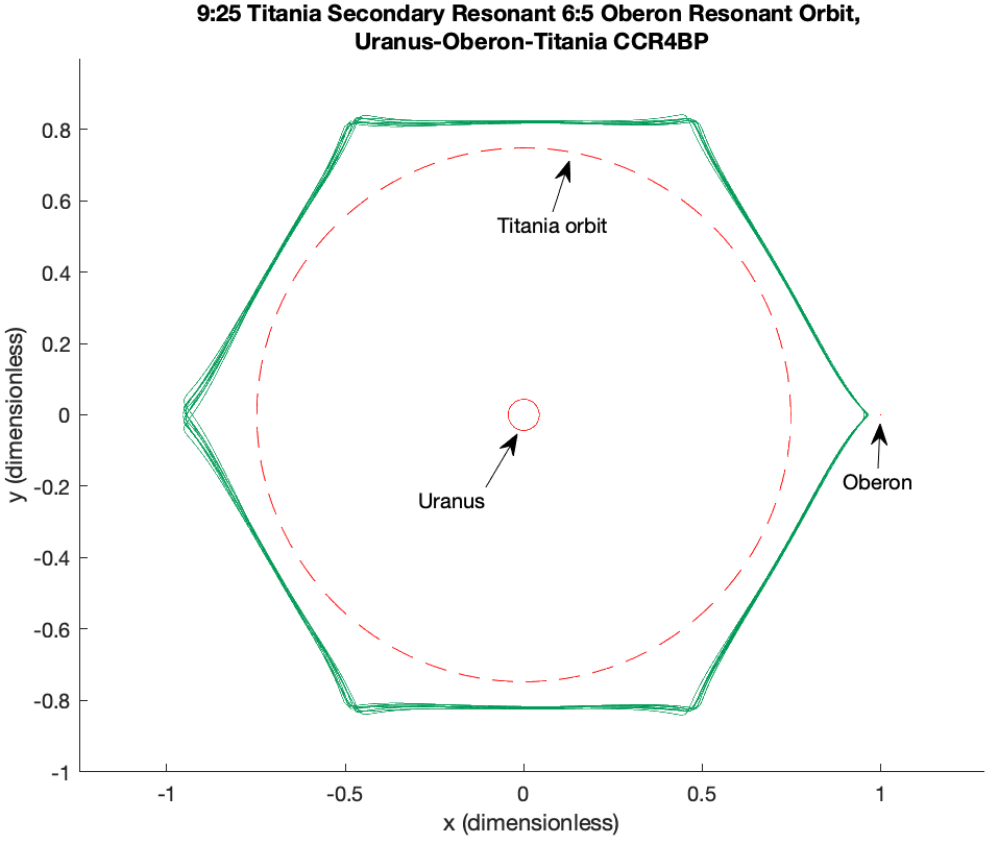}
\caption{\label{fig:65oberon} (L) Uranus-Oberon PCRTBP 6:5 MMR unstable periodic orbits, (R) 6:5 Oberon MMR 9/25 flow-SPO, Uranus-Titania-Oberon CCR4BP \cite{kumar2024aas}} 
\end{figure}

Finally, once again, separatrices of secondary resonant CCR4BP SPOs with $\lambda_1, \lambda_2$ hyperbolic were computed by using the Floquet directions given by the quasi-Newton method to initialize and carry out the parameterization method of Section \ref{parambigsection}. The resulting separatrix curves were then plotted in action-angle-like coordinates for the 6:5 MMR NHIM similar to those of Fig. \ref{fig:separatrices} for the 4:3 MMR. Fig. \ref{fig:separatrices2} displays the resulting plotted SPOs and separatrices, along with some of the persisting invariant tori that occur only in phase space regions further away from Titania's orbit. Here as well, the ability to compute the separatrices farther away from their base SPOs was key to detecting intersections of separatrices of consecutive secondary resonant periodic orbits, which again provides the mechanism of torus destruction in this region. The interested reader may refer to \cite{kumar2024aas} for more details of this Uranian study. 

\begin{figure}
\begin{centering}
\includegraphics[width=0.999\columnwidth]{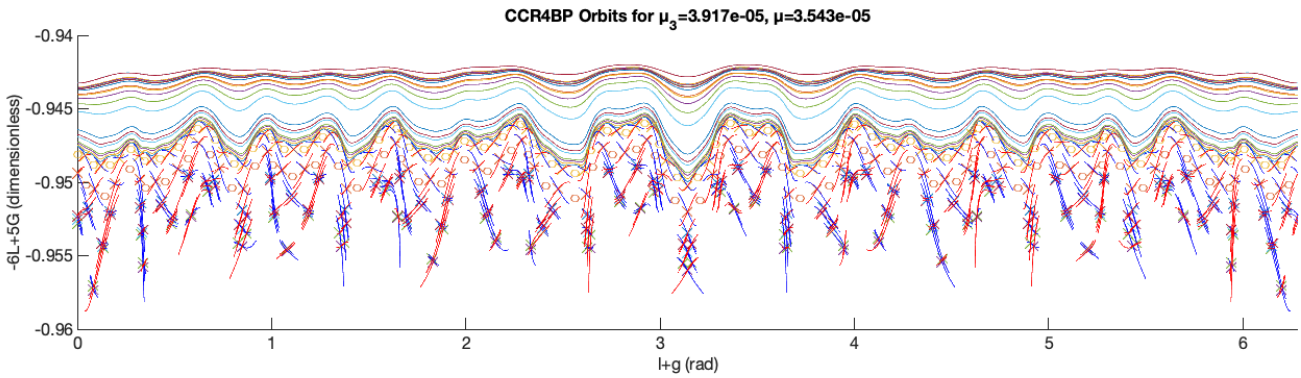}
\vspace{-12pt}
\caption{\label{fig:separatrices2} Oberon 6:5 MMR unstable tori, secondary resonant periodic orbits, and separatrices for stroboscopic map of physical Uranus-Titania-Oberon CCR4BP. Plot in ``action-angle-like'' coordinates. Stable separatrices in blue, unstable in red \cite{kumar2024aas}.}
\end{centering}
\vspace{-12pt}
\end{figure}

\section{Conclusion}

In this paper, we developed a fast quasi-Newton method for the simultaneous computation of subharmonic periodic orbits and their Floquet vectors \& multipliers for some common perturbative families of 4D symplectic maps. Our method improves the efficiency of the SPO calculation by eliminating the need for solving and finding eigenvectors of large linear systems of equations, as is required by the traditional multiple shooting methods used in the vast majority of existing literature. Furthermore, this paper's method provides Floquet directions and multipliers that traditional methods do not; this information can in turn be used to initialize an order-by-order method for the computation of Taylor series parameterizations of these SPO's weak stable and unstable separatrices. Our quasi-Newton method, for the first time ever, extends the invariant torus parameterization methods of e.g., \cite{haroLlave, haroetal, kumar2022}, to the direct computation of points of subharmonic periodic orbits. 
The developed quasi-Newton method can be used for continuation of SPOs and Floquet vectors by perturbation parameter, for whose initialization we presented methods of constructing an unperturbed-map solution. 

Notably, the methods of this paper apply to stroboscopic maps of 2.5 DOF Hamiltonian systems that arise from periodic perturbations of 2 DOF autonomous Hamiltonians. One is thus able to use this work's quasi-Newton method to numerically continue periodic orbits of the 2 DOF system whose periods under the flow are in resonance with that of the perturbation. Such a situation applies in real-life applications when considering the perturbative effect of a third large body on periodic orbits of the commonly-studied 2 DOF planar circular restricted 3-body problem. The methods of this paper for computing the resulting SPOs, Floquet directions, and separatrices have been successfully applied to such investigations in the Jovian and Uranian systems, providing a practical demonstration of these algorithms' real-world utility. 

While the methods developed in this work assume a 4D map phase space, an interesting future extension would be the case of subharmonic periodic orbits in higher dimensions. For example, one could consider a non-degenerate flow-periodic orbit of an $n$-DOF Hamiltonian flow on $\mathbb{R}^{2n}$ for which the monodromy matrix eigenvalue 1 has algebraic multiplicity 2, and study the effect of an $(n+0.5)$-DOF periodic perturbation on this orbit. If the orbit period under the $n$-DOF flow is resonant with that of the perturbation, this (subharmonic) periodic orbit may also persist into the $(n+0.5)$-DOF system for certain initial phases. An important example of one such orbit is the 9:2 near rectilinear halo orbit \cite{nrho} that NASA's Lunar Gateway space station plans to use, which is derived from a periodic Halo orbit of the 3 DOF (spatial) Earth-Moon CR3BP whose period is 2/9 that of the synodic solar perturbation. The essential steps of the quasi-Newton method, with additional equations to handle the additional Floquet multiplier pairs and directions, seem generalizable to such cases. Another useful extension would be to make this paper's method be able to handle fold bifurcations during numerical continuation. Indeed, we hope that this paper's generalization of the parameterization method from tori to periodic orbits will open new lines of future research as well.

\section*{Acknowledgements}

This research was carried out in part at the Jet Propulsion Laboratory, California Institute of Technology, under a contract with the National Aeronautics and Space Administration (80NM0018D0004). This work was supported in part by the National Science Foundation (NSF) under a Mathematical Sciences Postdoctoral Research Fellowship award no. DMS-2202994. 

\appendix

\section{Proof that $A,B$ are contraction maps in $\ell^\infty$ norm} \label{contractionProof}

We prove that $A$ is a contraction if $| \lambda_a(k)/\lambda_b(k)| <1$ for all $k = 0, 1, \dots q-1$; the same can be shown for $B$  very similarly if $| \lambda_a(k)/\lambda_b(k)| > 1$. Denote $C = \max_{k = 0, 1, \dots, q-1} | \lambda_a(k)/\lambda_b(k)|$ in the below proof. 
\begin{lemma}
If $0<C <1$, $A$ is a contraction map. Hence, in this case the iteration $u_{n+1} = A(u_{n})$ uniformly converges exponentially fast as $n \rightarrow \infty$ to the length-$q$ solution sequence $u$ of Eq. \eqref{xi3contract} (and thus also \eqref{cohomHyp}). 
\end{lemma}
\begin{proof}
Let $u_{1},u_{2}$ be finite sequences indexed by $k=0, 1, \dots, q-1$. Then, 
\begin{align} \begin{split}  \max_{k}  \| [A(&u_1)](k) - [A(u_2)](k) \| \\
&=  \max_{k} \left|\left| \frac{\lambda_{a}(k-1 \mod q) u_1(k-1 \mod q) -  \lambda_{a}(k-1 \mod q) u_2(k-1 \mod q) }{\lambda_{b}(k-1 \mod q)} \right|\right|  \\
&\leq  C \max_{k} \| u_1(k-1 \mod q) - u_2(k-1 \mod q)  \| = C \max_{k} \| u_1(k) - u_2(k) \| 
\end{split} \end{align} 
As $0<C<1$, $A$ is a contraction map under the $\ell^\infty$ norm. The contraction mapping theorem \citep{chicone2006} tells us that every such map has a unique fixed point; furthermore, the fixed point can be found by iterating any value in the domain of the map forwards until convergence. The solution of Eq. \eqref{xi3contract}  is by definition the fixed point of contraction map $A$. Hence, the iteration converges to $u$. 
\end{proof}

\section{Proof of constant $\tilde \lambda_s, \tilde \lambda_u$ solution with $\bold{v}_{s}, \bold{v}_{u}$ rescaled by $a_s, a_u$} \label{rescaleProof}

Using the same notation as in Section \ref{constantLambda}, note that since $X_{\varepsilon}$, $P_{\varepsilon}$, $\Lambda_{\varepsilon}$ was a solution of Eq. \eqref{bundleEquations}, then 
\begin{equation} \label{vsEquation} DF_{\varepsilon}(X_{\varepsilon}(k))  \bold{ v}_{s}(k) =  \lambda_{s}(k) \bold{ v}_{s}(k + 1 \mod q)\end{equation}
\begin{equation} \label{vuEquation} DF_{\varepsilon}(X_{\varepsilon}(k))  \bold{ v}_{u}(k) =  \lambda_{s}(k) \bold{ v}_{u}(k + 1 \mod q)\end{equation}
which can be derived from columns 3 and 4 of Eq. \eqref{bundleEquations}. Since the procedure of Section \ref{constantLambda} leaves columns 1 and 2 of $\tilde P_{\varepsilon}$ and $\tilde \Lambda_{\varepsilon}$ unchanged from $P_{\varepsilon}$ and $\Lambda_{\varepsilon}$ (and thus still satisfying Eq. \eqref{bundleEquations}), we only need to verify that equations similar to Eqs. \eqref{vsEquation}--\eqref{vuEquation} hold for columns 3 and 4 of $\tilde P_{\varepsilon}$ and $\tilde \Lambda_{\varepsilon}$ as well. This is proven for column 3 ($\bold{\tilde v}_{s}$ and $\bar \lambda_s$) below; the case of $\bold{\tilde v}_{u}$ and $\lambda_u$ can be proven in the exact same manner. 
\begin{lemma}
If $\bold{ v}_{s}(k)$, $\lambda_s(k)$ satisfy Eq. \eqref{vsEquation} for all $k=0, 1, \dots, q-1$, and $a_{s}(k)$,  $\bar \lambda_{s}$ satisfy Eq. \eqref{scalevs}, then $\bold{\tilde v}_{s}(k) = a_{s}(k) \bold{v}_{s}(k)$ satisfies
\begin{equation} DF_{\varepsilon}(X_{\varepsilon}(k))  \bold{\tilde v}_{s}(k) =  \bar \lambda_{s} \bold{\tilde v}_{s}(k + 1 \mod q)\end{equation}
\end{lemma}
\begin{proof}
Since $DF_{\varepsilon}(X_{\varepsilon}(k)) \bold{v}_{s}(k) = \lambda_{s}(k) \bold{v}_{s}(k + 1 \mod q)$, we have
\begin{align} \begin{split} DF_{\varepsilon}(X_{\varepsilon}(k))  \bold{\tilde v}_{s}(k) &= DF_{\varepsilon}(X_{\varepsilon}(k))   a_{s}(k) \bold{ v}_{s}(k)   =a_{s}(k) \lambda_{s}(k) \bold{ v}_{s}(k + 1 \mod q) \\
&= a_{s}(k + 1 \mod q)  \bar \lambda_{s} \bold{v}_{s}(k + 1 \mod q)) =  \bar \lambda_{s}  \bold{\tilde v}_{s}(k + 1 \mod q) \end{split} \end{align}
where $a_{s}(k) \lambda_{s}(k) = a_{s}(k + 1 \mod q) \bar \lambda_{s} $ follows from exponentiating Eq. \eqref{scalevs}. 
\end{proof}

\section{Proof of properties of $\bold{v}_c(k)$ and $\bold{v}_2(k)$} \label{sympConjProof}

Here, we show that the vectors $\bold{v}_{c}(k)$ and $\bold{v}_{2}(k)$ found using the procedures of Section \ref{initCols12} satisfy Eqs. \eqref{almostCol2} and \eqref{finalCol2}, respectively. This is proven as a result of the two lemmas below, both adapted from similar results for invariant tori \cite{kumar2022}. 

\begin{lemma} \label{sympconjLemma}
The vectors $\bold{v}_{c}(k)$ defined in Eq. \eqref{sympconj} satisfy Eq. \eqref{almostCol2} for all $k = 0, 1, \dots, q-1$, i.e.:
\begin{equation}  \label{almostCol2_2} DF_0(X_0(k)) \bold{v}_{c}(k) = T(k) DK_0(\theta_{k+1 \mod q}) + \bold{v}_{c}(k+1 \mod q)  \end{equation}
\end{lemma}

\begin{lemma} \label{constTLemma}
If $\bold{v}_{c}(k)$ and $a(k)$ satisfy Eqs. \eqref{almostCol2} and \eqref{Tkill} respectively for all $k = 0, 1, \dots, q-1$, with $ T = \frac{1}{q}\sum_{k=0}^{q-1} A(k) $ in Eq.  \eqref{Tkill}, then $\bold{ v}_{2}(k) = \bold{v}_{c}(k) + a(k) DK_0(\theta_{k})$ will satisfy Eq. \eqref{finalCol2}:
\begin{equation}  DF_{0}(X_{0}(k))  \bold{ v}_{2}(k) =  T DK_0(\theta_{k+1 \mod q}) + \bold{ v}_{2}(k + 1 \mod q)   \end{equation}
\end{lemma}

\begin{proof}[Proof of Lemma \ref{sympconjLemma}]
Applying Eq. \eqref{sympconj} and then Eq. \eqref{abcd}, and recalling that the $\bold{v}_{s}$, $\bold{v}_{u}$, $\lambda_s$, $\lambda_{u}$ of Section \ref{initCols34} satisfy $DF_0(X_0(k)) \bold{v}_{s}(k) = \lambda_{s} \bold{v}_{s}(k+1 \mod q)$ and $DF_0(X_0(k)) \bold{v}_{u}(k) = \lambda_{u}  \bold{v}_{u}(k+1 \mod q)$, we have
\begin{align} \begin{split} DF_0(X_0(k)) &\bold{v}_{c}(k) = DF_0(X_0(k)) \left( \frac{J^{-1} DK_0(\theta_k)}{ \|DK_0(\theta_k)\|^{2}} + f_{1}(k) \bold{v}_{s}(k)  + f_{2}(k) \bold{v}_{u}(k) \right) \\
= &A(k) DK_0(\theta_{k+1 \mod q}) + B(k) \frac{J^{-1} DK_0(\theta_{k+1 \mod q})}{ \|DK_0(\theta_{k+1 \mod q})\|^{2}} \\ 
&+ \left(C(k) + \lambda_{s}(k)f_{1}(k) \right) \bold{v}_{s}(k+1 \mod q) + \left( D(k) + \lambda_{u}(k)f_{2}(k) \right)\bold{v}_{u}(k+1 \mod q) 
\end{split} \end{align} 
Recalling Equations \eqref{f1} and \eqref{f2}, we thus have that
\begin{align} \begin{split} \label{DFtimesvc} DF_0(X_0(k)) \bold{v}_{c}(k) = &A(k) DK_0(\theta_{k+1 \mod q}) + B(k) \frac{J^{-1} DK_0(\theta_{k+1 \mod q})}{ \|DK_0(\theta_{k+1 \mod q})\|^{2}} \\ 
&+ f_{1}(k+1 \mod q)  \bold{v}_{s}(k+1 \mod q) + f_{2}(k+1 \mod q) \bold{v}_{u}(k+1 \mod q) 
\end{split} \end{align} 
As $F_0$ is symplectic, it satisfies $\Omega(\bold{v}_{1}, \bold{v}_{2}) = \Omega(DF_0(X_0(k)) \bold{v}_{1}, DF_0(X_0(k)) \bold{v}_{2})$ for all $\bold{v}_{1}$, $\bold{v}_{2} \in \mathbb{R}^{4}$, where $\Omega$ is the bilinear symplectic form defined on Euclidean $\mathbb{R}^{4}$ as $ \Omega(\bold{v}_{1}, \bold{v}_{2}) =  \bold{v}_{1}^{T} J \bold{v}_{2} $. It is easy to see that $\Omega(\bold{v}_{1}, \bold{v}_{1}) = 0$ for any $\bold{v}_{1} \in \mathbb{R}^{4}$. Furthermore, as $0<  \lambda_{s} < 1$, and recalling Eq.  \eqref{col1Lambda}, we have that
\begin{align} \begin{split} \label{deriveZero} \max_{k = 0, \dots, q-1} |\Omega(DK_{0}(\theta_k),\bold{v}_{s}(k)) |&= \max_{k = 0, \dots, q-1} \left| \Omega \left(DF_0(X_0(k)) DK_0(\theta_k), DF_0(X_0(k) ) \bold{v}_{s}(k) \right) \right| \\
&=\max_{k = 0, \dots, q-1}  \left| \Omega \left(DK_0(\theta_{k+1 \mod q}), \lambda_s \bold{v}_{s}(k+1 \mod q) \right) \right| \\
&=\max_{k = 0, \dots, q-1}  \lambda_s  \left| \Omega \left(DK_0(\theta_{k+1 \mod q}), \bold{v}_{s}(k+1 \mod q) \right) \right| \\
& = \lambda_s \max_{k = 0, \dots, q-1} \left| \Omega \left(DK_0(\theta_k), \bold{v}_{s}(k) \right) \right| \\
\end{split} \end{align} 
which implies that $\max_{k = 0, \dots, q-1} \left| \Omega \left(DK_0(\theta_k), \bold{v}_{s}(k) \right) \right| = 0$. Thus, for all $k$, $\Omega \left(DK_0(\theta_k), \bold{v}_{s}(k) \right) =0$. We can also show that $\Omega \left(DK_0(\theta_k), \bold{v}_{u}(k) \right) = 0$ in a very similar manner. Hence, using Eq. \eqref{sympconj} for $\bold{v}_{c}$, we find 
\begin{align} \label{areaOne} \begin{split} \Omega(DK_0(\theta_k),\bold{v}_{c}(k)) &= \Omega \left(DK_0(\theta_k),\frac{J^{-1} DK_0(\theta_k)}{ \|DK_0(\theta_k)\|^{2}} + f_{1}(k) \bold{v}_{s}(k)  + f_{2}(k) \bold{v}_{u}(k) \right) \\
&= \Omega \left(DK_0(\theta_k),\frac{J^{-1} DK_0(\theta_k)}{ \|DK_0(\theta_k)\|^{2}} \right) \\
&= DK_0(\theta_k)^{T}J \frac{J^{-1} DK_0(\theta_k)}{ \|DK_0(\theta_k)\|^{2}}  = \frac{ DK_0(\theta_k)^{T} DK_0(\theta_k)}{ \|DK_0(\theta_k)\|^{2}} = 1
\end{split} \end{align} 
Since $F_0$ is a symplectic map, using Eqs. \eqref{col1Lambda} and \eqref{DFtimesvc} we have that 
\begin{align} \begin{split} \label{BisOne} 1 = \Omega &(DK_0(\theta_k),\bold{v}_{c}(k)) \\
= \Omega &\left(DF_0(X_0(k)) DK_0(\theta_k),DF_0(X_0(k)) \bold{v}_{c}(k) \right) \\
= \Omega &\left(DK_0(\theta_{k+1 \mod q}),A(k) DK_0(\theta_{k+1 \mod q}) + B(k) \frac{J^{-1} DK_0(\theta_{k+1 \mod q})}{ \|DK_0(\theta_{k+1 \mod q})\|^{2}} \right. \\
& \quad \quad \quad \quad \quad \quad \left. + f_{1}(k+1 \mod q)  \bold{v}_{s}(k+1 \mod q) + f_{2}(k+1 \mod q) \bold{v}_{u}(k+1 \mod q) \vphantom{\frac{J^{-1} DK_0(\theta_{k+1 \mod q})}{ \|DK_0(\theta_{k+1 \mod q})\|^{2}}} \right) \\
= \Omega &\left(DK_0(\theta_{k+1 \mod q}), B(k) \frac{J^{-1} DK_0(\theta_{k+1 \mod q})}{ \|DK_0(\theta_{k+1 \mod q})\|^{2}} \right) \\
=B(&k) DK_0(\theta_{k+1 \mod q})^{T} J  \frac{J^{-1} DK_0(\theta_{k+1 \mod q})}{ \|DK_0(\theta_{k+1 \mod q})\|^{2}} = B(k)
\end{split} \end{align} 
proving that $B(k) =1$. Therefore, substituting this into Eq. \eqref{DFtimesvc} gives
\begin{align} \begin{split} \label{subInto} DF_0(X_0(k)) \bold{v}_{c}(k) = &A(k) DK_0(\theta_{k+1 \mod q}) +  \frac{J^{-1} DK_0(\theta_{k+1 \mod q})}{ \|DK_0(\theta_{k+1 \mod q})\|^{2}} \\ 
&+ f_{1}(k+1 \mod q)  \bold{v}_{s}(k+1 \mod q) + f_{2}(k+1 \mod q) \bold{v}_{u}(k+1 \mod q) 
\end{split} \end{align} 
Finally, we see from Eq. \eqref{sympconj} that the last 3 terms on the RHS of Eq. \eqref{subInto} are just $\bold{v}_{c}(k+1 \mod q)$. Hence,
\begin{equation}  DF_0(X_0(k)) \bold{v}_{c}(k) = A(k) DK_0(\theta_{k+1 \mod q}) + \bold{v}_{c}(k+1 \mod q) \end{equation} 
which is what we sought to prove. 
\end{proof}

\begin{proof} [Proof of Lemma \ref{constTLemma}]
Since $\bold{v}_{c}(k)$ satisfies Eq. \eqref{almostCol2} and $DF_{0}(X_{0}(k)) DK_0(\theta_{k}) = DK_0(\theta_{k+1 \mod q})$, we have
\begin{align} \begin{split} DF_{0}(X_{0}(k)) \bold{ v}_{2}(k) &= DF_{0}(X_{0}(k)) \left[ \bold{v}_{c}(k) + a(k) DK_0(\theta_{k}) \right] \\
&= \left[A(k)+a(k) \right] DK_0(\theta_{k+1 \mod q})+\bold{v}_{c}(k + 1 \mod q) \\
&= \left[ T+a(k + 1 \mod q) \right] DK_0(\theta_{k+1 \mod q}) +\bold{v}_{c}(k + 1 \mod q) \\
&=  T DK_0(\theta_{k+1 \mod q}) +\bold{ v}_{2}(k + 1 \mod q) \end{split} \end{align}
where the relation $A(k)+a(k) =  T+a(k + 1 \mod q)$ follows from Eq. \eqref{Tkill}. 
\end{proof}

\bibliographystyle{elsarticle-num}
\bibliography{references}   

\end{document}